\documentclass[a4paper,12pt]{amsart}
\usepackage{amssymb,amsmath}
\let\amslrcorner\lrcorner
\usepackage{mathabx}
\usepackage{stmaryrd}
\usepackage{amscd,amsthm,amssymb}
\usepackage{enumerate}
\usepackage{color}
\usepackage[all,cmtip]{xy}
\usepackage{hyperref}
\usepackage{mathrsfs} 

\let\lrcorner\amslrcorner

\usepackage{latexsym}

\usepackage{hyperref}
\hypersetup{
    colorlinks=true, 
    linktoc=all,     
   linkcolor=blue,  
}

\def\sideremark#1{\ifvmode\leavevmode\fi\vadjust{\vbox to0pt{\vss
\hbox to 0pt{\hskip\hsize\hskip1em%
\vbox{\hsize2cm\tiny\raggedright\pretolerance10000%
\noindent {\color{red}{#1}}\hfill}\hss}\vbox to8pt{\vfil}\vss}}}%

\def \bbR{\mathbb{R}}
\def\.{\cdot}

\def\la{\langle}
\def\ra{\rangle}

\def\t{\tilde}
\def\beq{\begin{equation}}
\def\eeq{\end{equation}}
\def\bea{\begin{eqnarray*}}
\def\eea{\end{eqnarray*}}
\def\beaa{\begin{eqnarray}}
\def\eeaa{\end{eqnarray}}
\def\ba{\begin{array}}
\def\ea{\end{array}}

\def \L{\mathscr{L}}

\def \bdel{\overline{\partial}}

\def\Ric{\mathrm{Ric}}
\def\id{\mathrm{id}}
\def\be{\begin{equation}}
\def\ee{\end{equation}}
\def\tr{\mathrm{tr}}

\def\Sym{\mathrm{Sym}}
\def\SL2{\mathfrak{sl}_2(\bbC)} 

\def\D{\mathbf{b}} 

\def\End{\mathrm{End}}

\def\vol{\mathrm{vol}}

\def\Sym{\mathrm{Sym}}

\def\LL{\mathrm{L}}
\def\grad{\mathrm{grad}}

\def \t5{\frac{1}{\sqrt{5}}}

\def \bfG{\mathbf{G}}
\def\FN{\mathrm{FN}} 

\DeclareMathOperator{\TT}{TT} 

\DeclareMathOperator{\ric}{\mathrm{ric}}

\def \bfv{\mathbf{v}}
\def \bfw{\mathbf{w}}

\DeclareMathOperator{\di}{d} 

\DeclareMathOperator{\bbC}{\mathbb{C}}

\newtheorem{pro}{Proposition}[section]
\newtheorem{conj}{Conjecture}[section]
\newtheorem{teo}[pro]{Theorem}
\newtheorem{lema}[pro]{Lemma}
\newtheorem{coro}[pro]{Corollary}
\theoremstyle{definition}
\newtheorem{defi}[pro]{Definition}
\newtheorem{rema}[pro]{Remark}

\title{Third order Einstein deformations for K\"ahler-Einstein metrics}
\author{Paul-Andi Nagy}
\address[Paul-Andi Nagy]{Center for Complex Geometry \\
Institute for Basic Science(IBS)\\
55 Expo-ro, Yuseong-gu \\
34126 Daejeon, South Korea}
\date{\today}

\begin{document}

\begin{abstract}
For compact K\"ahler manifolds $(M,g,J)$ with negative scalar curvature we study the existence problem 
for non-trivial Einstein deformations of $g$, that is small time curves $g_t$ of Einstein metrics with $g_0=g$. 
No asssumption on the complex structure $J$ is made; also we do not assume that the metrics $g_t$ are K\"ahler w.r.t. $J$. We determine explicitly 
the obstruction to third order Einstein deformation for $g$; that is we fully solve the equations 
$(\Ric^{g_t})^{(k)}(0)=0$ for $1 \leq k \le 3$ in terms of the Taylor expansion $g^{-1}g_t=\id+th_1+\tfrac{t^2}{2!}h_2+\tfrac{t^3}{3!}h_3+o(t^4)$ at $t=0$. Up to a suitable gauge transformation we show that third order integrability for the Einstein equation amounts to Maurer-Cartan type equations and polynomial identities relating the coefficients $h_3,h_2,h_1$. This result is interpreted 
in terms of the underlying complex geometry of $M$ by means of the Cayley transform of the metric $g$; the Cayley transform is also used for formulating conjectures for the higher order Einstein deformation problem.

\noindent
2020 {\it Mathematics Subject Classification}: Primary 32Q20; Secondary 
53C25, 53C55, 53C21.

\noindent{\it Keywords}: negative K\"ahler Einstein metric, Einstein deformation, complex deformation, obstruction to deformation; Kodaira-Spencer and Nijenhuis brackets, Cayley transform, gauge transformation.

\end{abstract}
\maketitle
\tableofcontents
\section{Introduction} \label{intro}
Let $(M^{2m},g,J)$ be compact and K\"ahler-Einstein with Einstein constant $E<0$. In this paper we study the Einstein 
deformation problem for the metric $g$, that is the existence of small time families of Riemannian metrics $g_t$ with 
$g_0=g$ which are Einstein with Einstein constant $E$, that is 
$$\Ric^{g_t}=E\id.$$
We {\it{do not}} assume a priori that the metrics $g_t$ are K\"ahler, nor related in any way to the complex structure $J$; also no assumption is made a priori on the complex structure $J$, in the sense we do not assume that $J$ can be deformed to a small time family $J_t$ of complex structures on $M$. Indicating with  
$\mathscr{M}_1$ the space of Riemannian metrics on $M$ having the same volume as $g$ we recall that we may reduce the problem to the case when $g_t \in \mathscr{M}_1$. For various normalisation procedures we will frequently 
use that the set of Einstein metrics $\mathscr{M}_1^{\mathrm{E}} \subseteq \mathscr{M}_1$ is naturally acted on by the gauge group $\bfG$ of volume preserving diffeomorphisms on $M$.

To the best of our knowledge the Einstein deformation problem at order $3$, that is solutions to the equation 
$$(\Ric^{g_t})^{(3)}(0)=0$$ 
in terms of the Taylor series expansion given by $g^{-1}g_t=\id+th_1+\tfrac{t^2}{2}h_2+\tfrac{t^3}{3!}h_3+o(t^4)$ at $t=0$ is widely open and no explicit information is available on the differential operators governing the theory. This paper fully solves this problem for negative K\"ahler-Einstein metrics as follows. To explain the statement of our main results we consider the vector fields 
\begin{equation} \label{norm-3i}
\begin{split}
\mathbf{X}_1(g_t)=&\delta^gh_1\\
\mathbf{X}_2(g_t)=&\delta^g(h_2-h_1^2)+\tfrac{1}{4}\di\!\tr(h_1^2)\\
\mathbf{X}_3(g_t)=&\delta^g \left (h_3-\tfrac{3}{2}\{h_1,h_2-h_1^2\}-\tfrac{3}{2}h_1^3+\tfrac{3}{8}\tr(h_1^2)h_1 \right )
+\tfrac{3}{8}\di\!\tr \{h_1,h_2-h_1^2\}
\end{split}
\end{equation} 
whenever the family of Riemannian metrics $g_t \in \mathscr{M}_1$ has Taylor series expansion given by $g^{-1}g_t=\id+th_1+\tfrac{t^2}{2!}h_2+\tfrac{t^3}{3!}h_3+o(t^4)$ at $t=0$. Here $\delta^g$ indicates the divergence operator $\delta^g :\Gamma \left ( \Sym^2TM\right ) \to \Omega^1M$.

Furthermore we consider the real version of the complex $\bdel$-operator, i.e., 
$$\bdel :\Omega^1(M,TM) \to \Gamma \left (\lambda^2_{-}(M,TM) \right )$$ 
where the real vector bundle $\lambda^2_{-}(M,TM):=\{\alpha \in \Omega^2(M,TM) : \alpha (\cdot, J \cdot)=-J\alpha(\cdot ,\cdot)\}$. 
We also consider the space 
$\mathrm{Harm}^{0,2}(M,T^{0,1}M):=\ker \bdel \cap \ker \delta^g \cap \Gamma \left (\lambda^2_{-}(M,TM) \right )$ which actually captures the harmonic representatives in the Dolbeault cohomology group $H^{0,2}(M,T^{0,1}M)$; for, 
$\mathrm{Harm}^{0,2}(M,T^{0,1}M) \otimes_{\bbR} \bbC$ is canonically isomorphic with $H^{0,2}(M,T^{1,0}M)$. Also recall that the bundle of symmetric $2$-tensors splits as 
\begin{equation} \label{typeS-i}
\Sym^{2}TM=\Sym^{2,-}TM \oplus \Sym^{2,+}TM
\end{equation}
into $J$-anti-invariant, respectively $J$-invariant components. Whenever $h \in \Gamma \left (\Sym^2TM \right )$ we indicate with $h^{\pm}$ the corresponding components with respect to this splitting. 

We also need to recall the initial data the deformation theory for Einstein metrics builds on, namely 
the space of so-called (essential) infinitesimal deformations
$$\varepsilon(g):=\{h \in \Gamma \left (\Sym^2TM \right ) \in \ker \Delta_E : 
\delta^gh=0 \ \mathrm{and} \ \tr(h)=0\}$$
where $\Delta_E$ is the Einstein operator acting on symmetric tensors. 
Indeed if $g_t \in \mathscr{M}_1$ is a family of Einstein metrics with Taylor expansion $g^{-1}g_t=\id+th_1+\tfrac{t^2}{2}h_2+\tfrac{t^3}{3!}h_3+o(t^4)$ then up to a gauge 
transformation in $\bfG$ we may assume that $h_1 \in \varepsilon(g)$(see e.g. \cite{Besse}).  Because $(M,g,J)$ is K\"ahler-Einstein with $E<0$ the space $\varepsilon(g)$ reduces to \cite{Koiso-I}
$$\varepsilon(g)=\varepsilon^{-}(g):=\{ h \in \Gamma \left (\Sym^{2,-}TM \right ) : \bdel h=0 \ \mathrm{and} \ 
\delta^gh=0\}.$$
This fact plays an important role in this paper and essentially motivates the assumption $E<0$ made in what follows. In particular the space $\varepsilon(g)$ coincides, up to gauge normalisation, with the infinitesimal 
deformation space of the complex structure $J$.
\begin{rema} \label{C-Y}
\begin{itemize}
\item[(i)] Assume that $(M,g,J)$ is compact and K\"ahler-Einstein with $E<0$. If one knows that the complex structure $J$ can be deformed to a small time family of complex structures $J_t$ on $M$ then, by work of Koiso 
\cite{Koiso-I}, the metric 
$g$ can also be deformed to a family of Einstein metrics $g_t$, in fact such that $(g_t,J_t)$ is K\"ahler.\\

\item[(ii)] For Calabi-Yau manifolds, that is when $E=0$, the deformation theory of the complex structure $J$ is unobstructed in the sense that for any $h_1 \in \varepsilon^{-}(g)$ there exists a small time family $J_t$ of complex structures on $M$ with $J_0=J$ which are tangent to $h_1$; see \cite{T1,T2}.
\item[(iii)] In relation to (i) we record that all explicit examples of negative K\"ahler-Einstein metrics in real dimension $2m \geq 6$ have 
the property that the deformation theory of the complex structure $J$ is fully unobstructed \cite{Dai}. However 
in real dimension $4$, Hans-Joachim Hein has remarked, see also \cite{S-Sw}[page 28], that there are examples of negative K\"ahler-Einstein metrics 
due to Horikawa \cite{Hori} where the complex structure $J$ is rigid; in fact the deformation theory for $J$ is obstructed to second order in these examples.
\end{itemize}
\end{rema}
With the notation thus set our main result is contained in the following 
\begin{teo} \label{main-i1}
Let $(M,g,J)$ be K\"ahler-Einstein with $E<0$ and let $g_t \in \mathscr{M}_1$ be a small time Einstein deformation of $g$ with Taylor series expansion $g^{-1}g_t=\id+th_1+\tfrac{t^2}{2!}h_2+\tfrac{t^3}{3!}h_3+o(t^4)$ at $t=0$. The following hold 
\begin{itemize}
\item[(i)] there exists a gauge transformation  $f_t \in \mathbf{G}$ with $f_0=\id$ such that the Einstein deformation $\tilde{g}_t:=f_t^{\star}g_t \in \mathscr{M}_1$ of $g$ is normalised to third order, in the sense that $\di\!\mathbf{X}_1(\tilde{g}_t)=\di\!\mathbf{X}_2(\tilde{g}_t)=\di\!\mathbf{X}_3(\tilde{g}_t)=0$
\item[(ii)] assuming that $g_t$ is normalised to third order we have that the Einstein equation up to third order for $g_t$, namely 
$(\Ric^{g_t})^{(k)}(0)=0$ for $1 \leq k \le 3$, is equivalent to the following 
\begin{itemize}
\item[a)] the $J$-invariant components of the tensors $h_2$ and $h_3$ are determined from 
\begin{equation} \label{alg-hs}
\begin{split}
&h_2^{+}=h_1^2 \ \mathrm{and} \ h_3^{+}=\tfrac{3}{2}\{h_1,h_2-h_1^2\}\\
\end{split}
\end{equation}
\item[b)] the infinitesimal Einstein deformation $h_1 \in \varepsilon(g)$ is such that the cohomology class of the Kodaira-Spencer bracket $[h_1,h_1]^c$ in $H^{0,2}(M,T^{0,1}M)$ vanishes, more precisely 
\begin{equation} \label{obs-i}
\bdel (h_2-h_1^2)+[h_1,h_1]^c=0
\end{equation}
\item[c)]
the tensor $\bdel (h_3^{-}-\tfrac{3}{2}h_1^3)+3[h_1,h_2-h_1^2]^c$ is harmonic, that is 
\begin{equation} \label{harm-I}
\bdel (h_3^{-}-\tfrac{3}{2}h_1^3)+3[h_1,h_2-h_1^2]^c \in \mathrm{Harm}^{0,2}(M,T^{0,1}M)
\end{equation}
\item[d)] the vector fields  $\mathbf{X}_1(g_t), \mathbf{X}_2(g_t)$ and $\mathbf{X}_2(g_t)$ vanish identically over $M$ that is 
$\mathbf{X}_1(g_t)=\mathbf{X}_2(g_t)=\mathbf{X}_3(g_t)=0$.
\end{itemize}
\end{itemize}
\end{teo}
We now comment on the content of Theorem \ref{main-i1} as follows. The vanishing of the cohomology class of $[h_1,h_1]^c$ in \eqref{obs-i} describes the obstruction to third order Einstein 
deformation for the metric $g$. It is the first time that the obstruction to third order Einstein deformation is computed explicitly in the literature. Also recall that the Einstein deformation theory is {\it{unobstructed}} to second order \cite{NS-E,N1}, i.e., whenever $h_1 \in \varepsilon(g)$ there exists a deformation $g_t$ of $g$ which is tangent to $h_1$ and satisfies $(\Ric^{g_t})^{(k)}(0)=0$ for $1 \leq k \le 2$; thus obstructions appear only at order $3$. 

Recall that a deformation $J_t$ of the complex structure is a small time
family of complex structures with $J_0=J$; the infinitesimal deformation space for $J_t$ coincides with 
$\varepsilon(g)$ since $E<0$. The obstruction to Einstein deformation in \eqref{obs-i} actually describes the obstruction to second order deformation of the complex structure $J$; in addition, the vanishing of the harmonic tensor $\bdel (h_3^{-}-\tfrac{3}{2}h_1^3)+3[h_1,h_2-h_1^2]^c $ from part (c) above is equivalent with saying that the deformation theory of $J$ is unobstructed to order $3$.
\begin{rema} \label{rkm-1}
By Theorem \ref{main-i1} the negative K\"ahler-Einstein metrics on Horikawa's surface \cite{Hori} (see also Remark \ref{C-Y}) are thus deformation rigid in the sense they do not admit non-trivial Einstein deformations; hence those metrics are isolated points in the moduli space $\mathscr{M}_1^{\mathrm{E}} \slash \bfG$ of Einstein metrics. See also Remark \ref{rkm-4} for a different, purely $4$-dimensional, argument.
\end{rema}
To prove \eqref{obs-i} a crucial ingredient is an entirely new --and fully explicit-- formula for the $\LL^2$-norm of the Kodaira-Spencer bracket which reads as follows.
\begin{teo} \label{NKS-I}
Let $h \in \Gamma \left (\Sym^{2,-}TM \right )$ satisfies $\bdel h=0$ and $\delta^gh=0$. Then 
\begin{equation} \label{norm-KSI}
2 \Vert [h,h]^c\Vert^2_{L^2}=\la (\Delta_E+4\ring{R})h^2,h^2\ra_{L^2}+\la \delta^{\star_g}(2\delta^gh^2+\di\!\tr h^2),h^2\ra_{L^2}.
\end{equation} 
\end{teo}
Here $\Delta_E$ denotes the Einstein operator and $\delta^{\star_g}$ is the formal adjoint of the divergence $\delta^g :\Gamma \left ( \Sym^2TM\right ) \to \Omega^1M$.

Remarkably, the $\LL^2$-norm of the Kodaira-Spencer bracket is expressed in terms of an operator acting solely on $h^2$. As a matter of fact we develop an entirely general 
formula in section \ref{norm-bra-new} of the paper with no assumptions on the tensor $h \in \Gamma \left (\Sym^{2,-}TM \right )$.
\subsection{Role of the Cayley transform}\label{Cay}
Letting $h_t:=g^{-1}g_t$ we define the Cayley transform 
$$ \mathrm{C}_t:=(1-h_t)(1+h_t)^{-1}.
$$
Note that the definition of $\mathrm{C}_t$ makes sense since the tensor $1+h_t>0$ with respect to the metric $g$.

The result in Theorem \ref{main-i1} has in fact a conceptual explanation in terms of the Cayley transform. Its use 
clarifies the geometric significance of the polynomial expressions in the coefficients of the Taylor expansion of $g_t$ which appear 
in parts a) and c) of our main result; also it serves to relate explicitely 
the obstruction to third order integrability in \eqref{obs-i} as well as the third order Einstein equation in \eqref{harm-I} to the complex deformation theory of $J$.

To proceed, we consider the Taylor series expansion 
$\mathrm{C}_t=\sum \limits_{n=0}^{\infty} \tfrac{t^n}{n!}\mathrm{C}_n$ for small time $t$. The first $3$ coefficients are given by $\mathrm{C}_0=0$ as well as 
\begin{equation} \label{cay-co}
\begin{split}
&-2\mathrm{C}_1=h_1, \ -2\mathrm{C}_2=h_2-h_1^2\\
&-2\mathrm{C}_3=h_3-\tfrac{3}{2}\{h_1,h_2-h_1^2\}-\tfrac{3}{2}h_1^3.
\end{split}
\end{equation}
The algebraic part in Theorem \ref{main-i1}, that is the constraint in \eqref{alg-hs} is to say that 
$$ \mathrm{C}^{+}_3= \mathrm{C}^{+}_2=0.
$$
Equivalently, 
$$J\mathrm{C}_t =-\mathrm{C}_tJ+o(t^4).$$
\begin{rema} \label{g-O}
Let $\omega=g(J \cdot ,\cdot) \in \Omega^2M$ be the K\"ahler form of the K\"ahler structure $(g,J)$. A Riemannian 
metric $g^{\prime}$ is said to be compatible with $\omega$, with shorthand notation $g^{\prime} \in \omega$, 
provided that the composition $I^{\prime}:=(g^{\prime})^{-1} \omega$ defines an almost complex structure on 
$M$. In general $I^{\prime}$ is not integrable, however when that is the case $(g^{\prime},I^{\prime})$ is automatically a K\"ahler structure. 
The vanishing of the coefficients $\mathrm{C}^{+}_k$ with $k=2,3$ can be intepreted thus in a more intrinsic way; that is after gauge normalisation we have $g_t \in \omega$ up to order $3$.
\end{rema}
The obstruction to deformation in the same theorem (see equation \eqref{obs-i}) guarantees that 
$ \bdel \mathrm{C}_1=0, \ \tfrac{1}{2}\bdel \mathrm{C}_2=[\mathrm{C}_1, \mathrm{C}_1]^c$
so the Cayley transform satisfies, up to order $2$, the Maurer-Cartan equation 
$$ \bdel \mathrm{C}_t=[\mathrm{C}_t,\mathrm{C}_t]^c+o(t^3).
$$
In addition the Einstein equation to order $3$ in Theorem \ref{main-i1} (see \eqref{harm-I}) ensures that the tensor $\tfrac{1}{2}\bdel \mathrm{C}^3-3[\mathrm{C}_1,\mathrm{C}_2]^c$ belongs to $\ker \bdel \cap 
\ker \delta^g$ and is thus harmonic; thus also \eqref{harm-I} is consistent with the expectations coming from the deformation theory of $J$.

The next observation relates directly to the interpretation, again in terms of the Cayley transform, of the gauge normalisation used in Theorem \ref{main-i1}.
\begin{rema} \label{conj-ga}
The gauge normalisation to $\delta^gh_1=\mathbf{X}_2(g_t)=\mathbf{X}_3(g_t)=0$ which is granted by part (d) in Theorem \ref{main-i1} may be re-interpreted in terms of the Cayley transform. Indeed, this gauge normalisation is equivalent to 
\begin{equation} \label{gau-int}
\delta^g\mathrm{C}_t=-\tfrac{1}{2}(1+\mathrm{C}_t)\di \ln \det(1+\mathrm{C}_t)+o(t^4).
\end{equation}
See Proposition \ref{cayley-1} in the body of the paper .
\end{rema}
In light of these facts it is reasonable to formulate the following 
\begin{conj} \label{conj1}
Let $(M,g,J)$ be compact and Einstein with $E<0$ and let $g_t \in \mathscr{M}_1$ be a small time Einstein deformation. Then the following hold 
\begin{itemize}
\item[(i)] there exists a gauge transformation  $f_t \in \mathbf{G}$ with $f_0=\id$ such that the Cayley transform $\mathrm{\tilde{C}}_t$ of the family $\tilde{g}_t:=f_t^{\star}g_t \in \mathscr{M}_1$ is normalised 
according to 
\begin{equation} \label{g-conj}
\delta^g\mathrm{\tilde{C}}_t=-\tfrac{1}{2}(1+\mathrm{\tilde{C}}_t)\di \ln \det(1+\mathrm{\tilde{C}}_t)+\grad_g \varphi_t
\end{equation}
for a time dependent family of smooth functions $\varphi_t$ on $M$
\item[(ii)] assuming that the Cayley transform $\mathrm{C}_t$ of $g_t$ is normalised as  in (i) we must have
$$ \mathrm{C}_t \in \Gamma \left (\Sym^{2,-}TM \right )\ \mathrm{and} \ \bdel \mathrm{C}_t=[\mathrm{C}_t,\mathrm{C}_t]^c
$$
as well as the divergence equation $\delta^g\mathrm{C}_t=-\tfrac{1}{2}(1+\mathrm{C}_t)\di \ln \det(1+\mathrm{C}_t)$
\item[(iii)] the almost complex structure given by $J_t=(1-\mathrm{C}_t)^{-1}J (1-\mathrm{C}_t)$ is integrable, thus $(g_t,J_t)$ is a family K\"ahler-Einstein structures on $M$ with constant K\"ahler form $\omega$.
\end{itemize}
\end{conj}
In section \ref{nos} of the paper we prove, up to order $3$, the normalisation in part (i) above for arbitrary Riemanniann metrics with fixed volume; the arguments utilised for that proof extend formally to obtain \eqref{g-conj} as formal power series. That is the gauge transformation $f_t$ is constructed by integrating a time dependent vector field build from a formal power series; the analytical aspects related to the convergence of that power series are a separate issue.

Also note that part (iii) in Conjecture \ref{conj1} is just restating in a geometric way the technical aspects in 
(ii) of the same conjecture; indeed 
having $g_t(J_t \cdot, \cdot)=\omega$ follows from 
$\mathrm{C}_t$ in $\Gamma \left (\Sym^{2,-}TM \right )$; also 
the Maurer-Cartan equation $\bdel \mathrm{C}_t=[\mathrm{C}_t,\mathrm{C}_t]^c$ is equivalent with
the integrability of $J_t$. We may now restate Theorem \ref{main-i1} as follows.

\begin{teo} \label{main3-i}
Let $(M,g,J)$ be compact and K\"ahler-Einstein with $E<0$ and let $g_t \in \mathscr{M}_1$ be a small time Einstein deformation. Then Conjecture \ref{conj1} is true up to order $3$ in the Taylor series expansion of the Cayley transform.
\end{teo}
Following the discussion in \cite{S-Sw} of LeBrun's results on supreme Einstein metrics
we also record below that Conjecture \ref{conj1} is essentially true in real dimension $4$ for specific dimensional reasons related to the use of Seiberg-Witten theory.
\begin{rema} \label{rkm-4}
Let $(M^4,g,J)$ be compact and K\"ahler-Einstein with $E<0$. Consider the Einstein–Hilbert functional 
$\mathcal{S}: \mathscr{M} \to \mathbb{R}$ defined by
$\mathcal{S}(g):=\int_M \mathrm{scal}_g \vol_g$ with $g$ in 
$\mathscr{M}$, 
where the latter denotes the space of all Riemannian metrics on $M$. The functional $\mathcal{S}$ is locally constant on the Einstein moduli space 
$\mathscr{M}_1^{\mathrm{E}} \slash \bfG$, see \cite{S-Sw}[Theorem 3.6].
On the other hand, by work in \cite{LeB-1,LeB-2}, the metric $g$ is supreme and conversely any other supreme Einstein metric on $M$ is K\"ahler-Einstein; thus following \cite{S-Sw}[Theorem 6.12] we have that any nearby Einstein metric, in this case $g_t$, is also supreme, hence K\"ahler–Einstein for some complex structure $J_t$. Moreover since the complex structures $J_t$ and $J$ are homotopic for small $t$, they must have the same first Chern class, hence the cohomology class of the K\"ahler form $\Omega_t=g_t(J_t \cdot, \cdot)$  coincides with that of $\omega$; this allows constructing, by Moser's method, a family of diffeomorphisms of $M$ with 
$\phi_t^{\star}\Omega_t=\omega$. It follows that the metric $(\phi^{-1}_t)^{\star}g_t$ is compatible with 
$\omega$ (see also Remark \ref{g-O}) and also K\"ahler-Einstein. Note that the results in \cite{LeB-1,LeB-2} are established using 
Seiberg-Witten theory and do not have analogues in arbitrary dimension.
\end{rema}

\subsection{Key ideas and outline of proofs} \label{key}
Let $(M,g)$ be Einstein with Einstein constant $E$. We consider small time deformations $g_t \in \mathscr{M}_1$ of $g$ with fixed volume and such that the metric $g_t$ is Einstein with Einstein constant $E$. In addition we consider the Taylor 
series expansion $g^{-1}g_t=\id+th_1+\tfrac{t^2}{2!}h_2+\tfrac{t^3}{3!}h_3+o(t^4)$. The first result in this paper is to derive the explicit differential equation which encodes the vanishing of the third order derivative of the Ricci tensor of $g$ at $t=0$. Note that, to the best of our knowledge, it is the first time such an equation is obtained in the literature. We prove 
\begin{teo} \label{D3-intro}
The Einstein equation to order $3$ reads 
\begin{equation} \label{E3-intro}
(\widetilde{\Delta}_E+\delta^{\star_g}\di\tr)\mathbf{u}_3=3\mathbf{v}(h_1,h_2)-
\tfrac{3}{4}\{h_1,(\widetilde{\Delta}_E+2E+\delta^{\star_g}\di\tr )h_2\}+3\mathbf{w}(h_1).
\end{equation}
\end{teo}
In the above statement the operator $\widetilde{\Delta}_E$ is a gauge-invariant perturbation of the Einstein 
operator $\Delta_E$ and $\mathbf{u}_3$ 
is the polynomial quantity 
\begin{equation*} 
\mathbf{u}_3:=-h_3+\tfrac{9}{4}\{h_1,h_2-h_1^2\}-\tfrac{3}{2}h_1^3 \in \Gamma \left (\Sym^2(TM) \right ).
\end{equation*}
The operator 
$\bfv : \Gamma \left ( \Sym^2TM\right ) \oplus \Gamma \left ( \Sym^2TM\right ) \to 
\Gamma \left ( \Sym^2TM\right )$ is determined from 
\begin{equation} \label{bfv-def}
\la \bfv(h_1,h_2),h_3 \ra_{L^2}=\mathfrak{S}_{abc} \la \delta^g[h_a,h_b]^{\FN},h_c \ra_{L^2}
\end{equation}
where $\mathfrak{S}$ indicates the cyclic sum and $[\cdot, \cdot ]^{\FN}$ is the Fr\"olicher-Nijenhuis bracket acting on symmetric tensors. In fact, as showed in \cite{NS-E}, this operator 
governs the second order Einstein deformation theory for the metric $g$. The new operator which appears when differentiating the Einstein equation to third order is $ \mathbf{w} : \Gamma \left (\Sym^2TM \right ) \to 
\Gamma \left (\Sym^2TM \right )$ defined according to 
\begin{equation} \label{w-def-i}
\la \mathbf{w}(h), H \ra_{L^2}:=\la 2\delta^gQ(h)-\mathbf{p}(h^2,h)-h \circ \mathbf{p}(h,h),H \ra_{L^2}.
\end{equation}
Note the algebraically, $\mathbf{w}$ is cubic in $h$ and that the main building blocks in its definition are given by 
\begin{itemize}
\item[$\bullet$] the quadratic differential operator 
$ \mathbf{p} : \Gamma \left (\Sym^2TM \right ) \oplus \Gamma \left (\Sym^2TM \right ) \to \Gamma \left (\End( TM )\right )
$
defined in a local orthonormal frame $\{e_i\}$ according to 
\begin{equation*}
g(\mathbf{p}(h_1,h_2)X,Y):=g((\nabla^g_{X}h_1)e_i,2(\nabla^g_{e_i}h_2)Y-(\nabla^g_{Y}h_2)e_i)
\end{equation*}
where $\nabla^g$ is the Levi-Civita connection of $g$. Above and in the rest of the paper we use systematically Einstein's summation convention.
\item[$\bullet$] the cubic differential operator $ Q: \Gamma \left (\Sym^2TM \right ) \to \Omega^2(M,TM)$ given by 
\begin{equation*} 
Q(h):=-\di_{\nabla^g}\!h^3+h^2 \sharp \di_{\nabla^g} h+\di_{\nabla^g} h(h \cdot ,h \cdot).
\end{equation*}
Above  $\di_{\nabla^g}:\Omega^{k}(M,TM) \to \Omega^{k+1}(M,TM)$ indicates the coupled exterior differential and 
the algebraic action $H \sharp \alpha:=\alpha(H \cdot, \cdot)+\alpha(\cdot ,H \cdot)$ whenever the pair 
$(H,\alpha)$ belongs to $\Gamma \left (\Sym^2TM \right ) \oplus \Omega^2(M,TM)$.
\end{itemize}
The proof of Theorem \ref{D3-intro} builds on the crucial observation (see Theorem \ref{E-sym} in the body of the paper) that the Einstein equation for $g_t$ can be brought to explicit divergence form and is thus amenable to direct differentiation with respect to $t$. 
\begin{rema} \label{apps--intro}
Theorem \ref{D3-intro} has many potential applications in all instances where the infinitesimal 
deformation space $\varepsilon(g)$ can described explicitly. These include nearly-K\"ahler \cite{MNS} and 
nearly $\mathrm{G}_2$-structures \cite{Al-Se,NSE_LMS}; of particular interest are also 
the cases of squashed $3$-Sasaki metrics in dimension $7$ \cite{NS-G2} and also the odd complex 
Grassmanians \cite{NS-E,HaSe}. Thus using theorem \ref{D3-intro}  is expected to produce obstructions to the higher order Einstein deformation problem for these geometries.
\end{rema}
 
The rest of the paper is devoted to fully render explicit equation \eqref{E3-intro} in the case when 
$(M,g,J)$ is a compact K\"ahler-Einstein manifold with negative scalar curvature. The main task consists in computing the operator $\mathbf{w}(h)$ when $h \in \Gamma \left ( \Sym^{2,-}TM\right )$ satisfies $\bdel h=0$ and 
$\delta^gh=0$. This is performed in several steps as follows; we consider $\bfv_2(h):=\bfw(h)+\bfv(h,h^2)$ and split 
\begin{equation*}
\bfv_2(h)=\bfv_2^{-}(h)+\bfv_2^{+}(h)
\end{equation*}
according to the type decomposition \eqref{typeS-i}. 

The component $\bfv_2^{-}$ is determined explicitly in section \ref{+part} of the paper, in several steps;  
first we determine the component of $\bfv$ acting on $\Gamma \left (\Sym^{2,-}TM \right ) \oplus \Gamma \left (\Sym^{2,-}TM \right ) \oplus \Gamma \left (\Sym^{2,+}TM \right ) $ by using the integral formulas for 
quantities $\la H_{+} \sharp \di_{\nabla^g}H_{-},\di_{\nabla^g}H_{-}\ra_{L^2}, H_{\pm} \in \Gamma \left ( \Sym^{\pm}TM
\right )$ previously established in \cite{N1}. The comparison formulas Fr\"olicher-Nijenhuis v.s. Kodaira-Spencer brackets proved in the same reference and some specific representation theory arguments eventually show that 
\begin{equation*}
\bfv_2^{-}(h)=-\left (\bfv(h,h^2)+\widetilde{\Delta}_Eh^3 \right )^{-}+\{h,\left ( \Delta_E+E+2\ring{R}+\tfrac{1}{4}\delta^{\star_g,+}\di\!\tr\right )h^2\}.
\end{equation*}
See proof of Theorem \ref{E3-1}, in particular equation \eqref{aux1}. Above $\delta^{\star_g,\pm }$ are the components of $\delta^{\star_g}$ with respect to 
\eqref{typeS-i}. In order to obtain the obstruction to deformation in \eqref{obs-i}, that is the obstruction to solving the latter equation with respect to $\mathbf{u}_3$, additional steps are 
required: 
\begin{itemize}
\item[$\bullet$] we take the $\LL^2$-scalar product with $h_1$ in \eqref{E3-intro} and use the expression for 
$\mathbf{v}_2^{-}$ found above
\item[$\bullet$] we use the fact ( see \cite{N1}[Proposition 4.14]) that the restriction of $\mathbf{v}$ to 
$\Gamma \left ( \Sym^{2,-}TM \right )$ is fully determined by the divergence of the Kodaira-Spencer bracket and
divergence terms. Since $h_2-h_1^2$ lies (after suitable gauge normalisation ) in $\Gamma \left ( \Sym^{2,-}TM \right )$ (see \cite{N1}[Theorem 1.6]) we end up with an explicit formula for $\bfv(h_1,h_1,h_2-h_1^2)$, see Lemma \ref{v-fin}.
\end{itemize}
Summarising the considerations above thus shows that the pair $(h_2,h_1)$ is constrained according to 
\begin{equation*}
2\la \delta^{\star_g}[h_1,h_1]^c, h_2-h_1^2 \ra_{L^2}+\la \left ( \Delta_E+4\ring{R}+\delta^{\star_g,+}(2\delta^g+\di\!\tr)\right )h_1^2,h_1^2\ra_{L^2}=0.
\end{equation*}
Remarkably the second summand above can be described solely in terms of the 
norm of the Kodaira-Spencer bracket by \eqref{norm-KSI}; the integrability condition in \eqref{obs-i} follows then by a straightforward Hodge theory argument.

The component $\bfv_2^{+}$ is determined explicitly in section \ref{+part} of the paper by using representation theory and again the comparison between the Fr\"olicher-Nijenhuis and Kodaira-Spencer brackets 
developed in \cite{N1}. We first develop a general formula for the remaining mixed type component in 
$\bfv$, which is the restriction of $\bfv$ to $\Gamma \left (\Sym^{2,+}TM \right ) \oplus \Gamma \left (\Sym^{2,+}TM \right ) \oplus \Gamma \left (\Sym^{2,-}TM \right ) $. After also finding explicit formulas for the actions 
$\mathbf{p}(h^2,h)$ respectively $h \circ \mathbf{p}(h,h)$ this leads to showing that $\bfv_2^{+}(h)$ is essentially 
build up from divergences; we show that
\begin{equation*}
\begin{split}
\mathbf{v}_2^{+}(h)=&\delta^{\star_g,+}\delta^gh^3-\tfrac{1}{2}\{h,\delta^{\star_g,-}\delta^gh^2\},
\end{split}
\end{equation*}
see Proposition \ref{w+}. As a direct consequence, we establish that after a gauge normalisation of $g_t$ such 
that the vector fields $\mathbf{X}_1(g_t)=\mathbf{X}_2(g_t)=0$, the component on 
$\Gamma \left ( \Sym^{2,+}TM \right )$ of \eqref{E3-intro} simply reads 
\begin{equation} \label{E333+}
\Delta_E\mathrm{C}_3^{+}=-\delta^{\star_g,+}\mathbf{X}_3(g_t).
\end{equation}
See Theorem \ref{vmmm2} in the paper. Equation \eqref{E333+} provides a justification for the expression of 
$\mathbf{X}_3(g_t)$ as given in \eqref{norm-3i}; it also leads to the vanishing of $\mathrm{C}_3^{+}$ by  operating a gauge transformation in $\mathbf{G}$ after which $ \mathbf{X}_3(g_t)$ is a gradient. 

\subsection{Concluding remarks} \label{conc-rmks}
The next step in proving Conjecture \ref{conj1} is thus differentiating the Einstein equation to higher order. 
Below we summarise some our expectations as follows.
Since the Taylor coefficient
$$-2\mathrm{C}_4=h_4+4\{h_1,\mathrm{C}_3\}-3(h_2-h_1^2)^2-3\{h_1^2,h_2-h_1^2\}-3h_1(h_2-h_1^2)h_1-3h_1^4$$
the expectation after differentiating the Einstein equation to order $4$ is to 
establish the identity $h_4^{+}=2\{h_1,h_3^{-}-\tfrac{3}{2}h_1^3\}
+3(h_2-h_1^2)^2+3h_1^4$, in other words to show that $\mathrm{C}_4^{+}=0$. The main step in this direction is to establish an analogue for \eqref{E333+}; more precisely, normalising the family $g_t$ as in part d) of Theorem \ref{main-i1} we conjecture that the component on $\Gamma \left (\Sym^{2,+}TM \right )$ for the Einstein equation to order $4$ i.e. $(\Ric^{g_t})^{(4)}(0)=0$, is given by 
$$\Delta_E\mathrm{C}_4^{+}=-\delta^{\star_g,+}\mathbf{X}_4(g_t)$$
for some vector field $\mathbf{X}_4(g_t)$. 
The correct normalisation of $h_4$ under the group $\mathbf{G}$ should be follow from the explicit form of 
$\mathbf{X}_4(g_t)$, once found. It is however reasonable to expect this gauge  normalisation to be an extension of \eqref{gau-int} to order $4$.

The integrability condition for the almost compplex structure $J_t$ to order $3$ (see also \eqref{harm-I}) namely having that $\bdel(h_3^{-}-\tfrac{3}{2}h_1^3)+3[h_1,h_2-h_1^2]^c=0$ is expected to appear only after differentiating the Einstein to order $5$.
\subsection*{Acknowledgments}
Paul-Andi Nagy was supported by the Institute for Basic Science (IBS-R032-D1). It is a pleasure to thank Hans-Joachim Hein and Uwe Semmelmann for useful exchanges during the preparation of this paper.
\section{Preliminaries} \label{prel}
Let $(M^n, g)$ be a Riemannian manifold and denote with $\nabla^g$ the Levi-Civita connection of $g$ on $TM$ and all
tensor bundles. For the Riemannian curvature tensor $R^g$ we use the convention $R^g(X,Y)=(\nabla^g)^2_{Y,X}-(\nabla^g)^2_{X,Y}$
for tangent vectors $X, Y \in TM$. 
In the following we will make  a notational difference between the Ricci form $\ric^g$ considered as a symmetric bilinear form and the Ricci tensor $\Ric^g$  
considered as a symmetric endomorphism. We will denote by $g^{-1}$ the isomorphism, induced by the metric, between symmetric bilinear forms and symmetric
endomorphisms, e.g. we have $g^{-1} \ric^g = \Ric^g$. The subbundle of $g$-symmetric endomorphisms   is denoted by $\Sym^2TM \subseteq \End(TM)$. It is preserved 
by $\nabla^g$ and  the curvature action $h \mapsto \ring{R}h:=\sum_i R^g(e_i,\cdot)he_i$, where $\{ e_i\}$ is some $g$-orthonormal frame. 

We will use the coupled exterior differential $\di_{\nabla^g}:\Omega^{k}(M,TM) \to \Omega^{k+1}(M,TM)$. For $k=0$, i.e. on vector fields,
it coincides with the covariant derivative $\nabla^g$. For $k=1$, i.e. on endo\-morphisms $h$ considered as elements in $ \Omega^1(M,TM)$, it  
can be described according to 
$\di_{\nabla^g}h(X,Y)=(\nabla^g_Xh)Y-(\nabla_Y^gh)X$. The operator formally  adjoint to $\di_{\nabla^g}$  is the divergence operator  $\delta^g:\Omega^{k}(M,TM) \to \Omega^{k - 1}(M,TM)$ defined according to the convention $\delta^g = -  \sum_i e_i \lrcorner \nabla^g_{e_i}$.

Restricting the divergence to symmetric $2$-tensors we have $\delta^g : \Gamma(\Sym^2TM) \to  \Omega^1M$. Its formal adjoint 
$\delta^{\star_g} : \Omega^1M \to \Gamma(\Sym^2TM)$ is the symmetric part of $\nabla^g$, i.e.  $\nabla^g=\delta^{\star_g}+\frac{1}{2}\di$ on $\Omega^1M$.
As a consequence we record for subsequent application the identities 
\begin{equation*}
\begin{split}
&\tr (\delta^{\star_g} \alpha)  = - \di^{\star}\!\alpha \ \mathrm{and} \ \delta^{\star_g}\alpha =\frac{1}{2}g^{-1}\L_{\alpha^{\sharp}}g
\end{split}
\end{equation*}
which hold for any $1$-form $\alpha$. The last identity actually characterises the kernel of $\delta^{\star_g}$
as $1$-forms dual to Killing vector fields.

In dealing with the trace and divergence of symmetric tensors we will use the Bianchi operator $\D^g: \Gamma(\Sym^2TM) \to \Omega^1M$ given by  $\D^g=2\delta^g+\di\! \circ \tr$. 
Certainly $\D^g$ vanishes on the space of the so-called $\TT$-tensors defined by
$$
\TT(g):=\{h \in \Gamma(\Sym^2TM) : \delta^g h=0  \ \mathrm{and} \ \tr(h)=0\} . 
$$
However terms involving divergence and trace must be kept track of in order to obtain gauge invariant equations. 
For latter use recall that $\D^g \, \Ric^g =0$, which is a consequence of the differential Bianchi identity 
(see \cite{Besse}, 12.33).

On symmetric $2$-tensors and for Einstein metrics $g$ with Einstein constant $E$ we will consider the Einstein operator 
$\Delta_E := (\nabla^g)^{\star}\nabla^g-2\ring{R} $. Note that $\Delta_E = \Delta_L - 2E$, where $\Delta_L$ is the Lichnerowicz
Laplacian on symmetric $2$-tensors (see \cite{Besse}, 1.143). We will also use a perturbation of the Einstein operator, 
the differential operator $\widetilde{\Delta}_E:
\Gamma(\Sym^2TM) \to \Gamma(\Sym^2TM)$ given by 
\begin{equation}\label{D-tilde}
 \widetilde{\Delta}_E \;:= \;  \Delta_E \, - \, 2\delta^{\star_g} \circ (\delta^g+\frac{1}{2}\di \circ \tr)\; = \; \Delta_E \, - \, \delta^{\star_g} \circ \D^g.
\end{equation}

The following Weitzenb\"ock formula on  $2$-tensors  
(see \cite{bourg}, Proposition 4.1) will be used at several  places in this paper.
\begin{equation}\label{wz0}
\delta^g \circ \di_{\nabla^g} \, + \, \di_{\nabla^g} \circ \delta^g \; =\; (\nabla^g)^{\star}\nabla^g \, + \, E \,  - \, \ring{R} \;= \;\Delta_E\,  + E \,  +\,  \ring{R}  . 
\end{equation}
Frequently it will be applied under the form 
\begin{equation} \label{wz1}
\delta^g \circ \di_{\nabla^g} \;=\;  -\nabla^g \circ \delta^g  \, + \, \Delta_E \, + E+\,  \ring{R} \; =\;  - (\delta^{\star_g}+\frac{1}{2}\di) \circ \delta^g\, +\, \Delta_E\, +E+\, \ring{R} .
\end{equation}
We still need another Weitzenb\"ock form, this time on $1$-forms. Here we have
\begin{equation} \label{wz2}
2 \delta^g \delta^{\star_g}  \, - \,   \di \di^{\star}  \; = \;  \Delta \, - \,  2E ,
\end{equation}
where $\Delta = \di \di^{\star} + \di^{\star} \di$ is the Hodge Laplacian and $g$ is again an Einstein metric. Note that the formula  is a special case of a 
Weitzenb\"ock formula on symmetric tensors of arbitrary degree  (see \cite{AU}, proof of Proposition 6.2, for further details).
\begin{lema} \label{tr-00}
Let $H$ belong to $\Gamma \left (\Sym^2TM \right )$. Then 
\begin{itemize}
\item[(i)] $\tr \delta^g\di_{\nabla^g}H=\di^{\star}(\delta^gH+\di\!\tr(H))$
\item[(ii)] $\tfrac{1}{2}\tr \widetilde{\Delta}_E H=\di^{\star}(\delta^gH+\di\!\tr(H))-E\tr(H)$.
\end{itemize}
\end{lema}
\begin{proof}
Record the identity $\tr \circ \delta^{\star_g}=-\di^{\star} $ and also that $\tr \ring{R}H=E \tr(H)$. Both claims follow now from \eqref{wz1}.
\end{proof}

\subsection{The Fr\"olicher-Nijenhuis bracket} \label{FN-br}

On a given manifold $M$ we first recall that the Fr\"olicher-Nijenhuis bracket  for sections $h$ of $\End(TM)$ reads 
$$
[h,h]^{\FN}(X,Y) \, = \, -h^2[X,Y] \, + \, h([hX,Y] \, + \, [X,hY]) \, - \, [hX,hY]  .
$$ 
See also \cite{Kollar}[section 8] for a detailed discussion of the Fr\"olicher-Nijenhuis bracket of arbitrary degree forms in $\Omega^{\star}(M,TM)$. From now on assume that $g$ is a Riemannian metric on $M$ and also that $M$ is compact and connected. Recast in terms of the Levi-Civita connection $\nabla^g$ of $g$ this bracket reads 
\begin{equation} \label{bra-na}
[h,h]^{\FN}(X,Y) \, = \, -(\nabla^g_{hX}h)Y \, + \, (\nabla^g_{hY}h)X  \, + \, (h  \circ \di_{\nabla^g}h)(X,Y) .
\end{equation}
Here  $\di_{\nabla^g}$ acts on $h$  considered as  an element of $\Omega^1(M, TM)$.
More concisely 
\begin{equation} \label{bra-nac}
[h,h]^{\FN} \, = \, -h\sharp \di_{\nabla^g}h \, + \, \di_{\nabla^g}h^2
\end{equation}
where the algebraic action $\alpha \in \Lambda^2(M,TM) \mapsto h \sharp \alpha \in \Lambda^2(M,TM)$ is defined according to $h \sharp \alpha(X,Y)=\alpha(hX,Y)+\alpha(X,hY)$.
The Fr\"olicher-Nijenhuis bracket is extended to a symmetric bracket on the space  $\Gamma(\Sym^2TM )$ via 
\begin{equation} \label{bra-nac2}
2[h_1,h_2]^{\FN} \, = \, -(h_1 \sharp \di_{\nabla^g}h_2 \, + \, h_2 \sharp \di_{\nabla^g}h_1) \, + \, \di_{\nabla^g} \{h_1,h_2\}
\end{equation}
where the algebraic anti-commutator $\{h_1,h_2\}:=h_1 \circ h_2+h_2 \circ h_1$. In particular we have 
$[\id ,h]^{\FN}=0$ for all $h \in \End(TM)$. 
Below we recall some properties of the divergence of the Fr\"olicher-Nijenhuis bracket as follows.
\begin{lema} \cite{N1}[Lemma 3.1]\label{sym-bra}
Assume that $\delta^g h=0$. The following hold 
\begin{itemize}
\item[(i)] the tensor $\delta^g\left ([h,h]^{\FN}-h \circ \di_{\nabla^g}h \right )- (\ring{R}h) \circ h$ is symmetric
\item[(ii)] we have 
\begin{equation*}
\begin{split}
\la \delta^g ([h,h]^{\FN},H \ra_{L^2}=&\la h \circ \di_{\nabla^g}h, \di_{\nabla^g}H \ra_{L^2}+
\la \di_{\nabla^g}h, H \circ \di_{\nabla^g}h \ra_{L^2}\\
&-\la \nabla^g_{he_i}h, \nabla^g_{e_i}H \ra_{L^2}\\
&+\la (\ring{R}h) \circ h-Eh^2-\mathrm{L}(h),H\ra_{L^2}
\end{split}
\end{equation*}
for all $H \in \Gamma \left (\Sym^2TM \right )$.
\end{itemize}
\end{lema}
We also recall the following alternative way of computing the operator $\bfv$, as defined in \eqref{bfv-def},  directly from the Fr\"olicher-Nijenhuis bracket. We have 
\begin{equation} \label{bfv-1}
\begin{split}
\la \bfv(h_1,h_2),H\ra_{L^2}=&2\la [h_1,h_2]^{\FN}, \di_{\nabla^g}H \ra_{L^2}-\la H \sharp \di_{\nabla^g}\!h_1,
\di_{\nabla^g} \!h_2\ra_{L^2}+
\la \tilde{\mathrm{L}}(h_1,h_2),H\ra_{L^2}\\
\end{split}
\end{equation}
for all $h_1,h_2$ and $H$ in $\Gamma \left (\Sym^2TM \right )$. Here the operator 
\begin{equation} \label{L-t}
2\tilde{\rm L}(h_1,h_2):=\{h_1,\delta^g \di_{\nabla^g}\!h_2\}+\{h_2,\delta^g \di_{\nabla^g}\!h_1\}-\delta^g
\di_{\nabla^g}\{h_1,h_2\}
\end{equation}
The proof is entirely similar to the argument used in the proof of Theorem 4.4 in \cite{NS-E} and hence will be omitted. Following \cite{Besse} we also recall the following 
\begin{defi} \label{IED}
Let $(M,g)$ be compact and Einstein. The space of essential infinitesimal Einstein deformations of $g$ is given by
$$ \varepsilon(g):=\{h \in \Gamma \left (\Sym^2TM \right ) : 
\Delta_Eh=0, \ \delta^gh=0 \ \mathrm{and} \ \tr(h)=0\}.
$$
\end{defi} 
To finish this section we also recall the following result which clarifies the structure theory of Einstein deformations up to second order.
\begin{teo} \cite{N1}[Theorem 1.1]\label{norm-thm}
Let $(M,g)$ be compact and Einstein and let $g_t \in \mathscr{M}_1$ be a family of Einstein metrics with Taylor expansion $g^{-1}g_t=\id+th_1+\tfrac{t^2}{2}h_2+o(t^3)$ at $t=0$. Up to a time dependent gauge 
transformation $g_t \mapsto f_t^{\star}g_t$ where $f_t \in \mathbf{G}$ we may assume that
\begin{itemize}
\item[(i)] the first order variation $h_1$ belongs to $\varepsilon(g)$
\item[(ii)] the second order variation $h_2$  satisfies 
\begin{equation} \label{o2-new}
\Delta_E(h_2-h_1^2)=\tfrac{1}{2}\widetilde{\Delta}_Eh_1^2+Eh_1^2-\bfv(h_1,h_1)
\end{equation}
as well as 
\begin{equation} \label{o2-no2new}
\delta^g(h_2-h_1^2)=-\frac{1}{4} \di\!\tr(h_1^2). 
\end{equation}
\end{itemize}
\end{teo}
Following \cite{N1} we note that for deformations $g_t \in \mathscr{M}_1$ we automatically have that the tensors 
$h_1,h_2-h_1^2$ are trace free. In fact the gauge normalisation used in the above result is such 
that the $1$-forms dual to the vector fields $\delta^gh_1$ and $\mathbf{X}_2(g_t)$(see \eqref{norm-3i}) are exact; imposing 
the Einstein equations to first and second order actually forces their vanishing due to having 
$h_1,h_2-h_1^2$ trace free. An extension of this gauge normalisation argument to third order 
is proved later on in Proposition \ref{norm-2nd}.  
\subsection{Facts from K\"ahler geometry} \label{K-g}
Assume that $(M^{2m},g,J)$ is K\"ahler Einstein, with Einstein constant $E$. We briefly review in this section some notation and main facts needed in what follows. The action of $J$ extends to $\Omega^{\star}M$ according to $J\alpha:=\alpha(J \cdot ,\ldots, J \cdot)$.
The bundle of symmetric 2-tensors splits as $\Sym^2 TM=\Sym^{2,+}TM \oplus \Sym^{2,-}TM$ where the summands
$$\Sym^{2,\pm }TM=\{h \in \Sym^2TM : hJ=\pm Jh \}.$$
This allows defining spaces of $\TT$-tensors according to 
$$\TT^{\pm}(M,g):=\TT(M,g) \cap \Gamma \left (\Sym^{2,\pm}TM \right ).$$ 
Furthermore, the operator $\delta^{\star_g}$ splits as 
$$ \delta^{\star_g}X=\delta^{\star_g,+}X+\delta^{\star_g,-}X
$$
according to $\Sym^2TM=\Sym^{2,+}TM \oplus \Sym^{2,-}TM $. The components are explicitly given by 
\begin{equation} \label{delta-K}
4\left (\delta^{\star_g,+}X \right )\circ J =(1+J)\di(JX)^{\flat} \ \ \mathrm{and} \ -2\delta^{\star_g,-}X=\L_{JX}J+\tfrac{1}{2}(1-J)\di\!X^{\flat}
\end{equation}
whenever $X \in TM$. 

Indicate with $\lambda^2M=\{\alpha \in \Lambda^{2}M : \alpha(J \cdot, J \cdot)=-\alpha\}$, so that $\Lambda^2M=\Lambda^{1,1}M \oplus \lambda^2M$. Accordingly, the exterior derivative $\di : \Omega^1M \to \Omega^2M$ splits into complex types according to $\di=\di^{+}+\di^{-}$ where the 
summands are given by 
$$ 2\di^{+}\! \alpha=\di\!\alpha+\di\!\alpha(J \cdot, J \cdot) \in \Omega^{1,1}M \ \mathrm{and} \ 2\di^{-}\! \alpha=\di\!\alpha-\di\!\alpha(J \cdot, J \cdot) \in \lambda^2M 
$$
whenever $\alpha$ belongs to $\Omega^1M$. Similarly, we consider the splitting of real bundles given by 
$$\Lambda^{2}(M,TM)=\Lambda^{1,1}(M,TM) \oplus \lambda^2(M,TM)$$ 
where in analogy with the case of $2$-forms we define 
$\lambda^2(M,TM):=\{\alpha \in \Lambda^{2}(M,TM) : \alpha(J \cdot, J \cdot)=-\alpha\}$. Accordingly, 
whenever $H$ belongs to $\Gamma \left (\Sym^{2,+}TM \right )$ we may consider the splitting $\di_{\nabla^g}H=\di^{+}_{\nabla^g}H+\di^{-}_{\nabla^g}H$ where 
$$\di^{\pm}_{\nabla^g}H:=\tfrac{1}{2}\left (\di_{\nabla^g}H\pm\di_{\nabla^g}H(J \cdot, J \cdot)\right ). 
$$
Then $\di^{+}_{\nabla^g}H \in \Omega^{1,1}(M,TM)$ respectively $\di^{-}_{\nabla^g}H \in \lambda^2(M,TM)$.

In addition we consider the vector bundle splitting 
\begin{equation} \label{split2}
\begin{split}
\lambda^2(M,TM)=\lambda^2_{+}(M,TM)\oplus \lambda^2_{-}(M,TM)
\end{split}
\end{equation}
where 
\begin{equation*} 
\begin{split}
&\lambda^2_{-}(M,TM):= \{\gamma \in \lambda^2(M,TM): \gamma(X,JY)=-J\gamma(X,Y)\} \ \mathrm{and} \\
&\lambda^2_{+}(M,TM):= \{\gamma \in \lambda^2(M,TM): \gamma(X,JY)=J\gamma(X,Y)\}.
\end{split}
\end{equation*}
Note that $\di^{-}_{\nabla^g}H \in \lambda^2(M,TM)$ is a section of $\lambda^2_{+}(M,TM)$ whenever 
$H$ belongs to $\Gamma \left (\Sym^{2,+}TM \right )$. For elements $h \in \Gamma \left (\Sym^{2,-}TM \right )$ we define 
the complex differential 
$$ \bdel h:=\tfrac{1}{2} \left (\di_{\nabla^g}h-\di_{\nabla^g}h(J \cdot ,J \cdot) \right ) 
$$
together with the Kodaira-Spencer bracket 
$$ [h,h]^c:=\tfrac{1}{2}  \left ( [h,h]^{\FN}-[Jh,Jh]^{\FN} \right ).$$
Both $\bdel h$ and $ [h,h]^c$ are sections of the bundle $\lambda^2_{-}(M,TM)$. These algebraic preliminaries allow recalling the comparaison 
formula between the Fr\"olicher-Nijenhuis respectively the Kodaira-Spencer bracket obtained in \cite{N1}.
In fact the type decomposition of the Fr\"olicher-Nijenhuis bracket below will be systematically used in this paper.
\begin{pro}\cite{N1}[Propositions 4.8 and 4.13] \label{t-FNK}
Assume that $h \in \Gamma \left (\Sym^{2,-}TM\right )$. Then 
\begin{equation} \label{t-FNKS}
[h,h]^{\FN}=[h,h]^c+h \circ \bdel h+ \left (\di_{\nabla^g}^{+}h^2-h \sharp \bdel h \right )
\end{equation}
according to $\Lambda^2(M,TM)=\lambda_{-}(M,TM) \oplus \lambda_{+}(M,TM) \oplus \Lambda^{1,1}(M,TM)$. If, in addition, we assume that 
$\bdel h=0$ and $\delta^gh=0$ then 
\begin{equation} \label{t-FNKS1}
\delta^g[h,h]^{\FN}=(\tfrac{1}{2}\Delta_E+\ring{R})h^2+\left (\delta^g[h,h]^c-\delta^{\star_g,-}\delta^gh^2 \right )-
\di^{-}\delta^gh^2
\end{equation}
according to the splitting $\End(TM)=\Sym^{2,+}TM \oplus \Sym^{2,-}TM \oplus \Lambda^2M$.
\end{pro}
We also recall the definition of the perturbed Kodaira-Spencer bracket which explains how the divergence 
of the latter relates to the space of $J$-anti-invariant $\TT$-tensors. Following \cite{N1}[Definition 4.10] define 
\begin{equation} \label{TT-defn}
[h,h]^c_{\TT}:=[h,h]^c +\tfrac{1}{2} \left ( \delta^g h \wedge h-\delta^g(Jh) \wedge Jh \right )
\end{equation}
whenever $h$ belongs to $\Gamma \left (\Sym^{2,-}TM \right )$. The perturbed Kodaira-Spencer bracket $[h,h]^c_{\TT}$ is a section 
of $\lambda^2_{-}(M,TM)$ and extends to a symmetric bracket on the space $\Gamma \left (\Sym^{2,-}TM \right )$ via the usual formula $2[h_1,h_2]_{\TT}^c=
[h_1+h_2,h_1+h_2]_{\TT}^c-[h_1,h_1]_{\TT}^c-[h_2,h_2]_{\TT}^c$. Its main properties are outlined below.
\begin{pro}\cite{N1}[Proposition 4.11] \label{del-TT}Assuming that $h$ is in $\Gamma \left (\Sym^{2,-}TM \right )$ the following hold
\begin{itemize}
\item[(i)] the divergence 
$$ \delta^g [h,h]^c_{\TT} \ \mathrm{belongs \ to} \ \TT^{-}(M,g)
$$
\item[(ii)] we have  $\mathrm{a} [h,h]^c=0$ and $\mathrm{a} [h,h]_{\TT}^c=0$.
\end{itemize}
\end{pro}
In the above statement $\mathrm{a}:\Omega^2(M,TM) \to \Omega^3M$ indicates the total antisymmetrisation map. 
We finish this section section by recalling a first set of Weitzenb\"ock formulas which will be used in this paper are 
\begin{pro} \label{wz_K}
The following hold
\begin{itemize}
\item[(i)]we have
\begin{equation*}
\delta^g\di^{+}_{\nabla^g}H=(\tfrac{1}{2}\Delta_E+\ring{R})H-\delta^{\star_g,-}\delta^gH
-\tfrac{1}{2}\di^{-} \delta^gH
\end{equation*}
for all $H \in \Gamma \left (\Sym^{2,+}TM \right )$
\item[(ii)] we have 
$$ \delta^g \bdel h=\tfrac{1}{2}\Delta_Eh-\delta^{\star_g,-}\delta^gh-\tfrac{1}{2}\di^{-} \delta^g h
$$
for all $h \in \Gamma \left (\Sym^{2,-}TM \right )$. 
\end{itemize}
\end{pro}
See \cite{N1} for details and proofs. In subsequent computations we will also need the following 
\begin{coro} \label{wz-K2}
The following hold
\begin{itemize}
\item[(i)] we have $\Vert \di^{-}_{\nabla^g}H\Vert^2_{L^2}=\la (\tfrac{1}{2}\Delta_E+E)H,H\ra_{L^2}-\Vert \delta^gH \Vert^2_{L^2}$ for all $H \in \Gamma \left (\Sym^{2,+}TM \right )$
\item[(ii)] we have $\la \partial^gh^3, \partial^g h\ra_{L^2}=\la (\tfrac{1}{2}\Delta_E+E
+\ring{R})h^3,h\ra_{L^2}$  for all $h \in  \Gamma \left (\Sym^{2,-}TM \right )$.
\end{itemize}
\end{coro}
The proof follows when combining Proposition \ref{wz_K} with the Weitzenb\"ock formula in \eqref{wz1}. To finish this section we recall the structure of infinitesimal Einstein deformation space for negative K\"ahler-Einstein. 
We have 
\begin{equation} \label{eps-K}
\varepsilon(g)=\varepsilon^{-}(g) \ \mathrm{where} \ \varepsilon^{-}(g):=\{h \in \Gamma (\Sym^{2,-}TM) : \bdel h=0 \ \mathrm{and} \ \delta^gh=0 \}.
\end{equation}
This fact essentially follows from part (i) in Proposition \ref{wz_K} and from the equality 
$$\ker \Delta_E \cap \Gamma (\Sym^{2,+}TM)=\{0\}.$$
See \cite{Koiso-I} for details.
\subsection{Review of second order deformation theory} \label{rev-O2}
We briefly recall here the structure results on second order Einstein deformations for negative K\"ahler-Einstein metrics developed in \cite{N1}. 
\begin{teo} \label{K2-N}
Let $(M,g,J)$ be compact and K\"ahler-Einstein with $E<0$ and let $g_t \in \mathscr{M}_1$ be an Einstein deformation of $g$ with Taylor expansion $g^{-1}g_t=\id+th_1+\tfrac{t^2}{2}h_2+o(t^3)$ at $t=0$. Up to a time dependent gauge 
transformation $g_t \mapsto f_t^{\star}g_t$ where $f_t \in \mathbf{G}$ we have that 
\begin{equation} \label{alg-T1}
h_2^{+}=h_1^2
\end{equation}
together with the divergence equations 
\begin{equation} \label{gauge-1}
\delta^gh_1=0 \ \mathrm{and} \ \delta^g(h_2-h_1^2)+\tfrac{1}{4}\di\!\tr(h_1^2)=0.
\end{equation}
\end{teo}
Because $h_1$ is divergence free we also have $h_1 \in \varepsilon^{-}(g)$, in particular 
$\bdel h_1=0$; also the tensor $h_2-h_1^2$ belongs to $\Gamma \left (\Sym^{2,-}TM \right )$, fact which will be used systematically in the rest of this paper. In addition, as showed in \cite{N1}, the Einstein equation to second order is fully determined by the divergence of the Kodaira-Spencer 
bracket, that is we have 
\begin{equation} \label{2nd-O}
h_2-h_1^2=\mathbf{h}_2-\tfrac{1}{2}\L_{J\grad \mathbf{f}}J 
\end{equation}
where the pair $(\mathbf{h}_2, \mathbf{f})$ in $\TT^{-}(M,g) \oplus C^{\infty}M$ satisfies 
\begin{equation} \label{2nd-034}
\begin{split}
&\Delta_E^g \mathbf{h}_2=-2\delta^g[h_1,h_1]^c \ \mathrm{and} \  (\Delta^g-2E)\mathbf{f}=-\tfrac{1}{2}\tr(h_1^2).
\end{split}
\end{equation}
To end this section we briefly explain the construction of the gauge in Theorem \ref{K2-N}. Consider a family of Einstein metrics $g_t$ in $\mathscr{M}_1$ parametrised as int he statement of that theorem and such that $h_1$ i normalised, up to gauge, to satisfy $\delta^gh_1=0$. Then type considerations show that the component 
on $\Gamma \left (\Sym^{2,+} \right )$ of the second order Einstein equation reads 
\begin{equation*}
\Delta_E(h_2-h_1^2)^{+}=2\delta^{\star_g,+}\mathbf{X}_2(g_t)
\end{equation*}
where the vector field $\mathbf{X}_2(g_t)=\delta^g(h_2-h_1^2)+\tfrac{1}{4}\di\! \tr(h_1^2)$; see Proposition 4.20 and proof of Theorem 4.21 in \cite{N1} for further details. The required gauge transformation is such 
that $\mathbf{X}_2(g_t)$ is a gradient; trace considerations then force the vanishing of $\mathbf{X}_2(g_t) $ and thus yield the latter part in \eqref{2nd-034}. We will show later on in the paper that this pattern re-occurs at order $3$.

\section{Normalisation of the Cayley transform to third order} \label{nos}
We consider the gauge group $\mathbf{G}$ of volume preserving diffeomorphisms of $M$ which naturally acts on the space 
$\mathscr{M}_1$; note that the action of $\mathbf{G}$ preserves the set of Einstein within $\mathscr{M}_1$. In this section we prove that 
the coefficients in the Taylor expansion of any family of Einstein metrics $g_t \in \mathscr{M}_1$ can be normalised up 
to third order by taking into account the gauge group action. To proceed with details we consider 
$$ \tilde{g}_t:=f_t^{\star}g_t
$$
where $f_t \in \mathbf{G}$ is a family of volume preserving diffeomorphisms of $M$ with $f_0=\id$. Differentiating shows that 
$$\tilde{g}_t^{\prime}=f_t^{\star} \left (g_t^{\prime}+\L_{X_t}g_t \right )
$$
where the time-dependent family of vector fields $X_t$ is determined from $f_t^{\prime}=X_t \circ f_t$. Using this formula to differentiate to third order shows that, at $t=0$, the coefficients of the Taylor series expansion $g^{-1} \tilde{g_t}=\id +t\tilde{h}_1+\tfrac{t^2}{2\!}\tilde{h}_2+\tfrac{t^3}{3!}\tilde{h}_3+o(t^4)$ satisfy
\begin{equation} \label{gauge-3}
\begin{split}
\tilde{h}_1=&h_1+2\delta^{\star_g}X_0\\
\tilde{h}_2=&h_2+2\delta^{\star_g}X_1+D(X_0,h_1)\\
\tilde{h}_3=&h_3+2\delta^{\star_g}X_2+D(X_0,X_1,h_1,h_2)
\end{split}
\end{equation}
where the vector fields $X_p:=X_t^{(p)}(0)$ for $p \geq 0$. The operators above are entirely explicit, although this fact does not play a role in what follows. Indeed 
\begin{equation*}
\begin{split}
D(X_0,h_1)=&g^{-1} \left (2 \L_{X_0}g_1+\L_{X_0}^2g \right )\\
D(X_0,X_1,h_1,h_2)=&g^{-1} \left (3\L_{X_1}g_1+3\L_{X_0}g_2+3\L_{X_0}^2g_1+2\L_{X_0}\L_{X_1}g+\L_{X_0}^3g \right )
\end{split}
\end{equation*}
where the symmetric tensors $g_k:=g_t^{(k)}(0)$ for $k \geq 1$. We also recall the following elementary 
\begin{lema} \label{ll-0}
Assume that $H \in \Gamma \left (\Sym^2TM \right )$. There exists $X \in \Gamma (TM)$ such that $\di^{\star_g}X=0$ and 
$\delta^g(H+\delta^{\star_g}X) \in \mathrm{Im} \di$.
\end{lema}
See e.g. \cite{N1} for proof. These preliminaries allow proving the following normalisation result.
\begin{pro} \label{norm-2nd}
Let $g_t \in \mathscr{M}_1$ be a family of Riemannian metrics with Taylor expansion 
$$g^{-1}g_t=\id+th_1+\tfrac{t^2}{2\!}h_2+\tfrac{t^3}{3!}\!h_3+o(t^4).$$
Up to a gauge transformation i.e. $g_t \mapsto f_t^{\star}g_t$ where $f_t \in \mathbf{G}$ satisfies $f_0=\id$ we may assume that $\delta^gh_1=\grad_g \varphi_1$ as well as 
$$ \mathbf{X}_2(g_t)=\grad_g \varphi_2 \ \mathrm{and} \  \mathbf{X}_3(g_t)=\grad_g \varphi_3
$$
for smooth functions $\varphi_k, 1 \leq k \leq 3$ on $M$.
\end{pro}
\begin{proof}
To shorten notation during this proof we introduce the polynomial operator given by $P(H_1,H_2)=-\tfrac{3}{2}\{H_1,H_2\}-\tfrac{3}{2}H_1^3+\tfrac{3}{8}\tr(H_1^2)H_1$ whenever $H_1,H_2$ are symmetric tensors. We use Lemma \ref{ll-0} for the tensor 
$H_0:=h_1$ in order to obtain a divergence free vector field $X_0$ such that the tensor 
$\tilde{H}_0:=h_1+2\delta^{\star_g}X_0$ in $\Gamma \left (\Sym^2TM \right )$ has exact divergence, that is 
$\delta^g\tilde{H}_0=\grad_g\!\varphi_0$. Next we consider the symmetric tensor
$$H_1:=h_2-\tilde{H}_0^2+D(X_0,h_1)$$ 
and use again Lemma \ref{ll-0} which yields a divergence free vector field 
$X_1$ such that we have $\delta^g(H_1+2\delta^{\star_g}X_1)=\grad_g \varphi_2$. Finally we consider the symmetric tensor 
$$H_2:=h_3+P(\tilde{H}_0,H_1+2\delta^{\star_g}X_1)+D(X_0,X_1,h_1,h_2)$$
and use again Lemma \ref{ll-0} to see that $\delta^g(H_2+2\delta^{\star_g}X_2)=\grad_g \varphi_3$ for some 
divergence free vector field $X_2$.

Now consider the family $X_t=X_0+tX_1+\tfrac{t^2}{2}X_2$ which integrates to a family 
$\varphi_t$ in $\mathbf{G}$ with $f_0=\id$ via $f_t^{\prime}=X_t \circ f_t$. According 
to the construction above the gauge transformation rule in \eqref{gauge-3} shows the coefficients $\tilde{h}_i, 1 \leq i \leq 3$ of the Taylor expansion of $g^{-1}\varphi_t^{\star}g_t$ thus satisfy 
\begin{equation*}
\begin{split}
&\tilde{h}_1=\tilde{H}_0, \ \tilde{h}_2-\tilde{h}_1^2=H_1+2\delta^{\star_g}X_1\\
& \tilde{h}_3+P(\tilde{h}_1,\tilde{h}_2-\tilde{h}_1^2)=H_2+2\delta^{\star_g}X_2.
\end{split}
\end{equation*}
From the definition of the vector fields $\mathbf{X}_i(\tilde{g}_t)$ (see \eqref{norm-3i} ) it follows that 
\begin{equation*}
\begin{split}
\mathbf{X}_2(\tilde{g}_t)&=\grad_g \left ( \varphi_2+\tfrac{1}{4}\tr(\tilde{h}_1^2) \right )\\
\mathbf{X}_3(\tilde{g}_t)&=\grad_g \left ( \varphi_3+\tfrac{3}{8}\tr(\{\tilde{h}_1,\tilde{h}_2- \tilde{h}_1^2 \} \right )
\end{split}
\end{equation*}
and the claim is completely proved.
\end{proof}
We finish this section by explaining the effect of the gauge normalisation which is dictated by the Einstein 
equations as in part (d) in Theorem \ref{main-i1}. In fact this has a direct geometric interpretation in terms 
of the Cayley transform of the family 
$g_t$.
\begin{pro} \label{cayley-1}
Let $g_t \in \mathscr{M}_1$ be a family of Riemannian metrics with Taylor expansion 
$g^{-1}g_t=\id+th_1+\tfrac{t^2}{2!}h_2+\tfrac{t^3}{3!}h_3+o(t^4)$ at $t=0$. Assume that $\delta^gh_1=0$ as well as 
$$ \mathbf{X}_2(g_t)=\mathbf{X}_3(g_t)=0.
$$
Then the Cayley transform $\mathrm{C}_t=(1-h_t)(1+h_t)^{-1}$ of the family 
$g_t$ satisfies 
\begin{equation*}
\delta^g\mathrm{C}_t=-\tfrac{1}{2}(1+\mathrm{C}_t)\di \ln \det(1+\mathrm{C}_t)+o(t^4).
\end{equation*}
\end{pro}
\begin{proof}
Write $C_t=1+t\mathrm{C}_1+\tfrac{t^2}{2!}\mathrm{C}_2+\tfrac{t^3}{3!}\mathrm{C}_3+o(t^4)$ at $t=0$. We also consider the function $\Lambda_t:=\ln \det(1+\mathrm{C}_t)$ together with the shorthand notation $U_t:=(1+\mathrm{C}_t)^{-1}$. Note that the derivatives of $U_t$ are given by 
$$ U_t^{(1)}(0)=-\mathrm{C}_1, \ U_t^{(2)}(0)=-\mathrm{C}_2+2\mathrm{C}_1^2, \ \mathrm{and} \ 
U_t^{(3)}(0)=-\mathrm{C}_3+3\{\mathrm{C}_2,\mathrm{C}_1\}-6\mathrm{C}_1^3.
$$
Using those and the general formula $\Lambda_t^{(1)}=\tr (U_t \circ \mathrm{C}^{(1)}_t)$ yields 
\begin{equation*}
\begin{split}
&\Lambda(0)=0 , \ \Lambda^{(1)}(0)=0\\
& \Lambda^{(2)}(0)=-\tr (\mathrm{C}_1 \circ \mathrm{C}_1) \ \mathrm{and} \ \Lambda^{(3)}(0)=-\tfrac{3}{2}\tr 
\{\mathrm{C}_1,\mathrm{C}_2\}.
\end{split}
\end{equation*} 
Next we use the expression for the Taylor coefficients of the Cayley transform in \eqref{cay-co}; also record that 
since $\delta^gh_1=0$ we must have $\delta^g (\tr(h_1^2)h_1)=-h_1(\grad_g \tr(h_1^2))$. Taking these facts into account 
shows that 
\begin{equation*}
\begin{split}
\mathbf{X}_2(g_t)&=-2\delta^g\mathrm{C}_2+\di \tr (\mathrm{C}_1 \circ \mathrm{C}_1)\\
\mathbf{X}_3(g_t)&=-2\delta^g\mathrm{C}_3+3 \mathrm{C}_1 \di \tr (\mathrm{C}_1 \circ \mathrm{C}_1)+\tfrac{3}{2}
\di \tr \{\mathrm{C}_1, \mathrm{C}_2\}.
\end{split}
\end{equation*} 
By combining these formulas with the 
expression for the derivatives of $\Lambda_t$ obtained above, a straightforward computation shows 
that 
\begin{equation} \label{cay-fin}
\left (\delta^g\mathrm{C}_t+\tfrac{1}{2}(1+\mathrm{C}_t)\di \Lambda_t \right )^{(k)}(0)=-\tfrac{1}{2}\mathbf{X}_k(g_t)
\end{equation}
for $0 \leq k \leq 3$. The claim follows now from the assumption of having $\mathbf{X}_k(g_t)=0$ for $k=2,3$.
\end{proof}
We end this section with the following 
\begin{rema} \label{ggen}
Let $g_t \in \mathscr{M}_1$ be a family of Riemannian metrics with Taylor expansion at $t=0$ given by 
$g^{-1}g_t=\id+th_1+\tfrac{t^2}{2\!}h_2+\tfrac{t^3}{3!}\!h_3+o(t^4)$ and satisfying $\delta^gh_1=0$. Using \eqref{cay-fin} we see that the gauge normalisation in Proposition \ref{norm-2nd} amounts to 
$$\delta^g\mathrm{C}_t=-\tfrac{1}{2}(1+\mathrm{C}_t)\di \ln \det(1+\mathrm{C}_t)+\grad_g \varphi_t+o(t^4)
$$
for some smooth function time dependent $\varphi_t$ on $M$. 
\end{rema}

\section{The quadratic differential operators}
In this section we assume that the Riemannian manifold $(M^n,g)$ is compact and Einstein with Einstein constant $E$. The aim of this section is to derive some properties of the quadratic differential operators which feature in the definition of the operator $\mathbf{w}$ as defined in \eqref{w-def-i}. Those properties will be used systematically in the rest of the paper. We begin by recalling that the operator 
$$ \mathbf{p} : \Gamma \left (\Sym^2TM \right ) \oplus \Gamma \left (\Sym^2TM \right ) \to \Gamma \left (\End( TM )\right )
$$
is defined according to 
\begin{equation*}
g(\mathbf{p}(h_1,h_2)X,Y):=g((\nabla^g_{X}h_1)e_i,2(\nabla^g_{e_i}h_2)Y-(\nabla^g_{Y}h_2)e_i).
\end{equation*}
This operator splits as $\mathbf{p}=2\mathbf{i}-\mathbf{j}$ where 
$ \mathbf{i} : \Gamma \left (\Sym^2TM \right ) \oplus \Gamma \left (\Sym^2TM \right ) \to \Gamma \left (\End( TM )\right )
$
is the quadratic differential operator determined from 
$$ \la \mathbf{i}(h_1,h_2),H\ra_{L^2}:=\la (\nabla^g_{e_k}h_1)e_i,(\nabla^g_{e_i}h_2)He_k\ra_{L^2}$$
whenever $h_1,h_2$ belong to $\Gamma \left (\Sym^2TM \right )$ and $H$ is in $\Gamma \left (\End (TM) \right )$
. The second summand in $\mathbf{i}$ is described by the operator 
$ \mathbf{j} : \Gamma \left (\Sym^2TM \right ) \oplus \Gamma \left (\Sym^2TM \right ) \to 
\Gamma \left (\End (TM)\right )
$ given by 
\begin{equation} \label{j-d}
g(\mathbf{j}(h_1,h_2)X,Y):=g(\nabla^g_{X}h_1,\nabla^g_{Y}h_2).
\end{equation}
The operator $\mathbf{p}$ appears naturally when studying third order Einstein deformations. The aim in this section is to prove a few facts regarding this operator which will be needed in the subsequent. In doing so we also need to consider the operator 
$$ (h_1,h_2) \in \Gamma \left (\Sym^2TM \right ) \oplus \Gamma \left (\Sym^2TM \right )\mapsto \mathrm{L}(h_1,h_2):=\tfrac{1}{2}\{ \nabla^g_{e_i}h_1, \nabla^g_{e_i}h_2\} \in \Gamma \left (\Sym^2TM \right )
$$
which clearly satisfies $\mathrm{L}(h_1,h_2)=\mathrm{L}(h_2,h_1)$. We first observe that  
\begin{lema} \label{quad-1L}
Assume that the tensors $h,H$ belong to $\Gamma \left (\Sym^2TM \right )$. Then 
$$ g(\mathbf{p}(h,h),H)=-g(H \sharp \di_{\nabla^g}\!h, \di_{\nabla^g}\!h)+g(\mathrm{L}(h,h),H).
$$
\end{lema}
\begin{proof}
We compute 
\begin{equation*}
\begin{split}
g\left ((\di_{\nabla^g}\!h(X,e_i), \di_{\nabla^g}\!h(Y,e_i) \right )=&g\left ((\nabla^g_Xh)e_i-(\nabla^g_{e_i}h)X, (\nabla^g_Yh)e_i-(\nabla^g_{e_i}h)Y\right )\\
=&g((\nabla^g_Xh)e_i,(\nabla^g_Yh)e_i)+g(\mathrm{L}(h,h)X,Y)\\
&-g((\nabla^g_Xh)e_i,(\nabla^g_{e_i}h)Y)-g((\nabla^g_Yh)e_i,(\nabla^g_{e_i}h)X).
\end{split}
\end{equation*}
Since $H$ is symmetric it follows that 
$$ g\left (\di_{\nabla^g}\!h(e_k,e_i), \di_{\nabla^g}\!h(He_k,e_i) \right )=-g(\mathbf{p}(h,h),H)+g(\mathrm{L}(h,h),H).
$$
At the same time 
$$ g\left (\di_{\nabla^g}\!h(e_k,e_i), \di_{\nabla^g}\!h(He_k,e_i) \right )=
\tfrac{1}{2}g\left (\di_{\nabla^g}\!h(e_k,e_i), (H \sharp \di_{\nabla^g}\!h)(e_k,e_i) \right )=
g(H \sharp \di_{\nabla^g}\!h, \di_{\nabla^g}\!h)$$ 
and the claim is proved.
\end{proof}
After polarisation in Lemma \ref{quad-1L} we get 
\begin{equation} \label{pol}
g\left (\mathbf{p}(h_1,h_2)+\mathbf{p}(h_2,h_1),H \right )=-2g(H \sharp \di_{\nabla^g}\!h_1,\di_{\nabla^g}\!h_2)+2g(\mathrm{L}(h_1,h_2),H)
\end{equation}
whenever $h_1,h_2,H$ belong to $\Gamma \left (\Sym^2TM \right )$. The next aim is to gain further insight in the structure of the operator 
$\mathbf{i}$ to which extent  we also define the auxiliary operator 
$$ \mathbf{q}(H_1,H_2)=\la (\nabla^g_{e_j}H_1)e_i, (\nabla^g_{e_i}H_2)e_j\ra_{L^2}
$$
whenever $H_1,H_2$ are sections of $\End(TM)$.
\begin{lema} \label{i-1}
We have 
\begin{equation*}
\la \mathbf{i}(h_1,h_2),H\ra_{L^2}=\la \delta^{g}(h_2 \circ \di_{\nabla^g}\!h_1),H\ra_{L^2}+\mathbf{q}(h_1, h_2 \circ H)-\la 
h_2, \mathrm{L}(H,h_1)
\ra_{L^2}
\end{equation*}
whenever $h_1,h_2,H$ are sections of $\Sym^2TM$. 
\end{lema}
\begin{proof}
We use first the Leibniz rule to arrive at 
\begin{equation*}
\begin{split}
\la \mathbf{i}(h_1,h_2),H\ra_{L^2}=&\la (\nabla^g_{e_k}h_1)e_i, \left (\nabla^g_{e_i}(h_2\circ H))e_k-h_2(\nabla^g_{e_i}H \right )e_k \ra_{L^2}\\
=&\mathbf{q}(h_1, h_2 \circ H)-\la \left ( \nabla^g_{e_k}h_1)e_i, h_2(\nabla^g_{e_i}H \right ) e_k \ra_{L^2}.
\end{split}
\end{equation*}
At the same time 
\begin{equation*}
\begin{split}
\la \left ( \nabla^g_{e_k}h_1)e_i, h_2(\nabla^g_{e_i}H \right ) e_k \ra_{L^2}=&
\la \di_{\nabla^g}\!h_1(e_k,e_i)+ (\nabla^g_{e_i}h_1)e_k,  h_2(\nabla^g_{e_i}H ) e_k \ra_{L^2}\\
=&\tfrac{1}{2}\la \di_{\nabla^g}\!h_1(e_k,e_i),h_2(\di_{\nabla^g}\! H )(e_i, e_k) \ra_{L^2}+\la \nabla^g_{e_i}h_1,  h_2 \circ \nabla^g_{e_i}H   \ra_{L^2}\\
=&-\la \di_{\nabla^g}\!h_1, h_2 \circ \di_{\nabla^g}\! H  \ra_{L^2}+\la \nabla^g_{e_i}h_1,  h_2 \circ \nabla^g_{e_i}H   \ra_{L^2}.
\end{split}
\end{equation*}
Since we are dealing with symmetric operators 
$$\la \nabla^g_{e_i}h_1,  h_2 \circ \nabla^g_{e_i}H   \ra_{L^2}=
\la h_2, \nabla^g_{e_i}h_1 \circ \nabla^g_{e_i}H   \ra_{L^2}=\tfrac{1}{2}\la h_2, \{\nabla^g_{e_i}h_1, \nabla^g_{e_i}H\}   \ra_{L^2}
=\la h_2, \mathrm{L}(h_1,H) \ra_{L^2}.$$ The claim follows now by gathering terms.
\end{proof}
Next we derive the explicit expression for $\mathbf{q}$ and a few of its properties.

\begin{lema} \label{q}
Assume that $H_1,H_2$ belong to $\Gamma \left ( \End(TM)\right )$. We have 
$$ \mathbf{q}(H_1,H_2)=\la \ring{R}H_1-EH_1,H_2 \ra_{L^2}+\la \delta^gH_1, \delta^gH_2 \ra_{L^2}.
$$
In particular $\mathbf{q}(H_1,H_2)=\mathbf{q}(H_2,H_1)$ whenever the tensor $H_2$ is symmetric.
\end{lema}
\begin{proof}
We compute, by integration by parts,
\begin{equation*}
\begin{split}
 \mathbf{q}(H_1,H_2)=&\la (\nabla^g_{e_j}H_1)e_i, (\nabla^g_{e_i}H_2)e_j\ra_{L^2}=-\la ((\nabla^g)^2_{e_i,e_j}H_1)e_i,H_2e_j\ra_{L^2}\\
 =&-\la ((\nabla^g)^2_{e_i,e_j}H_1)e_i-((\nabla^g)^2_{e_j,e_i}H_1)e_i,H_2e_j\ra_{L^2}-\la ((\nabla^g)^2_{e_j,e_i}H_1)e_i,H_2e_j \ra_{L^2}\\
 =&\la [R^g(e_i,e_j),H_1]e_i, H_2e_j \ra_{L^2}+\la \nabla^g_{e_j} \delta^gH_1, H_2e_j \ra_{L^2}\\
 =&\la \ring{R}H_1-EH_1,H_2 \ra_{L^2}-\la \delta^gH_1, (\nabla^g_{e_j}H_2)e_j \ra_{L^2}\\
 =&\la \ring{R}H_1-EH_1,H_2 \ra_{L^2}+\la \delta^gH_1, \delta^gH_2 \ra_{L^2}.
 \end{split}
\end{equation*}
Note that we have also used the Ricci identity to obtain the curvature terms.
\end{proof}
We also record the explicit expression for $\mathrm{L}(h_1,h_2)$ in the Lemma below.
\begin{lema} \label{L-exp}
Assume that $h_1,h_2$ belong to $\Gamma \left (\Sym^2TM \right )$. We have 
$$ 4\mathrm{L}(h_1,h_2)=\{(\Delta_E+2\ring{R})h_1,h_2\}+\{h_1,(\Delta_E+2\ring{R})h_2\}-(\Delta_E+2\ring{R})\{h_1,h_2\}
$$
\end{lema}
\begin{proof}
We have $ 2\mathrm{L}(h,h)=\{(\Delta_E+2\ring{R})h,h\}-(\Delta_E+2\ring{R})h^2$ according to \cite{N1}[Lemma 4.15]. Since $\mathrm{L}(h_1,h_2)=\mathrm{L}(h_2,h_1)$ the claim follows now by polarisation.
\end{proof}
Next we determine the skew-symmetric component in the operator $\mathbf{p}(h,h)$ for elements $h$ in the infinitesimal deformation space of $g$ as follows; in the proposition below we use the inner product on $\End(TM)$.
\begin{pro} \label{p-skew}
Assume that the tensor $h$ belongs to $\mathscr{E}(M,g)$ and that $A$ belongs to $\Omega^2M$. Then 
\begin{equation*}
\la \mathbf{p}(h,h),A\ra_{L^2}=2\la \delta^g[h,h]^{\FN},A \ra_{L^2}.
\end{equation*}
\end{pro}
\begin{proof}
Exactly as in the proof of Lemma \ref{i-1} we find 
$$ \la \mathbf{i}(h,h),A \ra_{L^2}=\la \delta^g(h \circ \di_{\nabla^g}\!h),A \ra_{L^2}+\mathbf{q}(h, h \circ A)-\la \nabla^g_{e_i}h, 
h \circ \nabla^g_{e_i}A \ra_{L^2}.
$$
However, since $h$ is symmetric 
\begin{equation*}
\begin{split}
\la \nabla^g_{e_i}h, 
h \circ \nabla^g_{e_i}A \ra_{L^2}=&\la h \circ \nabla^g_{e_i}h, \nabla^g_{e_i}A \ra_{L^2}=
\tfrac{1}{2} \la [h, \nabla^g_{e_i}h], \nabla^g_{e_i}A 
\ra_{L^2}\\
=&-\tfrac{1}{2} \la [h, (\nabla^g)^2_{e_i,e_i}h], A \ra_{L^2}=\tfrac{1}{2} \la [h, \nabla^{\star_g} \nabla^g h], A \ra_{L^2}
\end{split}
\end{equation*}
after integrating by parts. Since $\Delta_Eh=(\nabla^{\star_g} \nabla^g-2\ring{R})h=0$ it follows that 
\begin{equation*}
\begin{split}
\la \nabla^g_{e_i}h, 
h \circ \nabla^g_{e_i}A \ra_{L^2}= \la [h, \ring{R}h], A \ra_{L^2}.
\end{split}
\end{equation*}
Further on, since $A$ is skew-symmetric, part (i) in Lemma \ref{sym-bra} shows that 
$$ \la \delta^g(h \circ \di_{\nabla^g}\!h),A \ra_{L^2}=\la \delta^{g}[h,h]^{\FN}-\ring{R}h \circ h,A\ra_{L^2}=
\la \delta^{g}[h,h]^{\FN}+\tfrac{1}{2}[h,\ring{R}h],A\ra_{L^2}.
$$
Finally, since $h$ is divergence free, Lemma \ref{q} ensures that 
$$\mathbf{q}(h,h \circ A)=
\la h \circ \ring{R}h-Eh^2,A\ra_{L^2}=\la h \circ \ring{R}h,A\ra_{L^2}=\tfrac{1}{2}\la [h, \ring{R}h],A \ra_{L^2}.$$ The claim follows by gathering these facts and also 
using that $\la \mathbf{p}(h,h),A\ra_{L^2}=2\la \mathbf{i}(h,h),A\ra_{L^2}$. Note that the latter equality 
stems from having $\mathbf{p}=2\mathbf{i}-\mathbf{j}$ and $\la \mathbf{j}(h,h),A\ra_{L^2}=0$.
\end{proof}
We resume in this section by proving the following general cubic identity which generalises Lemma \ref{sym-bra}.
\begin{pro} \label{o3-p}
We have 
\begin{equation*}
\begin{split}
\la h_1 \sharp \di_{\nabla^g}\!h_2,\di_{\nabla^g}\!h_3\ra_{L^2}=&\la \nabla^g_{h_1e_i}h_2,\nabla^g_{e_i}h_3\ra_{L^2}-\la \di_{\nabla^g}\!h_1,h_2 \circ \di_{\nabla^g}\!h_3+
h_3 \circ \di_{\nabla^g}\!h_2 \ra_{L^2}\\
+&\la h_1, \mathrm{L}(h_2,h_3) \ra_{L^2}+\la h_2, \mathrm{L}(h_1,h_3) \ra_{L^2}+\la h_3, \mathrm{L}(h_1,h_2) \ra_{L^2}\\
-&\mathbf{q}(h_2 \circ h_1,h_3)-\mathbf{q}(h_3 \circ h_1,h_2)
\end{split}
\end{equation*}
for all $h_i \in \Gamma \left (\Sym^2TM \right )$, where $1 \leq i \leq 3$.
\end{pro}
\begin{proof}
According to the definition  of the action $\alpha \in \Omega^2(M,TM) \mapsto h_1 \sharp \alpha$ we have 
\begin{equation*}
\begin{split}
\la h_1 \sharp \di_{\nabla^g}\!h_2,\di_{\nabla^g}\!h_3\ra_{L^2}=&\tfrac{1}{2}\la 
\di_{\nabla^g}\!h_2(h_1e_i,e_j)+ \di_{\nabla^g}h_2(e_i,h_1e_j),\di_{\nabla^g}\!h_3(e_i,e_j)\ra_{L^2}\\
=&
\la \di_{\nabla^g}\!h_2(h_1e_i,e_j)+ \di_{\nabla^g}\!h_2(e_i,h_1e_j),(\nabla^g_{e_i}h_3)e_j \ra_{L^2}.
\end{split}
\end{equation*}
Expanding $\di_{\nabla^g}\!h_2$ leads to 
\begin{equation*}
\begin{split}
\la h_1 \sharp \di_{\nabla^g}\!h_2,\di_{\nabla^g}\!h_3\ra_{L^2}=&\la \nabla^g_{h_1e_i}h_2,\nabla^g_{e_i}h_3\ra_{L^2}+\la (\nabla^g_{e_i}h_2) \circ h_1,
(\nabla^g_{e_i}h_3)\ra_{L^2}\\
-& \la (\nabla^g_{e_j}h_2)h_1e_i, (\nabla^g_{e_i}h_3)e_j\ra_{L^2}-\la (\nabla^g_{e_j}h_2)e_i, (\nabla^g_{e_i}h_3)h_1e_j\ra_{L^2}.
\end{split}
\end{equation*}
Because we are dealing with symmetric tensors we clearly have 
$$\la (\nabla^g_{e_i}h_2) \circ h_1,
(\nabla^g_{e_i}h_3)\ra_{L^2}=\la h_1,\mathrm{L}(h_2,h_3)\ra_{L^2}.$$
Further on, using the product rule shows that 
\begin{equation*}
\begin{split}
&\la (\nabla^g_{e_j}h_2)h_1e_i, (\nabla^g_{e_i}h_3)e_j\ra_{L^2}
=\la (\nabla^g_{e_j}(h_2h_1))e_i-h_2(\nabla^g_{e_j}h_1)e_i,(\nabla^g_{e_i}h_3)e_j\ra_{L^2}\\
&=\mathbf{q}(h_2\circ h_1,h_3)-\la (\nabla^g_{e_j}h_1)e_i,h_2(\nabla^g_{e_i}h_3)e_j\ra_{L^2}\\
&=\mathbf{q}(h_2 \circ h_1,h_3)-\la \di_{\nabla^g}\!h_1(e_j,e_i)+(\nabla^g_{e_i}h_1)e_j,h_2(\nabla^g_{e_i}h_3)e_j\ra_{L^2}\\
&=\mathbf{q}(h_2 \circ h_1,h_3)-\la \di_{\nabla^g}\!h_1(e_j,e_i),h_2(\nabla^g_{e_i}h_3)e_j\ra_{L^2}-\la h_2, \mathrm{L}(h_1,h_3) \ra_{L^2}\\
&=\mathbf{q}(h_2 \circ h_1,h_3)+\la \di_{\nabla^g}\!h_1,h_2 \circ \di_{\nabla^g}\!h_3 \ra_{L^2}-\la h_2, \mathrm{L}(h_1,h_3) \ra_{L^2}.
\end{split}
\end{equation*}
Note that we have taken into account that 
$$\la (\di_{\nabla^g}\!h_1(e_j,e_i),h_2(\nabla^g_{e_i}h_3)e_j\ra_{L^2}=\tfrac{1}{2}
\la (\di_{\nabla^g}\!h_1(e_j,e_i),h_2(\di_{\nabla^g}h_3)(e_i,e_j)\ra_{L^2}=
-\la \di_{\nabla^g}\!h_1,h_2 \circ \di_{\nabla^g}\!h_3 \ra_{L^2}.$$

After operating the variable change $(h_2,h_3) \mapsto (h_3,h_2)$ in the above formula, we also get 
\begin{equation*}
\begin{split}
\la (\nabla^g_{e_i}h_3)h_1e_j, (\nabla^g_{e_j}h_2)e_i\ra_{L^2}
=&\mathbf{q}(h_3 \circ h_1,h_2)+\la \di_{\nabla^g}\!h_1,h_3 \circ \di_{\nabla^g}\!h_2 \ra_{L^2}-\la h_3, \mathrm{L}(h_1,h_2) \ra_{L^2}.
\end{split}
\end{equation*}
The claim follows now by gathering terms.
\end{proof}
Note that part (ii) in Lemma \ref{sym-bra} can be re-derived from the above proposition by taking $h_1=h_2$ and using the expressions for the operators $\mathbf{q}$ respectively $\mathrm{L}$ obtained in Lemma \ref{q} respectively 
Lemma \ref{L-exp}. To finish this section we prove yet another set of identities which will be useful in section \ref{3-rd}. 

\begin{lema} \label{diff-Ls}
Assume that the tensors $h,H$ belong to $\Gamma \left (\Sym^{2}TM \right )$. The following hold
\begin{itemize}
\item[(i)] we have 
\begin{equation*}
\la 2\delta^g[h,h ]^{\FN}+\mathbf{p}(h,h),H\ra_{L^2}=\la \bfv(h,h)+
(\mathrm{L}-\mathrm{\tilde{L}})(h,h),H \ra_{L^2}
\end{equation*}
\item[(ii)] the operator $\mathrm{L}-\mathrm{\tilde{L}}$ is determined from 
\begin{equation*}
(\mathrm{L}-\mathrm{\tilde{L}})(h,h)=\{h,\mathbf{L}h\}-\mathbf{L}h^2
\end{equation*}
where $\mathbf{L}=-\tfrac{1}{2}\left (\Delta_E-2\delta^{\star_g}\delta^g+2E \right )+\tfrac{1}{2}\di \delta^g$.
\end{itemize}
\end{lema}
\begin{proof}
(i) Follows by combining Lemma \ref{quad-1L} and \eqref{bfv-1}.\\ 
(ii) Using Lemma \ref{L-exp} for the expression of $\mathrm{L}$ whilst also taking into account 
the definition of $\mathrm{\tilde{L}}$ from \eqref{L-t} yields 
$(\mathrm{L}-\mathrm{\tilde{L}})(h,h)=\{h,\mathbf{L}h\}-\mathbf{L}h^2$ where 
$\mathbf{L}=\tfrac{1}{2}(\Delta_E+2\ring{R})-\delta^g\di_{\nabla^g}$. The claim follows from the Weitzenb\"ock formula \eqref{wz1}.
\end{proof}

\section{The explicit form of the Einstein equation } \label{E-direct}
Let $(M^n,g)$ be an Einstein manifold with Einstein constant $E$. We consider a family of metrics $g_t:=g(h_t \cdot ,\cdot)$ where 
$h_t$ is symmetric and satisfies $h_0=\id$. We furthermore require that $g_t \in \mathscr{M}_1$, that is $\det h_t=1$.
Direct computation based on Koszul's formula shows that the Levi-Civita connection $\nabla^t$ of $g_t$ satisfies 
\begin{equation*}
\nabla^t_XY=\nabla^g_XY+\frac{1}{2}\eta_XY
\end{equation*}
where the tensor $\eta :TM \to \End(TM)$, denoted with $X \mapsto \eta_X$, is explicitly determined from 
\begin{equation} \label{eta1}
g_t(\eta_XY,Z)=g((\nabla^g_Xh_t)Y,Z)+g(X,(\nabla^g_Yh_t)Z-(\nabla^g_Zh_t)Y).
\end{equation}
Note that $\eta_XY=\eta_YX$ so the above formula can also be proved by checking that $\nabla^t$ preserves $g_t$ and is torsion free. The $(3,1)$ curvature tensor of $g_t$ is determined from 
\begin{equation} \label{curv-comp}
R^t(X,Y)Z=R^g(X,Y)Z-\frac{1}{2}\bigl ((\nabla^g_X\eta)_YZ-(\nabla^g_Y\eta)_XZ \bigr )-\frac{1}{4}[\eta_X,\eta_Y]Z.
\end{equation} 
The Einstein equation for the family $g_t$ reads 
$R^{t}(E_i,X)E_i=EX$ where $E_i$ is an ON basis w.r.t. $g_t$. However any such basis is of the form 
$E_i=h_t^{-\frac{1}{2}}e_i$ where $e_i$ is orthonormal w.r.t. $g$; since $h_t$ is symmetric w.r.t. 
$g$ it follows that  
\begin{equation} \label{eins-main}
\Ric^{g_t}X=R^t(e_i,X)h_t^{-1}e_i
\end{equation}
for some local orthonormal frame w.r.t. $g$. Since this equation is of trace type w.r.t. $g$ it is basis independent. The aim is to differentiate this equation up to third order, 
at $t=0$. 

First we compute directly the Ricci tensor of $g_t$; directly from \eqref{curv-comp} we derive
\begin{equation} \label{i1}
\begin{split}
R^t(e_i,X)h_t^{-1}e_i=&R^g(h_t^{-1})X-\tfrac{1}{2}(\nabla^g_{e_i}\eta)_Xh_t^{-1}e_i-\tfrac{1}{4} \left (2\eta_{e_i}(\nabla^g_Xh_t^{-1})e_i
+\eta_{e_i}\eta_Xh_t^{-1}e_i \right )\\
&+\tfrac{1}{2}\nabla^g_XA_t+\tfrac{1}{4}\eta_X A_t\\
\end{split}
\end{equation}
where the $1$-form $A_t:=\eta_{e_i}h_t^{-1}e_i$. The next Lemma shows that the third summand above is determined 
algebraically from $\di_{\nabla^g}h_t$ and $\di_{\nabla^g}h_t^{-1}$ by means of the quadratic differential operator 
$\mathbf{p}$ introduced in the previous section. In this paper we will systematically use the quantity $\mathbf{D}_t:=\di_{\nabla^g}\!h_t(h_t^{-1} \cdot, h_t^{-1} \cdot) \in \Omega^2(M,TM)$.
\begin{lema} \label{L1}
We have 
\begin{equation*}
\begin{split}
g_t(2\eta_{e_i}(\nabla_X^gh_t^{-1})e_i+\eta_{e_i}\eta_Xh_t^{-1}e_i,Y)=&g (\mathbf{p}(h_t^{-1},h_t)X,Y).
\end{split}
\end{equation*}
\end{lema}
\begin{proof}
According to \eqref{eta1} we have 
$$ g_t(\eta_Xh_t^{-1}e_i,Y)=g((\nabla_X^gh_t)h_t^{-1}e_i,Y)+g(X,\di_{\nabla^g} \!h_t(h_t^{-1}e_i,Y)).
$$
Substituting $Y \mapsto h_t^{-1}Y$ thus leads to 
\begin{equation} \label{quad-1}
g(\eta_Xh_t^{-1}e_i,Y)=-g((\nabla_X^gh_t^{-1})e_i,Y)+g(X,\mathbf{D}_t(e_i,Y)).
\end{equation}
Next, using again \eqref{eta1} shows that 
\begin{equation*}
\begin{split}
g_t(\eta_{e_i}\eta_Xh_t^{-1}e_i,Y)=&g(\eta_Xh_t^{-1}e_i,(\nabla^g_{e_i}h_t)Y)+g(e_i, \di_{\nabla^g}\!h_t(e_j,Y))g(\eta_Xh_t^{-1}e_i,e_j).
\end{split}
\end{equation*}
After replacing the expression for $\eta_Xh_t^{-1}e_i$ found in \eqref{quad-1} above  we arrive at 
\begin{equation*}
\begin{split}
g_t(\eta_{e_i}\eta_Xh_t^{-1}e_i,Y)=&-g((\nabla
^g_{X}h_t^{-1})e_i,(\nabla^g_{e_i}h_t)Y)+g(X,\mathbf{D}_t(e_i,(\nabla^g_{e_i}h_t)Y)
)\\
&+g(e_i, \di_{\nabla^g}\!h_t(e_j,Y)) \left (-g((\nabla_{e_i}h_t^{-1})e_i,e_j)+g(X, \mathbf{D}_t(e_i,e_j)) \right )\\
=&-g((\nabla^g_{X}h_t^{-1})e_i,(\nabla^g_{e_i}h_t)Y)\\
&+g(X, \mathbf{D}_t(e_i,(\nabla^g_{e_i}h_t)Y))+g(X, \mathbf{D}_t(\di_{\nabla^g}\!h_t(e_j,Y),e_j))\\
&-g((\nabla^g_Xh_t^{-1})e_j,\di_{\nabla^g}\!h_t(e_j,Y)).
\end{split}
\end{equation*}
Because $h_t$ is symmetric and $\mathbf{D}_t$ is skew-symmetric in the first two slots the second line above vanishes identically. Thus 
\begin{equation*}
\begin{split}
g_t(\eta_{e_i}\eta_Xh_t^{-1}e_i,Y)
=-g((\nabla^g_{X}h_t^{-1})e_i,(\nabla^g_{e_i}h_t)Y)-g((\nabla^g_Xh_t^{-1})e_j,\di_{\nabla^g}\!h_t(e_j,Y)).
\end{split}
\end{equation*}
At the same time, using again \eqref{eta1} reveals that 
\begin{equation*}
\begin{split}
2g_t(\eta_{e_i}(\nabla_{X}h_t^{-1})e_i,Y)=&2g((\nabla^g_{e_i}h_t)(\nabla^g_{X}h^{-1})e_i,Y)+2g(e_i, \di_{\nabla^g}\!h_t((\nabla^g_Xh_t^{-1})e_i,Y))\\
=&2g((\nabla^g_{X}h_t^{-1})e_i,(\nabla^g_{e_i}h_t)Y)+2g((\nabla^g_Xh_t^{-1})e_j,\di_{\nabla^g}\!h_t(e_j,Y))
\end{split}
\end{equation*}
after taking into account that $h$ is symmetric. The claim follows now by collecting terms and using the definition of 
$\mathbf{p}$.
\end{proof}

We also observe that 
\begin{lema} \label{L2}
We have 
\begin{equation*}
\begin{split}
g_t \left ((\nabla_{e_i}^g\eta)_Xh_t^{-1}e_i,Y \right )=&g_t(\eta_X \delta^gh_t^{-1}+\nabla_X^g \delta^gh_t^{-1},Y)\\
+&g(X,(R^g(h_t^{-1}) \circ h_t)Y)-(\delta^g \mathbf{D}_t)_{h_tY}-EY).
\end{split}
\end{equation*}
\end{lema}
\begin{proof}
Differentiating \eqref{quad-1} in direction of $e_i$ shows that 
\begin{equation*}
\begin{split}
g\left ((\nabla_{e_i}^g\eta)_Xh_t^{-1}e_i,Y \right )=g(\eta_X \delta^gh_t^{-1},Y)-g(\left ((\nabla^g)^2_{e_i,X}h_t^{-1} \right )e_i,Y)-g(X, 
(\delta^g\mathbf{D}_t)_Y).
\end{split}
\end{equation*}
The Ricci identity leads to 
$$-\left ((\nabla^g)^2_{e_i,X}h_t^{-1}\right )e_i=[R^{g}(e_i,X),h_t^{-1}]e_i-((\nabla^g)^2_{X,e_i}h_t^{-1})e_i=(\ring{R}h_t^{-1})X-Eh_t^{-1}X+\nabla^g_X 
\delta^gh_t^{-1}.$$ 
The claim follows by collecting terms and substituting $Y \mapsto h_tY$. 
\end{proof}
To obtain the final final form of the Einstein equation we collect the terms depending on $A_t$ and $\delta^g h_t^{-1}$
which appear in \eqref{i1} in the differential operator given by 
$$ g(M_tX,Y)=g_t(\nabla^g_X \left (\delta^gh^{-1}_t-A_t \right )+\eta_X \left (\delta^gh_t^{-1}-\tfrac{1}{2}A_t \right ),Y).
$$
Below we work out the explicit expression for $M_t$ by taking account the assumption that the curve $g_t$ has fixed volume.
\begin{lema} \label{Mt}
The following hold 
\begin{itemize}
\item[(i)] we have $A_t=2\delta^gh_t^{-1}$
\item[(ii)] the tensor $M_t$ satisfies 
$ -g(M_tX,h_t^{-1}Y)=g(\nabla^g_X(\delta^gh_t^{-1}),Y)
$
for all $X,Y$ in $TM$.
\end{itemize}
\end{lema}
\begin{proof}
(i) Taking into account \eqref{eta1} shows that 
$$ g_t(A,X)=2g((\nabla^g_{e_i}h)h_t^{-1}e_i,X)-\tr (h_t^{-1} \circ \nabla^g_X h_t).
$$
having $\vol(g_t)$ constant amounts to $\det(h_t)=1$. Differentiating in direction of $X$ shows that 
$\tr (h_t^{-1} \circ \nabla^g_X h_t)=0$. The claim follows by operating the variable change $X \mapsto h_t^{-1}X$.\\
(ii) according to (i), we have 
$-g(M_tX,Y)=g_t(\nabla^g_X(\delta^gh_t^{-1}),Y)$ hence the claim follows by operating the substitution $Y \mapsto h_t^{-1}Y$.
\end{proof}
Summarising, we can express the Einstein equation for $g_t$ as a non-linear divergence type equation 
in terms of the tensor $h_t$ only. 
\begin{pro} \label{eqn-1}
The metric $g_t \in \mathscr{M}_1$ is Einstein with Einstein constant $E$ if and only if we have 
\begin{equation*} 
\begin{split}
&g(X, (\delta^g\mathbf{D}_t+\ring{R}h_t^{-1}+Eh_t^{-1})Y)+g(\nabla^g_X(\delta^gh_t^{-1}),Y)-\tfrac{1}{2}g(\mathbf{p}(h_t^{-1},h_t)X,h_t^{-1}Y)
=2Eg(X,Y)
\end{split}
\end{equation*}
whenever $X,Y$ belong to $TM$.
\end{pro}
\begin{proof}
The Einstein equation for $g_t$ reads $g_{h_t}(R^{t}(e_i,X)h_t^{-1}e_i,Y)=Eg(h_tX,Y)$. Thus using Lemma \ref{L1} and Lemma \ref{L2} together with \eqref{i1} leads to 
\begin{equation*} \label{prel-E}
\begin{split}
&g(X, (\delta^g\mathbf{D_t})_{h_tY}+R^g(h_t^{-1})h_tY)-\tfrac{1}{2}g(\mathbf{p}(h_t^{-1},h_t)X,Y)-g(M_tX,Y)\\
=&2Eg(h_tX,Y)-Eg(X,Y)
\end{split}
\end{equation*}
after re-arranging terms. Operating the variable change $Y \mapsto h_t^{-1}Y$ and taking into account Lemma \ref{Mt} then yields the claim.
\end{proof}
This equation has two parts, according to the vector bundle splitting 
$$\End(TM)=\Sym^2TM \oplus \Lambda^2M.$$
The leading part in the Einstein equation is its symmetric component which reads as indicated below. 
\begin{teo} \label{E-sym}
Let $g_t \in \mathscr{M}_1$ be a curve of Riemannian metrics and let $h_t:=g^{-1}g_t$. Then $g_t$ is Einstein with Einstein constant $E$ if and only if we have 
\begin{equation} \label{mainES}
g\left (\delta^g \mathbf{D}_t+\ring{R}h_t^{-1}+\delta^{\star_g}\delta^gh_t^{-1}-\tfrac{1}{2} h_t^{-1} \circ \mathbf{p}(h_t^{-1},h_t), H\right )=
2E\tr(H)-Eg(h_t^{-1},H)
\end{equation}
for all $H \in \Gamma \left (\Sym^2TM \right )$ and also 
\begin{equation} \label{mainESkew}
g\left (\delta^g \mathbf{D}_t+\tfrac{1}{2} h_t^{-1} \circ \mathbf{p}(h_t^{-1},h_t)-\tfrac{1}{2}\di \delta^gh_t^{-1}, A \right )=0
\end{equation}
for all $A \in \Omega^2M$.
\end{teo}
\begin{proof}
We take $X=He_i$ respectively $X=Ae_i$ and $Y=e_i$ in Proposition \ref{eqn-1} and sum over $i$.
Since the tensor $\ring{R}h_t^{-1}$ is symmetric with respect to $g$ the claim follows from having $\nabla^g=\delta^{\star_g}+\tfrac{1}{2}\di$ on $1$-forms.
\end{proof}
Thus, a priori, the Einstein equation for $g_t$ also has skew-symmetric component with respect to the metric $g$. 
Below we show that the skew-symmetric component to order $p$ is actually determined by the lower order derivatives 
in $\Ric^{g_t}$; see also \cite{NS-E}[Remark 3.8]. As a consequence the Einstein equations reduce to equations on symmetric tensors, as illustrated below.
\begin{pro}\label{red-toS}
Let $g_t \in \mathscr{M}_1$ be a curve of Riemannian metrics and assume that $g_0=g$ is Einstein with Einstein constant $E$. Then $(\Ric^{g_t})^{(k)}(0)=0$ for $1 \leq k \leq 3$ if and only if 
\begin{equation*}
g((\Ric^{g_t})^{(k)}(0),H)=0 \ \mathrm{for \ all } \ H \in \Gamma \left ( \Sym^2TM \right ).
\end{equation*}
\end{pro}
\begin{proof}
We shall use the shorthand notation $\Ric_k:=(\Ric^{g_t})^{(k)}(0) \in \Gamma \left (\End(TM) \right )$. 
Differentiate with respect to $t$ in the identity $g_t(\Ric^{g_t}X,Y)=g_t(\Ric^{g_t}Y,X)$, where 
$X,Y$ belong to $TM$. It follows that 
\begin{equation*}
\begin{split}
g_1(\Ric^gX,Y)+g(\Ric_1X,Y)=&h_1(X,Y)\\
g_2(\Ric^gX,Y)+2g_1(\Ric_1X,Y)+g(\Ric_2X,Y)=&h_2(X,Y)\\
g_3(\Ric^gX,Y)+3g_2(\Ric_1X,Y)+3g_1(\Ric_2X,Y)+g(\Ric_3X,Y)=&h_3(X,Y)
\end{split}
\end{equation*}
where the bilinear forms $h_1,h_2,h_3$ are symmetric. Since $g$ is Einstein it follows that 
$\Ric_1$ is symmetric with respect to $g$. In addition, if $\Ric_1=0$ the above formulas 
show that $\Ric_2$ is symmetric with respect to $g$. Finally, the vanishing of $\Ric_1$ and $\Ric_2$ ensures 
that $\Ric_3$ is symmetric with respect to $g$ and the claim is proved.
\end{proof}
\subsection{The third order variation of the Ricci tensor } \label{3-rd}
Whenever $A=A(t),B=B(t)$ are time dependent sections of $\End(TM)$ we record that the time derivative 
of the product $A \circ B$ satisfies 
\begin{equation} \label{prod-f}
\begin{split}
(A \circ B)^{(2)}(t)=&A^{(2)}(t) \circ B(t)+2A^{(1)}(t) \circ B^{(1)}(t)+A(t) \circ B^{(2)}(t)\\
(A \circ B)^{(3)}(t)=&A^{(3)}(t) \circ B(t)+3A^{(2)}(t) \circ B^{(1)}(t)+3A^{(1)}(t) \circ B^{(2)}(t)+A(t) \circ B^{(3)}(t)
\end{split}
\end{equation}
Now consider a family of Einstein metrics $g_t:=g(h_t \cdot, \cdot)$ in $\mathscr{M}_1$ with Einstein constant $E$ and $g_0=g$; we further indicate with $h_k:=h_t^{(k)}(0)$ in $\Sym^2TM$ for $k \geq 1$ the coefficients of the Taylor expansion of $h_t$ at $t=0$. Similarly we let $I_t:=h_t^{-1}$  and denote $I_k:=I_t^{(k)}(0)$ in $\Sym^2TM$ for $k \geq 1$. 

Using \eqref{prod-f} shows that 
\begin{equation}\label{der-hinv}
\begin{split}
&I_1=-h_1\\
&I_2=-h_2+2h_1^2\\
&I_3=-h_3+3\{h_1,h_2\}-6h_1^3=-h_3+3\{h_1,h_2-h_1^{2}\}.
\end{split}
\end{equation}
We begin this section by recording the trace properties of the tensors $h_1,h_2$ and $h_3$, as entailed by working with curves of metrics with fixed volume.
\begin{lema} \label{tr-u3}
Let $g_t \in \mathscr{M}_1$ be a family of Riemannian metrics with Taylor expansion at $t=0$ given by 
$g^{-1}g_t=\id+th_1+\tfrac{t^2}{2\!}h_2+\tfrac{t^3}{3!}\!h_3+o(t^4)$. We have 
\begin{equation*}
\begin{split}
&\tr(h_1)=\tr(h_2-h_1^2)=0\\
&\tr \left (h_3-\tfrac{3}{2}\{h_1,h_2-h_1^2\}-h_1^3\right )=0.
\end{split}
\end{equation*}
\end{lema}
\begin{proof}
Since $\det h_t=1$, differentiating up to third order shows that 
\begin{equation*}
\begin{split}
&\tr(I_t \circ h_t^{(1)})=0\\
& \tr(I_t \circ h_t^{(2)}+I_t^{(1)} \circ h^{(1)}_t)=0\\
&\tr (I_t \circ h_t^{(3)}+2(I_t)^{(1)} \circ h_t^{(2)}+
(I_t)^{(2)} \circ h_t^{(1)})=0.
\end{split}
\end{equation*}
The claim follows by taking $t=0$ and using \eqref{der-hinv}.
\end{proof}
Next we investigate the time derivatives of the quantity involving $\mathbf{p}$ which features in the Einstein equation 
\eqref{mainES}, that is $\mathbf{p}_t:=\mathbf{p}(I_t,h_t)$ in $\Gamma\left (\End(TM) \right )$. The first few steps in that direction are purely formal and based solely on the product rule \eqref{prod-f}; in what follows we write $\mathbf{p}_k:=\mathbf{p}_t^{(k)}(0)$ for $k \geq 0$.
Further on, since $\mathbf{p}$ is a bilinear form and $\mathbf{p}(\cdot ,\id)=\mathbf{p}(\id, \cdot)=0$ we have, by also taking 
\eqref{der-hinv} into account  
\begin{equation} \label{der-bfp}
\begin{split}
&\mathbf{p}_0=\mathbf{p}_1=0\\
&\mathbf{p}_2=-2\mathbf{p}(h_1,h_1)\\
&\mathbf{p}_3=3\mathbf{p}(-h_2+2h_1^2,h_1)-3\mathbf{p}(h_1,h_2).\\
\end{split}
\end{equation}
Using the first two equations above we see that the product rule \eqref{prod-f} grants further 
\begin{equation*}
\begin{split}
\left (I_t \circ \mathbf{p}_t \right )^{(3)}(0)=&
3I_1 \circ \mathbf{p}_2+\mathbf{p}_3\\
=&6h_1 \circ \mathbf{p}(h_1,h_1)+\mathbf{p}_3.
\end{split}
\end{equation*}
At this stage it is thus convenient to group terms and observe that the symmetric component of the Einstein equation in \eqref{mainES} reads, after differentiating to order $3$, 
\begin{equation} \label{ES2}
\begin{split}
&\la \left (\delta^g \mathbf{D}_t-\tfrac{1}{2}\mathbf{p}(h_t^{-1},h_t)+\ring{R}(h_t^{-1})+Eh_t^{-1}+\delta^{\star_g}\delta^gh_t^{-1} \right )^{(3)}(0), H \ra_{L^2}\\
=&3 \la h_1 \circ \mathbf{p}(h_1,h_1),H\ra_{L^2}.
\end{split}
\end{equation}
In order to understand this in more detail we need explicit expressions for the third order time derivatives of $\mathbf{D}_t$; for notational convenience we indicate the higher order time derivatives of 
$\mathbf{D}_t$ with $\mathbf{D}_k:=(\mathbf{D}_t)^{(k)}(0)$ for $k \ge 1 $.

We consider the differential operator $Q:\Gamma \left (\Sym^2TM \right ) \to \Omega^2(M,TM)$ defined according to 
\begin{equation} \label{defn-Q}
Q(h):=-\di_{\nabla^g}\!h^3+h^2 \sharp \di_{\nabla^g} h+\di_{\nabla^g} h(h \cdot ,h \cdot)
\end{equation}
and prove the following 
\begin{lema} \label{der-Df}
We have 
\begin{equation*}
\mathbf{D}_3=-\di_{\nabla^g} I_3+6[h_1,h_2]^{\FN}+6Q(h_1).
\end{equation*}
\end{lema}
\begin{proof}
By taking into account that $I_0=\id$ a purely algebraic computation reveals 
that 
$$ \mathbf{D}_3=\di_{\nabla^g} h_3+3I_1 \sharp \di_{\nabla^g}\!h_2+3I_2 \sharp \di_{\nabla^g}\!h_1+6\di_{\nabla^g} h_1(h_1 \cdot ,h_1 \cdot).
$$
Now let 
$D_1:=\di_{\nabla^g} h_1(h_1 \cdot ,h_1 \cdot)$ and recall that $I_1=-h_1$ and $I_2=-h_2+2h_1^2$. Using 
the expression in \eqref{bra-nac2} for the bracket $[\cdot, \cdot]^{\FN}$ and also \eqref{der-hinv} we find 
\begin{equation*}
\begin{split}
\mathbf{D}_3=&\di_{\nabla^g} h_3-3(h_1 \sharp \di_{\nabla^g}\!h_2+h_2 \sharp \di_{\nabla^g} \!h_1)+6(h_1^2 \sharp \di_{\nabla^g} h_1+D_1)\\
=&\di_{\nabla^g} h_3+3(2[h_1,h_2]^{\FN}-\di_{\nabla^g}\{h_1,h_2\})+6(h_1^2 \sharp \di_{\nabla^g} h_1+D_1)\\
=&-\di_{\nabla^g} I_3+6([h_1,h_2]^{\FN}-\di_{\nabla^g}\!h_1^3+h_1^2 \sharp \di_{\nabla^g} h_1+D_1).
\end{split}
\end{equation*}
The claim is now fully proved.
\end{proof}
As far as lower order time derivatives of $\mathbf{D}_t$ are concerned the following remark is in order.
\begin{rema} \label{2nd-ord}
By an argument similar to Lemma \ref{der-Df} we also have
\begin{equation*} \label{der-bfD}
\begin{split}
& \mathbf{D}_1=\di_{\nabla^g}\! h_1\\
& \mathbf{D}_2=\di_{\nabla^g}\! h_2-2 h_1 \sharp \di_{\nabla^g}\! h_1=\di_{\nabla^g}(h_2-2h_1^2)+2[h_1,h_1]^{\FN}.\\
\end{split}
\end{equation*}
These can also be used to re-derive, as a consistency check, the deformation theory to first and second order directly from \eqref{mainES}.
\end{rema}
In particular Lemma \ref{der-Df} shows that the leading term in the Einstein equation \eqref{ES2} can be fully expressed at third order as follows.
\begin{coro} \label{st2}
We have 
\begin{equation*} 
\begin{split}
&\la \left (\delta^g \mathbf{D}_t+\ring{R}(h_t^{-1})+Eh_t^{-1}+\delta^{\star_g}\delta^gh_t^{-1} \right )^{(3)}(0), H \ra_{L^2}\\
=&-\la \left ( \widetilde{\Delta}_E+\delta^{\star_g}\di \tr \right )I_3 , H \ra_{L^2}+6\la \delta^g[h_1,h_2]^{\FN}+\delta^gQ(h_1), H \ra_{L^2}
\end{split}
\end{equation*}
for all $H \in \Gamma \left (\Sym^2TM \right )$.
\end{coro}
\begin{proof}
By taking Lemma \ref{der-Df} into account we obtain 
\begin{equation*} 
\begin{split}
&\la \left (\delta^g \mathbf{D}_t+\ring{R}(h_t^{-1})+Eh_t^{-1}+\delta^{\star_g}\delta^gh_t^{-1} \right )^{(3)}(0), H \ra_{L^2}\\
=&\la \left (-\delta^g \di_{\nabla^g}+\ring{R}+E+\delta^{\star_g}\delta^g \right )I_3, H \ra_{L^2}
+6\la \delta^g[h_1,h_2]^{\FN}+\delta^gQ(h_1), H \ra_{L^2}.
\end{split}
\end{equation*}
However the comparison formula \eqref{wz1} together with the definition of the modified Einstein operator 
$\widetilde{\Delta}_E$ ensure that 
\begin{equation*}
\la \left (-\delta^g \di_{\nabla^g}+\ring{R}+E+\delta^{\star_g}\delta^g \right ) I_3,H \ra_{L^2}=
-\la (\widetilde{\Delta}_E+\delta^{\star_g}\di \tr) I_3,H \ra_{L^2} 
\end{equation*}
and the claim follows.
\end{proof}
The next step is to harness the third order derivatives of the tensor $\mathbf{p}_t$ as follows.
\begin{lema} \label{step3}
We have 
\begin{equation*} 
\begin{split}
&\tfrac{1}{3}\la 6\delta^g[h_1,h_2]^{\FN}-\tfrac{1}{2}\mathbf{p}_3, H\ra_{L^2}\\
=&\la \bfv(h_1,h_2)+
(\mathrm{L}-\tilde{\mathrm{L}})(h_1,h_2)-\mathbf{p}(h_1^2,h_1),H \ra_{L^2}
\end{split}
\end{equation*}
for all $H \in \Gamma \left (\Sym^2TM \right )$.
\end{lema}
\begin{proof}
Polarise the identity in part (i) of Lemma \ref{diff-Ls} whilst taking into account that 
the bracket $[\cdot, \cdot]^{\FN}$, the operator $\bfv$ as well as $\mathrm{L}-\mathrm{\tilde{L}}$
are symmetric. It follows that 
\begin{equation*}
\la 2\delta^g[h_1,h_2 ]^{\FN}+\tfrac{1}{2}\mathbf{p}(h_1,h_2)+\tfrac{1}{2}\mathbf{p}(h_2,h_1),H\ra_{L^2}=\la \bfv(h_1,h_2)+
(\mathrm{L}-\mathrm{\tilde{L}})(h_1,h_2),H \ra_{L^2}.
\end{equation*}
The claim follows from the expression for $\mathbf{p}_3$ given in \eqref{bfv-1}.
\end{proof}
To be able to summarise these calculations we consider the symmetric tensor 
\begin{equation} \label{u3-d}
\mathbf{u}_3:=I_3-\tfrac{3}{4}\{h_1,h_2\}
\end{equation}
which turns out to be the perturbation of $h_3$ needed for the deformation theory. For further use we also record the following explicit expression 
\begin{equation} \label{u3-d+}
\mathbf{u}_3:=-h_3+\tfrac{9}{4}\{h_1,h_2-h_1^2\}-\tfrac{3}{2}h_1^3
\end{equation}
which follows from \eqref{der-hinv}. The new differential operator which together with $\bfv$ governs the Einstein deformation theory to third order is $ \mathbf{w} : \Gamma \left (\Sym^2TM \right ) \to \Gamma \left (\Sym^2TM \right )$ defined according to 
\begin{equation} \label{w-def}
\la \mathbf{w}(h),H \ra_{L^2}:=\la 2\delta^gQ(h)-\mathbf{p}(h^2,h)-h \circ \mathbf{p}(h,h),H \ra_{L^2}
\end{equation}
whenever $H \in \Gamma \left (\Sym^2TM \right )$.

The main result in this section is the following explicit form for the Einstein equation to third order. We also 
re-derive, for consistency, the Einstein equations to order $1$ and $2$ using our present approach.
\begin{teo}\label{thm-54} Let $g_t \in \mathscr{M}_1$ be a family of Riemannian metrics with Taylor expansion at $t=0$ given by 
$g^{-1}g_t=\id+th_1+\tfrac{t^2}{2!}h_2+\tfrac{t^3}{3!}h_3+o(t^4)$.
Assume that $g$ is Einstein with Einstein constant $E$. Then 
$$(\Ric^{g_t})^{(k)}(0)=0 \ \mathrm{ for} \ 1 \leq k \leq 3$$ if and only if the following hold 
\begin{eqnarray}
\widetilde{\Delta}_Eh_1&=&0 \label{xxx1}\\
\widetilde{\Delta}_E(h_2-h_1^2)&=&Eh_1^2-\bfv(h_1,h_1)+\tfrac{1}{2}\widetilde{\Delta}_Eh_1^2+\tfrac{1}{2}\delta^{\star_g} \di\! \tr(h_1^2) \label{xxx2}\\
(\widetilde{\Delta}_E+\delta^{\star_g}\di\tr)\mathbf{u}_3&=&3\mathbf{v}(h_1,h_2)-
\tfrac{3}{4}\{h_1,(\widetilde{\Delta}_E+2E+\delta^{\star}\di\tr )h_2\}+3\mathbf{w}(h_1). \label{xxx3}
\end{eqnarray}
%

\end{teo}
\begin{proof}
To prove the first equation differentiate \eqref{mainES} to first order with respect to $t$, then take $t=0$. Since $\mathbf{p_0}=\mathbf{p}_1=0$ it follows that $\left (I_t \circ \mathbf{p}_t \right )^{(1)}(0)=0$. Thus 
$$\la \delta^g \mathbf{D}_1+(\ring{R}+E+\delta^{\star_g}\delta^g)I_1,H\ra_{L^2}=0$$
whenver $H \in \Gamma \left ( \Sym^2TM\right )$. Since $\mathbf{D}_1=\di_{\nabla^g}\! h_1$ and 
$I_1=-h_1$ the claim follows from \eqref{wz1}, by also using that $\tr h_1=0$.

To prove \eqref{xxx2} we differentiate \eqref{mainES} to second order with respect to $t$, then take $t=0$; since $\mathbf{p_0}=\mathbf{p}_1=0$ it follows that 
$\left (I_t \circ \mathbf{p}_t \right )^{(2)}(0)=\mathbf{p}_2=-2\mathbf{p}(h_1,h_1).$
Thus the Einstein equation to second order reads 
$$ \la \delta^g \mathbf{D}_2+\left (\ring{R}+E+\delta^{\star_g}\delta^g \right )I_2+\mathbf{p}(h_1,h_1),H\ra_{L^2}=0.
$$ 
Now $\mathbf{D}_2=\di_{\nabla^g}(h_2-2h_1^2)+2[h_1,h_1]^{\FN}$ according to Remark \ref{2nd-ord}; hence 
\begin{equation*}
\begin{split}
\la \delta^g \mathbf{D}_2+\mathbf{p}(h_1,h_1),H\ra_{L^2}=\la \delta^g\di_{\nabla^g}(h_2-2h_1^2)+\bfv(h_1,h_1)
+\{h_1,\mathrm{K}h_1\}-\mathrm{K}h_1^2, H\ra_{L^2}
\end{split}
\end{equation*}
after taking into account Lemma \ref{diff-Ls}. Because $I_2=-h_2+2h_1^2$ re-arranging terms leads further to 
$\la \left (\delta^g\di_{\nabla^g}-\delta^{\star_g}\delta^g -\ring{R}-E \right ) (h_2-2h_1^2)+\bfv(h_1,h_1)
+\{h_1,\mathbf{L}h_1\}-\mathbf{L}h_1^2, H\ra_{L^2}=0.$ 
Now use the Weitzenb\"ock formula in \eqref{wz1} for the first summand above together with the fact that 
$\mathbf{L}h_1=-Eh_1$ which is granted by $\widetilde{\Delta}_Eh_1=0$ and $\tr(h_1)=0$. By also taking into account the expression for $\mathbf{L}$ we get that the symmetric component of $-\mathbf{L}h_1^2$ equals 
$\tfrac{1}{2} \left (\Delta_E-
2\delta^{\star_g}\delta^g \right )h_1^2+Eh_1^2$. Gathering these facts yields 
$$\la \left (\Delta_E-
2\delta^{\star_g}\delta^g \right )(h_2-2h_1^2)+\bfv(h_1,h_1)-Eh_1^2+\tfrac{1}{2} \left (\Delta_E-
2\delta^{\star_g}\delta^g \right )h_1^2 , H\ra_{L^2}=0 .$$
Since the curve of Riemannian metrics $g_t$ has fixed volume we have that $\tr(h_2-h_1^2)=0$ and the claim follows.

We now prove \eqref{xxx3} by differentiating \eqref{mainES} to third order with respect to $t$, then taking $t=0$. Using Corollary \ref{st2} and Lemma \ref{step3} in equation \eqref{ES2} leads to 
 \begin{equation*}
 \begin{split}
(\widetilde{\Delta}_E+\delta^{\star_g}\di\tr)I_3=&3\mathbf{v}(h_1,h_2)+3(\mathrm{L}-\tilde{\mathrm{L}})(h_1,h_2)+
3\mathbf{w}(h_1).\\
\end{split}
\end{equation*}
Thus in order to prove the claim we only need to deal with the term $(\mathrm{L}-\tilde{\mathrm{L}})(h_1,h_2)$, which is done as follows. Polarising part (ii) in Lemma \ref{diff-Ls} yields 
\begin{equation} \label{diff-L}
2(\mathrm{L}-\tilde{\mathrm{L}})(h_1,h_2)=\{\mathbf{L}h_1,h_2\}+\{h_1,\mathbf{L}h_2\}-\mathbf{L}\{h_1,h_2\}.
\end{equation}
By the same Lemma we see that the component $\mathbf{L}_{sym}$ of $\mathbf{L}$ on $\Gamma \left (\Sym^2TM \right )$ is given by $-2\mathbf{L}_{sym}=\widetilde{\Delta}_E+\delta^{\star_g}\di\tr+2E$. After taking into account that $h_1$ satisfies $\widetilde{\Delta}_Eh_1=0$ and $\tr(h_1)=0$ 
it follows that 
\begin{equation*}
\la (\mathrm{L}-\tilde{\mathrm{L}})(h_1,h_2),H \ra_{L^2}=-\tfrac{1}{4}\la 
\{h_1,(\widetilde{\Delta}_E+2E+\delta^{\star}\di\tr )h_2\},H \ra_{L^2}+\tfrac{1}{4}\la (\widetilde{\Delta}_E+\delta^{\star_g}\di\tr)\{h_1,h_2\},H \ra_{L^2}.
\end{equation*}
The claim follows now easily from the definition of $\mathbf{u}_3$ in \eqref{u3-d}.
\end{proof}
\begin{rema} 
Part (ii) in Theorem \ref{thm-54} corrects an error made in the terms containing 
the divergence $\delta^gh_1$ in \cite{NS-E}[Theorem 3.13]. As showed in Theorem \ref{thm-54} when 
$\tr(h_1)=0$
the rest of the terms in the second order Einstein equation as stated in \cite{NS-E}[Theorem 3.13]
are unaffected; hence all ulterior applications of \cite{NS-E}[Theorem 3.13], including 
those in \cite{N1} are correct, since they assume the normalisation $\delta^gh_1=0$. 
\end{rema}
Record that the Einstein equations in Theorem \ref{thm-54} do not depend on the choice of a specific gauge in $\mathbf{G}$; particular gauge choices will be made only later on in the paper.



\section{Deformation theory for K\"ahler Einstein metrics} \label{K-E}
\subsection{Multiplicative properties of the complex bracket} \label{mult-c}
We start from the following identity which plays a fundamental role in this paper. It allows using type considerations to describe the obstruction to deformation in the K\"ahler case.
\begin{pro} \label{id-main}
Assume that $h \in \Gamma \left (\Sym^{2}TM \right )$. Then 
$$ \di_{\nabla^g}\!h(h \cdot ,h \cdot)=h \sharp \di_{\nabla^g}\!h^2-\di_{\nabla^g}\!h^3+h \circ [h,h]^{\FN}.
$$
\end{pro}
\begin{proof}
This is a purely tensorial computation based on \eqref{bra-na}. Indeed, using the product rule  $h(\nabla^g_Xh)=\nabla^g_Xh^2-(\nabla^g_X h)h$ we obtain 
\begin{equation*}
\begin{split}
h \circ [h,h]^{\FN}(X,Y)=&-h(\nabla^g_{hX}h)Y+h(\nabla^g_{hY}h)X+(h^2 \circ \di_{\nabla^g}\!h)(X,Y)\\
=&-(\nabla^g_{hX}h^2)Y+(\nabla^g_{hX}h)hY+(\nabla^g_{hY}h^2)X-(\nabla^g_{hY}h)hX+(h^2 \circ \di_{\nabla^g}\!h)(X,Y)\\
=&\di_{\nabla^g}\!h(hX,hY)-(\nabla^g_{hX}h^2)Y+(\nabla^g_{hY}h^2)X+(h^2 \circ \di_{\nabla^g}\!h)(X,Y)\\
=&\di_{\nabla^g}\!h(hX,hY)-\di_{\nabla^g}\!h^2(hX,Y)-(\nabla^g_Yh^2)hX\\
+&\di_{\nabla^g}\!h^2(hY,X)+(\nabla^g_Xh^2)hY+(h^2 \circ \di_{\nabla^g}\!h)(X,Y)\\
=&\di_{\nabla^g}\!h(hX,hY)-(h \sharp\di_{\nabla^g}\!h^2)(X,Y)\\
+&(\nabla^g_Xh^2)hY-(\nabla^g_Yh^2)hX+h^2\left ((\nabla^g_Xh)Y-(\nabla^g_Yh)X\right )\\
=&\di_{\nabla^g}\!h(hX,hY)-(h \sharp\di_{\nabla^g}\!h^2)(X,Y)+\di_{\nabla^g}\!h^3(X,Y)
\end{split}
\end{equation*}
which proves the claim.
\end{proof}
Note the above identity works on arbitrary Riemannian manifolds, not necessarily K\"ahler.
Returning to the K\"ahler set-up we derive below a few useful first consequences of Proposition \ref{id-main}.
\begin{coro} \label{ids-2}
Assume that $(M^{2m},g,J)$ is K\"ahler and that $h \in \Gamma \left (\Sym^{2,-}TM \right )$. Then 
$$ \la h \sharp \di_{\nabla^g}^{+}h^2, \bdel H\ra_{L^2}+\la h^2 \circ \bdel h,\bdel H\ra_{L^2}=\la \bdel h^3,\bdel H\ra_{L^2}$$
for all $H \in \Gamma \left (\Sym^{2,-}TM \right )$.
\end{coro}
\begin{proof}
According to the comparison formula \eqref{t-FNKS} the component of $[h,h]^{\FN}$ on $\Gamma \left (\lambda^2(M,TM) \right )$ is given by $[h,h]^c+h \circ \bdel h$. Projecting the identity in Proposition \ref{id-main} onto $\lambda^2(M,TM)$ thus yields 
\begin{equation*}
\bdel h(h \cdot, h \cdot)=h \sharp \di_{\nabla^g}^{+}h^2-\bdel h^3+h \circ [h,h]^c+h^2 \circ \bdel h.
\end{equation*}
Now take into account that $\bdel h(h \cdot, h \cdot)$ and $h \circ [h,h]^c$ belong to $\lambda^2_{+}(M,TM)$. Thus taking the scalar 
product with $\bdel H$, which belongs to $\lambda^2_{-}(M,TM)$ leads to the claim.
\end{proof}
In fact a more general version of the above Corollary can be proved directly and reads as follows. Below we indicate 
with $\widetilde{\Delta}_E^{-}$ the component on $\Gamma \left (\Sym^{2,-}TM \right )$ of the restriction of $\widetilde{\Delta}_E$ to 
$\Gamma \left (\Sym^{2,-}TM \right )$; since the Einstein operator preserves 
$\Gamma \left (\Sym^{2,-}TM \right )$ and elements of the latter space are trace-free we have 
$$\widetilde{\Delta}_E^{-}=\Delta_E-2\delta^{\star_g,-} \circ \delta^g.
$$
\begin{pro} \label{ids-3}
Assume that $(M^{2m},g,J)$ is K\"ahler and that $H_{\pm} \in \Gamma \left (\Sym^{2,\pm}TM \right )$. Then 
\begin{equation} \label{ids-K5}
\left (H_{-} \sharp \di_{\nabla^g}^{+}H_{+} \right )_{\lambda^2_{-}(M,TM)}=\bdel (H_{+} \circ H_{-})-H_{+} \circ \bdel H_{-}.
\end{equation}
In particular we have 
\begin{equation} \label{ids4-K}
\la H_{-} \sharp \bdel H_{-},\di_{\nabla^g}H_{+}\ra_{L^2}+\la H_{+} \circ \bdel H_{-},\bdel H_{-} \ra_{L^2}=\tfrac{1}{4}
\la \{H_{-}, \widetilde{\Delta}_E^{-} H_{-}\}+[H_{-},\di^{-} \delta^gH_{-}],H_{+} \ra_{L^2}.
\end{equation}
\end{pro}
\begin{proof}
Whenever $\gamma$ belongs to $\lambda^2(M,TM)$ we indicate with $\gamma_{-}$ its component on $\lambda^2_{-}(M,TM)$; explicitely 
$2\gamma_{-}(X,Y)=\gamma(X,Y)+J\gamma(X,JY)$. Now let $\gamma:=H_{-} \sharp \di_{\nabla^g}^{+}H_{+}$ and recall that 
$2\di_{\nabla^g}^{+}H_{+}=\di_{\nabla^g}H_{+}+\di_{\nabla^g}H_{+}(J \cdot ,J\cdot)$. A straightforward tensorial computation which only uses that $H_{\pm}J=\pm J H_{\pm}$ and the expansion of $\di_{\nabla^g}$ in terms of $\nabla^g$ reveals that 
\begin{equation*}
2 \gamma_{-}(X,Y)=(\nabla^g_XH_{+})H_{-}Y-(\nabla^g_YH_{+})H_{-}X+(\nabla^g_{JX}H_{+})JH_{-}Y-(\nabla^g_{JY}H_{+})JH_{-}X.
\end{equation*}
The product rule for $\nabla^g$ combined with the definition of $\bdel$ thus leads to having the equality 
$\gamma_{-}=\bdel (H_{+} \circ H_{-})-H_{+} \circ \bdel H_{-}$ and the first claim is proved. To prove \eqref{ids4-K} we first observe that 
$H_{-} \sharp \bdel H_{-}$ belongs to $\Omega^{1,1}(M,TM)$; hence 
\begin{equation*}
\la H_{-} \sharp \bdel H_{-},\di_{\nabla^g}H_{+}\ra_{L^2}=\la H_{-} \sharp \bdel H_{-},\di^{+}_{\nabla^g}H_{+}\ra_{L^2}=
\la \bdel H_{-},H_{-} \sharp \di^{+}_{\nabla^g}H_{+}\ra_{L^2}.
\end{equation*}
Because $\bdel H_{-}$ is a section of  $\lambda^2_{-}(M,TM)$ using \eqref{ids-K5} shows that 
\begin{equation*} 
\begin{split}
\la H_{-} \sharp \bdel H_{-},\di_{\nabla^g}H_{+}\ra_{L^2}+\la H_{+} \circ \bdel H_{-},\bdel H_{-} \ra_{L^2}=&\la  \bdel H_{-},
\bdel (H_{+} \circ H_{-})\\
=&\la \delta^g \bdel H_{-},H_{+} \circ H_{-}\ra_{L^2}.
\end{split}
\end{equation*}
By part (ii) in Proposition \ref{wz_K} we have the type decomposition 
$\delta^g \bdel H_{-}=\tfrac{1}{2}(\widetilde{\Delta}^{-}_E-\di^{-} \circ \delta^g)H_{-}$ in 
$\Gamma \left ( \Sym^{2,-}TM \right ) \oplus \Gamma \left (\lambda^2M \right )$; furthermore the composition $H_{+} \circ H_{-}$ is $J$-anti-invariant and splits as $H_{+} \circ H_{-}=\tfrac{1}{2}\{H_{+},H_{-}\}+
\tfrac{1}{2}[H_{+},H_{-}]$ into symmetric respectively skew-symmetric components. It follows that 
\begin{equation*} 
\begin{split}
\la \delta^g \bdel H_{-},H_{+} \circ H_{-}\ra_{L^2}=&\tfrac{1}{2}\la  \widetilde{\Delta}^{-}_E H_{-}-\di^{-} \delta^g H_{-},H_{+} \circ H_{-} \ra_{L^2}\\
=&
\tfrac{1}{4}\la  \{H_{-},\widetilde{\Delta}^{-}_E H_{-} \}+[H_{-},\di^{-} \delta^gH_{-}],H_{+} \ra_{L^2}
\end{split}
\end{equation*}
after also taking into account that $H_{\pm}$ are symmetric tensors. This finishes the proof of the claim.
\end{proof}
\subsection{Computation of $\bfv(h,h^2)$} \label{v12}
We begin by recalling that the restriction of $\bfv$ to $\Gamma \left (\Sym^{2,-}TM \right )$ can be determined solely in terms of the complex bracket.
\begin{lema}\cite{N1}[Proposition 4.14]\label{3-}
Assume that $h_i \in \Gamma \left (\Sym^{2,-}TM \right )$ for $1 \leq i \leq 3$. Then 
$$ \bfv(h_1,h_2,h_3)=2 \mathfrak{S}_{abc} \la [h_a,h_b]^c, \bdel h_c\ra_{L^2}-\tfrac{1}{2} \mathfrak{S}_{abc} \la \delta^{g}\{h_a,h_b \}, \delta^gh_c\ra_{L^2}
$$
with cyclic permutations on the indices $abc$.
\end{lema}
Furthermore we define the operator 
\begin{equation} \label{div-1}
\begin{split}
 \mathrm{div}_1(h_1,h_2):=&\tfrac{1}{2} \{\delta^{\star_g,-}\delta^g h_1,h_2\}+\tfrac{1}{2}\{h_1,\delta^{\star_g,-}\delta^gh_2\}-\delta^{\star_g,+}
(h_1 \delta^gh_2+h_2\delta^gh_1)\\
+&\tfrac{1}{4}[h_1,\di^{-} \delta^gh_2]+\tfrac{1}{4}[h_2,\di^{-} \delta^gh_1]
\end{split}
\end{equation}
for all $h_1,h_2$ in $\Gamma \left (\Sym^{2,-}TM \right )$. We also define 
\begin{equation*}
2\mathbf{D}_1(h_1,h_2)=\{\ring{R}h_1,h_2\}+\{h_1,\ring{R}h_2\}-(\tfrac{1}{2}\Delta_E-\ring{R})\{h_1,h_2\}+
\delta^{\star_g}(2\delta^g+\tfrac{1}{2}\di\!\tr)\{h_1,h_2\}
\end{equation*}
and $\mathbf{D}(h_1,h_2)=\mathbf{D}_1(h_1,h_2)+\mathrm{div}_1(h_1,h_2)$. We will frequently use the shorthand notation $\mathbf{D}(h)=\mathbf{D}(h,h)$ in what follows. Record that the operator $\mathbf{D}$ enters the following 
\begin{pro} \label{hash-o}\cite{N1}[Proposition 4.19]
Assume that $H_{\pm}$ belong to $\Gamma \left (\Sym^{2,\pm}TM \right )$. Then 
\begin{equation*}
\begin{split}
\la H_{+} \sharp \di_{\nabla^g}H_{-}, \di_{\nabla^g}H_{-} \ra_{L^2}=&\la \mathbf{D}(H_{-},H_{-}) ,H_{+}\ra_{L^2}+2 \la \delta^g (H_{-} \circ \bdel H_{-} ), H_{+}\ra_{L^2}\\
+&2\la H_{+} \sharp \bdel H_{-}, \bdel H_{-} \ra_{L^2}-2 \la \bdel H_{-}, \delta^gH_{+} \wedge H_{-}\ra_{L^2}
\end{split}
\end{equation*}
\end{pro}
Based on this we show below how to determine the component of $\bfv$ acting on the space $2\Gamma (\Sym^{2,-}TM) \oplus \Gamma (\Sym^{2,+}TM)$.
\begin{pro} \label{vm-7}
Assume that $H_{\pm}$ belong to $\Gamma \left (\Sym^{2,\pm}TM \right )$. We have 
\begin{equation*}
\begin{split}
\tfrac{1}{2}\bfv(H_{-},H_{-},H_{+})=&\tfrac{1}{4} \la \{H_{-}, \widetilde{\Delta}_E^{-}H_{-}\}+(\widetilde{\Delta}_E+2E)H_{-}^2-2[H_{-},\di^{-} \delta^gH_{-}],H_{+}\ra_{L^2}\\
+&\la \delta^{\star_g,+}(H_{-}\delta^gH_{-}),H_{+} \ra_{L^2}+
\la  \bdel H_{-}, \delta^gH_{+} \wedge H_{-} \ra_{L^2}\\
+&\la H_{+} \circ \bdel H_{-}, \bdel H_{-} \ra_{L^2}-\la H_{+} \sharp \bdel H_{-},  \bdel H_{-} \ra_{L^2}.
\end{split}
\end{equation*}
\end{pro}
\begin{proof}
By \eqref{bfv-1} we have 
\begin{equation*}
\begin{split}
\bfv(H_{-},H_{-},H_{+})=&2\la [H_{-},H_{-}]^{\FN}, \di_{\nabla^g}H_{+} \ra_{L^2}-\la H_{+} \sharp \di_{\nabla^g}\!H_{-},
\di_{\nabla^g} \!H_{-}\ra_{L^2}+
\la \tilde{\mathrm{L}}(H_{-},H_{-}),H_{+}\ra_{L^2}.
\end{split}
\end{equation*}
Due to the comparison formula $[H_{-},H_{-}]^{\FN}=[H_{-},H_{-}]^c+H_{-} \circ \bdel H_{-}+\di^{+}_{\nabla^g}H_{-}^2-H_{-} \sharp \bdel H_{-}$(see \eqref{t-FNKS}) and since the divergence $\delta^g[H_{-},H_{-}]^c$ belongs to 
$ \Gamma \left (\Sym^{2,-}TM \right )$ we get 
\begin{equation*}
\la [H_{-},H_{-}]^{\FN}, \di_{\nabla^g}H_{+} \ra_{L^2}=\la \delta^g(H_{-} \circ \bdel H_{-}),H_{+}\ra_{L^2}+
\la \delta^g\di\!^{+}_{\nabla^g}H_{-}^2,H_{+} \ra_{L^2}-\la 
H_{-} \sharp \bdel H_{-},\di\!_{\nabla^g}H_{+}
\ra_{L^2}.
\end{equation*}
Using the expression for $\la H_{+} \sharp \di_{\nabla^g}H_{-}, \di_{\nabla^g}H_{-} \ra_{L^2}$ provided by Proposition 
\ref{hash-o}  
thus yields 
\begin{equation} \label{vhh2-New}
\begin{split}
\tfrac{1}{2}\bfv(H_{-},H_{-},H_{+})
=&\la \delta^g\di\!^{+}_{\nabla^g}H_{-}^2-\tfrac{1}{2}\mathbf{D}(H_{-},H_{-})+\tfrac{1}{2} \widetilde{\mathrm{L}}(H_{-},H_{-}) ,H_{+}\ra_{L^2}\\
-&\la H_{-} \sharp \bdel H_{-}, \di\!_{\nabla^g}H_{+} \ra_{L^2}\\
-&\la H_{+} \sharp \bdel H_{-}, \bdel H_{-} \ra_{L^2}+ \la \bdel H_{-}, \delta^gH_{+} \wedge H_{-}\ra_{L^2}.
\end{split}
\end{equation}
The operators featuring on the first displayed line above may be explicited as follows.
\begin{itemize}
\item[$\bullet$] $\widetilde{\mathrm{L}}(H_{-},H_{-})=\{H_{-}, \delta^g\di\!_{\nabla^g}H_{-}\}-\delta^g\di_{\nabla^g}H_{-}^2$ 
according to \eqref{L-t}
\item[$\bullet$] $\la \delta^g\di\!^{+}_{\nabla^g}H_{-}^2, H_{+}\ra_{L^2}=\la (\tfrac{1}{2}\Delta_E+\ring{R})H_{-}^2,H_{+} \ra_{L^2}$ by type considerations based on part (i) in Proposition \ref{wz_K}. 
\end{itemize}
Next we use the Weitzenb\"ock formula \eqref{wz1} for $H_{-}$ respectively $H_{-}^2$ in order to further expand $\widetilde{\mathrm{L}}(H_{-},H_{-})$; after also taking into account the explicit 
expression for $\mathbf{D}(H_{-},H_{-})$ a lentghy, though entirely algebraic, calculation leads to 
\begin{equation*} 
\begin{split}
&\la \delta^g\di\!^{+}_{\nabla^g}H_{-}^2-\tfrac{1}{2}\mathbf{D}(H_{-},H_{-})+\tfrac{1}{2} \widetilde{\mathrm{L}}(H_{-},H_{-}) ,H_{+}\ra_{L^2}\\
=& \la \tfrac{1}{2}\{H_{-}, \widetilde{\Delta}_E^{-}H_{-}\}+\tfrac{1}{4}(\widetilde{\Delta}_E+2E)H_{-}^2-\tfrac{1}{4}[H_{-},\di^{-} \delta^gH_{-}]+\delta^{\star_g,+}(H_{-}\delta^{g}H_{-}),H_{+}\ra_{L^2}.
\end{split}
\end{equation*}
We plug this into \eqref{vhh2-New}, together with the expression for $\la H_{-} \sharp \bdel H_{-}, \di\!_{\nabla^g}H_{+} \ra_{L^2}$ given in \eqref{ids4-K}. The claim follows by gathering terms.
\end{proof}
For our needs in this paper it suffices to polarise the general formula in Proposition \ref{vm-7} on specific elements as described below. 
\begin{pro} \label{vm-7bis}
Assume that $H_{\pm}$ belong to $\Gamma \left (\Sym^{2,\pm}TM \right )$ and that $h \in \Gamma \left (\Sym^{2,-}TM \right )$ additionally satisfies $\bdel h=0$ and $\delta^gh=0$. We have
\begin{equation*} 
\begin{split}
\bfv(h,H_{-},H_{+})=&\tfrac{1}{4} \la \{h, \widetilde{\Delta}_E^{-}H_{-}\}+(\widetilde{\Delta}_E+2E)\{h,H_{-}\}-2[h,\di^{-} \delta^gH_{-}],H_{+}\ra_{L^2}\\
+&\la \delta^{\star_g,+}(h\delta^gH_{-}),H_{+} \ra_{L^2}+
\la  \bdel H_{-}, \delta^gH_{+} \wedge h \ra_{L^2}.
\end{split}
\end{equation*}
\end{pro}
\begin{proof}
The idea is to take $H_{-}=h+tH_{-}$ with $t \in \mathbb{R}$ in Proposition \ref{vm-7} and isolate the term in $t$ in each of the terms featuring in the equation therein. In doing so we 
proceed as follows. Since $\widetilde{\Delta}_E^{-}h=0$ the term in $t$ in 
$$\{h+tH_{-}, \widetilde{\Delta}_E^{-}(h+tH_{-})\}=t\{h+tH_{-},\widetilde{\Delta}_E^{-}H_{-}\}$$
equals $\{h,\widetilde{\Delta}_E^{-}H_{-}\}$. Having $(h+tH)^2=h^2+t\{h,H\}+t^2H^2$ shows that the coefficient of $t$ in 
$(\widetilde{\Delta}_E+2E)(h+tH_{-})^2$ is $(\widetilde{\Delta}_E+2E)\{h,H_{-}\}$. Since $\delta^gh=0$ the coefficient of $t$ in 
$$[h+tH_{-},\di^{-} \delta^g(h+tH_{-})]=t[h+tH_{-},\di^{-} \delta^gH_{-}]$$
thus equals $[h,\di^{-} \delta^gH_{-}]$. Entirely similar considerations based on linearity and having $\bdel h=0$ and $\delta^gh=0$ 
shows that the coefficient of $t$ in 
$$\la \delta^{\star_g,+}((h+tH_{-})\delta^g(h+tH_{-})),H_{+} \ra_{L^2}+
\la  \bdel (h+tH_{-}), \delta^gH_{+} \wedge (h+tH_{-}) \ra_{L^2}$$
respectively 
$$\la H_{+} \circ \bdel (h+tH_{-}), \bdel (h+tH_{-}) \ra_{L^2}-\la H_{+} \sharp \bdel (h+tH_{-}),  \bdel(h+t H_{-}) \ra_{L^2}
$$
equals $\la \delta^{\star_g,+}(h\delta^gH_{-}),H_{+} \ra_{L^2}+
\la  \bdel H_{-}, \delta^gH_{+} \wedge h \ra_{L^2}$ respectively $0$. Finally, since $\bfv$ is symmetric we may expand $\bfv(h+tH_{-},h+tH_{-})=\bfv(h,h)+2t\bfv(h,H_{-})+t^2\bfv(H_{-},H_{-})$ and the claim follows by collecting the facts above. 
\end{proof}
To end this section we use the above Proposition for determining the component of $\bfv(h,h^2)$ on $\Gamma \left ( \Sym^{2,-}TM \right )$, whenever $h$ lies in $\varepsilon^{-}(g)$. For temporary use we also consider the linear operator $\mathrm{div}_2 : \varepsilon^{-}(g) \to \Gamma \left (\Sym^{2,-}TM \right )$ determined from 
\begin{equation} \label{div-2}
\la \mathrm{div}_2(h),H \ra_{L^2}=\la \delta^{\star_g,-}(h \delta^gh^2)-\tfrac{1}{2}(\di^{\star}\delta^gh^2)h,H \ra_{L^2}+\la \bdel H, \delta^gh^2 \wedge h \ra_{L^2}.
\end{equation}

\begin{teo} \label{vm-3}
Let $h \in \Gamma \left (\Sym^{2,-}TM \right )$ satisfy $\bdel h=0$ and $\delta^gh=0$. Then
\begin{equation*}
\begin{split}
\la \bfv(h,h^2),H \ra_{L^2}=& \tfrac{1}{4}\la \{h, (\Delta_E-2\delta^{\star_g,+}\delta^g)h^2\},H \ra_{L^2}+\tfrac{1}{2} \la (\widetilde{\Delta}^{-}_E+2E)h^3,H \ra_{L^2}\\
+&\la \mathrm{div}_2(h),H \ra_{L^2}
\end{split}
\end{equation*}
for all $H \in \Gamma \left (\Sym^{2,-}TM \right )$.
\end{teo}
\begin{proof}
Taking $H_{+}=h^2$ and $H_{-}=H$ in Proposition \ref{vm-7bis} reveals that 
\begin{equation*}
\begin{split}
\bfv(h,h^2, H)=&\tfrac{1}{4} \la \{h, \widetilde{\Delta}_E^{-}H\}+(\widetilde{\Delta}_E+2E)\{h,H\},h^2 \ra_{L^2}\\
+&\la \delta^{\star_g,+}(h\delta^gH),h^2\ra_{L^2}+
\la  \bdel H, \delta^gh^2 \wedge h \ra_{L^2}
\end{split}
\end{equation*}
since the operators $h$ and $h^2$ mutually commute. Since the operator 
$\widetilde{\Delta}_E^{-}$ is self-adjoint the first term in the r.h.s above reads  $\tfrac{1}{4}\la \{h, \widetilde{\Delta}_E^{-}H\},h^2 \ra_{L^2}=\tfrac{1}{2} 
\la \widetilde{\Delta}_E^{-}H,h^3 \ra_{L^2}=\tfrac{1}{2} \la \widetilde{\Delta}_E^{-}h^3,H \ra_{L^2}$. Further on we observe that 
$$\la \delta^{\star_g}\di \tr\{h,H\},h^2 \ra_{L^2}=\la \di \tr\{h,H\}, \delta^gh^2 \ra_{L^2}=2\la (\di^{\star}\delta^gh^2)h,H \ra_{L^2}.
$$
Since the operators $\Delta_E, \delta^{\star_g}\delta^g$ respectively $\{h, \cdot\} : \Sym^2TM \to \Sym^2TM$ are self-adjoint using type considerations thus yields 
\begin{equation*}
\la (\widetilde{\Delta}_E+2E)\{h,H\},h^2 \ra_{L^2}=\la \{h, (\Delta_E-2\delta^{\star_g,+}\delta^g)h^2\}-2
(\di^{\star}\delta^gh^2)h+4Eh^3,H \ra_{L^2}.
\end{equation*}
Because $\la \delta^{\star_g,+}(h\delta^gH),h^2\ra_{L^2}=\la \delta^{\star_g,-}(h\delta^gh^2) ,H\ra_{L^2}$ by duality, the claim follows by collecting these facts after also taking into account the definition of $\mathrm{div}_2(h)$.

\end{proof}
\subsection{Description of the component of $\mathbf{w}$ on 
$\Gamma \left (\Sym^{2,-}TM \right )$} \label{w-comp}

We begin by observing that a crucial consequence of Proposition \ref{id-main} is the following 
\begin{pro} \label{cons1-K}
Assume that $h,H \in \Gamma \left (\Sym^{2,-}TM \right )$ and moreover that $\bdel h=0$ and $\delta^gh=0$. Then 
$$ \la Q(h), \di_{\nabla^g}\!H\ra_{L^2}=\la \delta^g(h \circ \di_{\nabla^g}\!h^2)-2\delta^g
[h^2,h]^{\FN}, H \ra_{L^2}-\la  h \circ \di\!_{\nabla^g}h^2, \bdel H \ra_{L^2}.
$$
\end{pro}
\begin{proof}
The comparison formula \eqref{t-FNKS} ensures that $[h,h]^{\FN}=[h,h]^c+\di\!_{\nabla^g}^{+}h^2$ since $\bdel h=0$. At the same time $[h,h]^c$ belongs to 
$\lambda^2_{-}(M,TM)$ and hence $h \circ [h,h]^c \in \lambda^2_{+}(M,TM)$. In particular the divergence 
$\delta^g (h \circ [h,h]^c)$ in $\End(TM)$ commutes with $J$ hence 
$$\la h \circ [h,h]^c, \di_{\nabla^g}\!H \ra_{L^2}=0.
$$
Furthermore, 
$$ \la h \circ \di\!_{\nabla^g}^{+}h^2, \di_{\nabla^g}\!H \ra_{L^2}=\la  h \circ \di\!_{\nabla^g}h^2, \di_{\nabla^g}\!H-\bdel H \ra_{L^2}=
\la \delta^g(h \circ \di_{\nabla^g}\!h^2), H \ra_{L^2}-\la  h \circ \di\!_{\nabla^g}h^2, \bdel H \ra_{L^2}.
$$
After gathering terms and using Proposition \ref{id-main} we get 
\begin{equation} \label{prod-2}
\begin{split}
\la \di_{\nabla^g}\!h(h \cdot ,h \cdot), \di_{\nabla^g}\!H\ra_{L^2}=&\la \delta^g(h \circ \di_{\nabla^g}\!h^2), H \ra_{L^2}+\la 
h \sharp \di_{\nabla^g}\!h^2-\di_{\nabla^g}\!h^3, \di_{\nabla^g}\!H \ra_{L^2}\\
-&\la  h \circ \di\!_{\nabla^g}h^2, \bdel H \ra_{L^2}.
\end{split}
\end{equation}
Recalling that $Q(h)=\di_{\nabla^g}\!h(h \cdot ,h \cdot)+h^2 \sharp \di_{\nabla^g}\!h-\di_{\nabla^g}\!h^3$ the claim follows now directly by expressing $[h,h^2]^{\FN}$ according to the definition of the Fr\"olicher-Nijenhuis bracket in 
\eqref{bra-nac2}.
\end{proof}
Now we take into account the expression for $\mathbf{p}(h^2,h)$ previously obtained. Since $h$ belongs to $\Gamma \left (\Sym^{2,-}TM \right )$ and 
$h^2 \in \Gamma \left (\Sym^{2,+}TM \right )$ the definition of $\mathbf{j}$ in \eqref{j-d} grants that 
$\mathbf{j}(h^2,h)=0$ thus 
$$ \la \mathbf{p}(h^2,h),H \ra_{L^2}=2\la \mathbf{i}(h^2,h),H \ra_{L^2}.
$$
Furthermore 
$$ \la \mathbf{i}(h^2,h),H \ra_{L^2}=\la \delta^g(h \circ \di_{\nabla^g}\!h^2),H\ra_{L^2}+\mathbf{q}(h^2,h \circ H)-
\la h, \mathrm{L}(H,h^2)\ra_{L^2}
$$
by Lemma \ref{i-1}. At the same time taking into account \eqref{bfv-1} shows that 
\begin{equation*}
\bfv(h^2,h,H)=2\la [h^2,h]^{\FN},\di_{\nabla^g}\!H\ra_{L^2}-\la H \sharp \di_{\nabla^g}\!h,\di_{\nabla^g}\!h^2 \ra_{L^2}+\la \widetilde{\mathrm{L}}(h,h^2),H\ra_{L^2}.
\end{equation*}
Putting these facts together shows that 
\begin{equation*} \label{Q11}
\begin{split}
\la \delta^g(h \circ \di_{\nabla^g}\!h^2)-2\delta^g
[h^2,h]^{\FN}, H \ra_{L^2}=&\tfrac{1}{2}\la \mathbf{p}(h^2,h), H \ra_{L^2}-\bfv(h^2,h,H)-\mathbf{q}(h^2,h \circ H)\\
+&\la h, \mathrm{L}(H,h^2) \ra_{L^2}+\la \widetilde{\mathrm{L}}(h,h^2),H \ra_{L^2}\\
-&\la H \sharp \di_{\nabla^g}\!h, \di_{\nabla^g}\!h^2 \ra_{L^2}.
\end{split}
\end{equation*}
After also using Proposition \ref{cons1-K} it follows that 
\begin{equation} \label{Q11}
\begin{split}
\la \delta^gQ(h)-\tfrac{1}{2}\mathbf{p}(h^2,h),H\ra_{L^2}=&-\bfv(h^2,h,H)
+\la h, \mathrm{L}(H,h^2) \ra_{L^2}+\la \widetilde{\mathrm{L}}(h,h^2),H \ra_{L^2}\\
&-\mathbf{q}(h^2,h \circ H)-\la H \sharp \di_{\nabla^g}\!h, \di_{\nabla^g}\!h^2 \ra_{L^2}\\
&-\la h \circ \di_{\nabla^g}h^2, \bdel H\ra_{L^2}.
\end{split}
\end{equation}
This can be further refined to obtain the following 
\begin{pro} \label{Q12} Assume that the pair $(h,H)$ belongs to $\varepsilon^{-}(g) \oplus \Gamma \left (\Sym^{2,-}TM \right )$. Then 
\begin{equation*}
\begin{split}
\la \delta^gQ(h)-\tfrac{1}{2}\mathbf{p}(h^2,h),H\ra_{L^2}=&-\bfv(h^2,h,H)+\la \partial^gH, h\circ \di_{\nabla^g}\!h^2+h^2\circ \di_{\nabla^g}\!h \ra_{L^2}\\
+&\la \widetilde{\mathrm{L}}(h^2,h)-\mathrm{L}(h^2,h),H \ra_{L^2}-\la h^2, \mathrm{L}(h,H)\ra_{L^2}+\mathbf{q}(h^2 \circ H,h).
\end{split}
\end{equation*}
\end{pro}
\begin{proof}
Observe that $\la \nabla^g_{He_i}h,  \nabla^g_{e_i}h^2\ra_{L^2}=0$ since $h$ is $J$-anti-invariant whilst $h^2$ is $J$-invariant; 
thus letting $h_1=H, h_2=h, h_3=h^2$ in Proposition \ref{o3-p} yields 
\begin{equation*}
\begin{split}
\la H \sharp \di_{\nabla^g}\!h, \di_{\nabla^g}\!h^2 \ra_{L^2}=&-\la \di_{\nabla^g}\!H, h \circ \di_{\nabla^g}\!h^2+h^2 \circ \di_{\nabla^g}\!h \ra_{L^2}\\
+& \la H, \mathrm{L}(h,h^2)\ra_{L^2}+\la h, \mathrm{L}(H,h^2)\ra_{L^2}+\la h^2, \mathrm{L}(h,H)\ra_{L^2}\\
-&\mathbf{q}(h \circ H, h^2)-\mathbf{q}(h^2 \circ H,h).
\end{split}
\end{equation*}
In addition 
$$\la \di_{\nabla^g}\!H, h \circ \di_{\nabla^g}\!h^2+h^2 \circ \di_{\nabla^g}\!h \ra_{L^2}=\la \partial^g H, h \circ \di_{\nabla^g}\!h^2+h^2 \circ \di_{\nabla^g}\! h \ra_{L^2}+
\la \bdel H, h \circ \di_{\nabla^g}\!h^2+h^2 \circ \di_{\nabla^g}\! h \ra_{L^2}$$
and $\la \bdel H, h \circ \di_{\nabla^g}\!h^2+h^2 \circ \di_{\nabla^g}\! h \ra_{L^2}=\la \bdel H, h \circ \di_{\nabla^g}\!h^2 \ra_{L^2}$ since 
$\di_{\nabla^g}\!h \in \Omega^{1,1}(M,TM)$. Now we plug this facts into \eqref{Q11}; a short computation which furthermore uses that 
the operators $\mathbf{q}, \mathrm{L}, \widetilde{\rm L}$ are symmetric (see Lemma \ref{q} and also \eqref{L-t}) yields the claim.
\end{proof}
In order to bring this to final form we need to render fully explicit the terms involving the operators $\mathbf{q}, \mathrm{L}$ and $\widetilde{\rm L}$ and also 
the determine the quantity $\la \partial^g H, h \circ \di_{\nabla^g}\!h^2+h^2 \circ \di_{\nabla^g}\! h \ra_{L^2}$. The latter step is performed in the Lemma below.
 \begin{lema} \label{quad-K}
Let $h,H$ be in $\Gamma \left (\Sym^{2,-}TM \right )$. Then 
$$ g(h^2\circ \di_{\nabla^g}h+h \circ \di_{\nabla^g}h^2,\partial^g H)=g(\di_{\nabla^g}h^3, \partial^g\!H).$$

\end{lema}
\begin{proof}
A tensorial computation shows that 
$$ (h^2\circ \di_{\nabla^g}h+h \circ \di_{\nabla^g}h^2)(X,Y)=\gamma_XY-\gamma_YX+\di_{\nabla^g}\!h^3(X,Y)
$$
where the tensor $\gamma:TM \to \lambda^2_J$ is given by $\gamma_X=[h^2,\nabla^g_Xh]$. Now split $\gamma=\gamma^{-}+\gamma^{+}$ where 
$\gamma^{-}_{JX}JY=-\gamma^{-}_XY$ and $\gamma^{+}_{JX}JY=\gamma^{+}_XY$; thus
$g(\gamma^{+}_X \cdot, \cdot)=T(X, \cdot ,\cdot)-T(X, J\cdot ,J\cdot)$ for some form $T \in \Lambda^{(2,1)+(1,2)}M$. It follows that 
\begin{equation*}
\begin{split}
g \left (\gamma^{+}_{e_i}e_j-\gamma^{+}_{e_j}e_i, \partial^gH(e_i,e_j) \right )=&g(T, \partial^gH)
-g(T(e_i,Je_j)+T(Je_i,e_j), J \circ 
\partial^gH(e_i,e_j))\\
\end{split}
\end{equation*}
Since $T$ has real type $(1,2)+(2,1)$ we have $T(e_i,Je_j)+T(Je_i,e_j)=J \circ \left (T(e_i,e_j)-T(Je_i,Je_j) \right )$; 
the last summand above thus reads 
$$g(T(e_i,Je_j)+T(Je_i,e_j), J \circ 
\partial^gH(e_i,e_j))=g \left ( T(e_i,e_j)-T(Je_i,Je_j), \partial^gH(e_i,e_j) \right )=0$$ 
since $\partial^gH$ belongs to $\Omega^{1,1}(M,TM)$. Furthermore, since $H$ and $HJ$ are both symmetric the total alternation $\mathrm{a}(\partial^gH)=0$ thus the scalar product $g(T, \partial^g H)=0$ and therefore $g \left (\gamma^{+}_{e_i}e_j-\gamma^{+}_{e_j}e_i, \partial^gH(e_i,e_j) \right )=0$ as well. The claim follows now by taking into account that $(X,Y) \mapsto \gamma_X^{-}Y-\gamma^{-}_YX$ is $J$-anti-invariant, whilst 
$(X,Y) \mapsto \partial^gH(X,Y)$ is $J$-invariant.
\end{proof}
\begin{rema} \label{rmk-simp}
From the proof above we see that Lemma \ref{quad-K} can be formulated in a more algebraic way; namely it asserts the vanishing of the component of $h^2\circ \di_{\nabla^g}h+h \circ \di_{\nabla^g}h^2-\di_{\nabla^g}h^3$ on the Cartan summand $\Lambda^{1,1}(M,TM) \cap \ker \mathrm{a} \subseteq  \Lambda^{1,1}(M,TM)$. Additional information on the tensor  $h^2\circ \di_{\nabla^g}h+h \circ \di_{\nabla^g}h^2-\di_{\nabla^g}h^3$ will be obtained later on in Lemma \ref{quad-K+}.
\end{rema}
We are now able to record further progress towards determining the explicit form of the operator $\mathbf{w}$ as follows.
\begin{pro} \label{Q13}
Assume that $h,H \in \Gamma \left (\Sym^{2,-}TM \right )$ and moreover that we have $\bdel h=0$ and $\delta^gh=0$. Then 
\begin{equation*}
\begin{split}
\la \delta^gQ(h)-\tfrac{1}{2}\mathbf{p}(h^2,h),H\ra_{L^2}=&\la -\bfv(h^2,h)+\tfrac{1}{2}\{h,\delta^g\di_{\nabla^g}\!h^2 \}-\tfrac{1}{2} \widetilde{\Delta}_Eh^3
, H\ra_{L^2}.
\end{split}
\end{equation*}
\end{pro}
\begin{proof}
Since the operator $\Delta_E
+2\ring{R}$ is symmetric, an algebraic computation based on the formula for $\mathrm{L}$ in Lemma \ref{L-exp} and also on having $\Delta_Eh=0$ leads to 
\begin{equation*}
\la \mathrm{L}(h^2,h),H \ra_{L^2}+\la h^2, \mathrm{L}(h,H)\ra_{L^2}=\la H, \{ h^2, \ring{R}h \}\ra_{L^2}.
\end{equation*}
Further on we have $\mathbf{q}(h^2 \circ H,h)=\la \tfrac{1}{2}\{h^2,\ring{R}h\}-Eh^3,H\ra_{L^2}$ by Lemma \ref{q}, since $h$ is divergence free.
In addition, by using that $\delta^g \di_{\nabla^g}\!h=Eh+\ring{R}h$ we find that (see also \eqref{L-t} for the definition of $\widetilde{\mathrm{L}}$)
$$ \widetilde{\mathrm{L}}(h^2,h)=\tfrac{1}{2}\{h^2,Eh+\ring{R}h\}+\tfrac{1}{2} \{h, \delta^g \di_{\nabla^g}\!h^2\}-\delta^g\di_{\nabla^g}\!h^3.$$ Putting these facts together in Proposition \ref{Q12} shows that 
\begin{equation*}
\begin{split}
\la \delta^gQ(h)-\tfrac{1}{2}\mathbf{p}(h^2,h),H\ra_{L^2}=&\la -\bfv(h^2,h)+\tfrac{1}{2}\{h,\delta^g\di_{\nabla^g}\!h^2 \},H\ra_{L^2}\\
+& \la \partial^gH, h\circ \di_{\nabla^g}\!h^2+h^2 \circ \di_{\nabla^g}\!h \ra_{L^2}-\la \di_{\nabla^g}\!h^3, \di\!H\ra_{L^2}.
\end{split}
\end{equation*}
Taking into account Lemma \ref{quad-K} leads to 
\begin{equation*}
\begin{split}
\la \partial^gH, h\circ \di_{\nabla^g}\!h^2+h^2\circ \di_{\nabla^g}\!h \ra_{L^2}-\la \di_{\nabla^g}\!h^3, \di\!H\ra_{L^2}=&
\la \di_{\nabla^g}\!h^3, \partial^g H\ra_{L^2}-\la \di_{\nabla^g}\!h^3, \di_{\nabla^g}\!H\ra_{L^2}\\
=&-\la \bdel h^3, \bdel H\ra_{L^2}.
\end{split}
\end{equation*}
At the same time using part (ii) in Proposition \ref{wz_K} shows that 
$$ \la \bdel h^3, \bdel H\ra_{L^2}=\la \delta^g\bdel h^3, H\ra_{L^2}=\tfrac{1}{2}\la \widetilde{\Delta}_Eh^3,H\ra_{L^2}.
$$
The claim follows now easily.
\end{proof}
The final step needed to fully determine the component of the operator $\mathbf{w}(h)$ on 
$\Gamma \left (\Sym^{2,-}TM \right )$ is the following 
\begin{lema} \label{w-step}
Assume that $h,H$ belong to $\Gamma \left (\Sym^{2,-}TM \right )$ and that $\bdel h=0$ and $\delta^gh=0$. Then 
\begin{equation*}
\la h \circ \mathbf{p}(h,h),H \ra_{L^2}=-\la \{h, \ring{R}h^2+\delta^{\star_g}(\delta^g+\tfrac{1}{4}\di\!\tr )h^2 \},H\ra_{L^2}.
\end{equation*}
\end{lema}
\begin{proof}
We have $2\la h \circ \mathbf{p}(h,h),H \ra_{L^2}=\la \mathbf{p}(h,h), \{h,H\}\ra_{L^2}+\la \mathbf{p}(h,h), A\ra_{L^2}$
where $A:=[h,H]$ in $\Omega^{1,1}M$. 
Due to the type decomposition of $\delta^{g}[h,h]^{\FN}$ in \eqref{t-FNKS1} the skew-symmetric component in $\delta^{g}[h,h]^{\FN}$ satisfies 
$\la \delta^{g}[h,h]^{\FN},A\ra_{L^2}=0$ since $A \in \Omega^{1,1}M$. It follows that 
$$ \la \mathbf{p}(h,h), A\ra_{L^2}=0
$$
by using Proposition \ref{p-skew}. Now denote $H_{+}:=\{h,H\}$ which certainly satisfies $H_{+}J=JH_{+}$. By using Lemma 
\ref{quad-1} we get 
$\la \mathbf{p}(h,h), H_{+} \ra_{L^2}=-\la H_{+} \sharp \di\!h, \di\!h\ra_{L^2}+\la \mathrm{L}(h,h),H_{+} \ra_{L^2}$. We recall that 
\begin{equation*}
\begin{split}
&\mathrm{L}(h,h)=\{\ring{R}(h),h\}-(\tfrac{1}{2}\Delta_E+\ring{R})h^2\\
&\la H_{+} \sharp \di_{\nabla^g}\!h, \di_{\nabla^g}\!h\ra_{L^2}=\la \{\ring{R}(h),h\}-(\tfrac{1}{2}\Delta_E-\ring{R})h^2, H_{+}\ra_{L^2}+\la \delta^{\star_g}(\delta^gh^2+\tfrac{1}{2}\di\!\tr h^2),H_{+}\ra_{L^2}.
\end{split}
\end{equation*}
The first equation follows from Lemma \ref{L-exp} in which we take $h_1=h_2=h$ and use that $\Delta_Eh=0$; the second equation is entailed by Proposition \ref{hash-o} after taking into account that $\bdel h=0$ and $\delta^gh=0$.
It follows that $-\la \mathbf{p}(h,h), H_{+} \ra_{L^2}=\la 2\ring{R}h^2+\delta^{\star_g}(\delta^gh^2+\tfrac{1}{2}\di\!\tr h^2),H_{+}\ra_{L^2}$
and the claim is proved by gathering terms.
\end{proof}
\begin{lema} \label{l-step}
Let $g_t \in \mathscr{M}_1$ be an Einstein deformation with Taylor series expansion given by $g^{-1}g_t=\id+t h_1+\tfrac{t^2}{2}h_2+\tfrac{t^3}{6}h_3
+o(t^4)$, normalised according to $\delta^gh_1=0$ and \eqref{gauge-1}. Then 
$$ \la \{h_1,(\widetilde{\Delta}_E+2E+\delta^{\star}\di\tr )h_2\},H \ra_{L^2}=\la \{h_1, (\widetilde{\Delta}_E+2E+\tfrac{3}{2}\delta^{\star_g}\di\!\tr)h_1^2\}, H\ra_{L^2}$$
whenever $H$ belongs to $\Gamma \left (\Sym^{2,-}TM \right )$. 
\end{lema}
\begin{proof}
We consider again  $H_{+}:=\{h_1,H\} \in \Gamma \left (\Sym^{2,+}TM \right )$ and compute first the quantity  
$$\la (\widetilde{\Delta}_E+2E+\delta^{\star}\di\tr )h_2,H_{+} \ra_{L^2}$$ 
by type considerations. We recall that the normalisation of $h_2$ under the gauge group $\mathbf{G}$ in \eqref{gauge-1} reads 
$\delta^g(h_2-h_1^2)=-\tfrac{1}{4}\di\!\tr h_1^2$. Thus 
\begin{equation*}
\begin{split}
(\widetilde{\Delta}_E+2E+\delta^{\star}\di\tr )h_2=&(\Delta_E+2E)(h_2-h_1^2)+(\Delta_E+2E)h_1^2-2 \delta^{\star_g}(\delta^g-\tfrac{1}{4}\di\!\tr)h_1^2\\
=&(\Delta_E+2E)(h_2-h_1^2)+(\widetilde{\Delta}_E+2E+\tfrac{3}{2}\delta^{\star_g}\di\!\tr)h_1^2.
\end{split}
\end{equation*}
Since $h_2-h_1^2$ belongs to $\Gamma \left (\Sym^{2,-}TM \right )$ by \eqref{alg-T1} and since the Einstein operator preserves complex type this leads to 
$\la (\widetilde{\Delta}_E+2E+\delta^{\star}\di\tr )h_2,H_{+} \ra_{L^2}=\la (\widetilde{\Delta}_E+2E+\tfrac{3}{2}\delta^{\star_g}\di\!\tr)h_1^2,H_{+} \ra_{L^2}$. The claim follows now easily.
\end{proof}
All ingredients are now in place in order to make explicit the component on $\Gamma \left (\Sym^{2,-}TM \right )$ of the Einstein equation to third order. We 
consider the tensor 
$$ \tilde{\mathbf{u}}_3:=\tfrac{1}{3}\mathbf{u}_3+\tfrac{3}{2}h_1^3 \ \mathrm{in} \ \Gamma \left (\Sym^{2}TM \right )
$$
and prove the following 
\begin{teo} \label{E3-1} 
Assume that the family $g_t$ is normalised according ot $\delta^gh_1=0$ and \eqref{gauge-1}. 
The component on $\Gamma \left (\Sym^{2,-}TM \right )$ of the Einstein equation to third order in \eqref{xxx3}
reads 
\begin{equation*}
\begin{split}
\la \widetilde{\Delta}_E \tilde{\mathbf{u}}_3+\tfrac{1}{3}\delta^{\star_g}\di\!\tr  \mathbf{u}_3,H\ra_{L^2}=&\bfv(h_1,h_2-h_1^2,H)\\
+&
\tfrac{1}{2} \la \{h_1, \left (\Delta_E+4\ring{R}+\delta^{\star_g,+}\circ (2\delta^g+\tfrac{1}{4}\di\!\tr ) \right )h_1^2\},H \ra_{L^2}\\
-&\la \mathrm{div}_2(h_1),H\ra_{L^2}
\end{split}
\end{equation*}
for all $H \in \Gamma \left (\Sym^{2,-}TM \right )$.
\end{teo}
\begin{proof}This follows by collecting the facts derived in the previous sections by a direct formal computation. Since this process is rather lengthy and involves a large number of terms we clearly indicate the step needed as follows. First recall that 
$$\mathbf{w}(h_1):=2\delta^gQ(h_1)-\mathbf{p}(h_1^2,h_1)-h_1 \circ \mathbf{p}(h_1,h_1).$$
We use Proposition \ref{Q13} to express the quantity $2\delta^gQ(h_1)-\mathbf{p}(h_1^2,h_1)$. After replacing the symmetric component in 
$\delta^g\di_{\nabla^g}\!h_1^2$ by $-\delta^{\star_g}\delta^gh_1^2+(\Delta_E+E+\ring{R})h^2_1$ and also using Lemma \ref{w-step} for the term $h_1 \circ \mathbf{p}(h_1,h_1)$ we obtain 
\begin{equation} \label{aux1}
\begin{split}
\la \mathbf{w}(h_1)+\bfv(h_1,h_1^2),H \ra_{L^2}=&-\la \bfv(h_1,h_1^2)+\widetilde{\Delta}_E h_1^3,H\ra_{L^2}\\
+&\la \{h_1,\left ( \Delta_E+E+2\ring{R}+\tfrac{1}{4}\delta^{\star_g}\di\!\tr\right )h_1^2\},H\ra_{L^2}.
\end{split}
\end{equation}
Next we use Lemma \ref{l-step} and also Theorem \ref{vm-3} for the expression of the component of $\bfv(h_1,h_1^2)$ on $\Gamma \left (\Sym^{2,-}TM \right )$; a straightforward computation using only the formula for $\widetilde{\Delta}_E$ in \eqref{D-tilde} and $\la \widetilde{\Delta}^{-}_Eh_1^3,H\ra_{L^2}=\la \widetilde{\Delta}_Eh_1^3,H\ra_{L^2}$ leads to 
\begin{equation} \label{aux2}
\begin{split}
&-\bfv(h_1,h_1^2,H)-\tfrac{1}{4}\la \{h_1,(\widetilde{\Delta}_E+2E+\delta^{\star}\di\tr )h_2\},H \ra_{L^2}\\
=&-\tfrac{1}{2}
\la \{h_1,\left (\Delta_E-2\delta^{\star_g,+}\delta^g+E+\tfrac{1}{4}\delta^{\star_g,+}\di\!\tr \right )h_1^2\} ,H\ra_{L^2}\\
&-\la \tfrac{1}{2}(\widetilde{\Delta}_E+2E)h_1^3+\mathrm{div}_2(h_1),H \ra_{L^2}.
\end{split}
\end{equation}
To conclude we use Theorem \ref{thm-54} which asserts 
that the component of the third order Einstein equation on $\Gamma \left (\Sym^{2,-}TM \right )$ reads 
\begin{equation*}
\begin{split}
\la \tfrac{1}{3}(\widetilde{\Delta}_E+\delta^{\star_g}\di\!\tr)\mathbf{u}_3,H \ra_{L^2}=&\bfv(h_1,h_2-h_1^2,H)\\
+&\la  \mathbf{w}(h_1)+\bfv(h_1,h_1^2),H \ra_{L^2}-\tfrac{1}{4}\la \{h_1,(\widetilde{\Delta}_E+2E+\delta^{\star}\di\tr )h_2\},H \ra_{L^2}.
\end{split}
\end{equation*}
The claim follows by plugging \eqref{aux1} and \eqref{aux2} into the right hand side of the above equation; after a purely algebraic calculation this yields 
the right hand side in the claim corrected by the term $-\tfrac{3}{2}\la \widetilde{\Delta}_Eh_1^3,H\ra_{L^2}$, which is responsible for the appearance of the perturbation $\tilde{\mathbf{u}}_3$ in the statement. 

\end{proof}
This can be further simplified by using Lemma \ref{3-} in order to express the term $\bfv(h_1,h_2-h_1^2,H)$ in the above theorem. However in order to do so we need to derive first an explicit formula for the $\LL^2$-norm of the Kodaira-Spencer bracket.
\subsection{The $\LL^2$-norm of the Kodaira-Spencer bracket} \label{norm-bra-new}
This section is independent of the deformation theory material previously developed. In order to determine the norm 
$\Vert [h,h]^c \Vert^2_{L^2}$ where $h$ in $\Gamma \left (\Sym^{2,-}TM \right )$
the idea is to compute the quantity $\la \di_{\nabla^g}\!h(h \cdot ,h \cdot), \di_{\nabla^g}\!h \ra_{L^2}$ in two different ways. 
The first one is essentially based on the identity in Proposition \ref{id-main}; it further uses the comparison of the brackets 
$[h,h]^{\FN}$ respectively $[h,h]^c$ and also takes into account the algebraic type of the Kodaira-Spencer bracket. The result is contained in lemma below, where we have isolated the terms involving $\bdel h$. The reason for doing so is two-folded; firstly such terms vanish when $h \in \mathscr{E}(M,g)$. More generally, for elements of the type $h+tH$ where $(h,H) \in \mathscr{E}^{-}(M,g) \oplus\Gamma \left (\Sym^{2,-}TM \right )$ and
$t \in \mathbb{R}$ such terms are polynomial in $t$ and have vanishing coefficient on $t$. This property will will used in the polarisation arguments 
in the next section.

As this will be systematically used in what follows we recall(see \eqref{t-FNKS}) that the splitting of $[h,h]^{\FN}$ where 
$h \in \Gamma \left (\Sym^{2,-}TM \right )$ reads 
\begin{equation} \label{type-N0}
[h,h]^{\FN}=[h,h]^c+h \circ \bdel h+\left ( \di_{\nabla^g}^{+}h^2-h \sharp \bdel h \right )
\end{equation}
according to $\Lambda^2(M,TM)=\lambda^2_{-}(M,TM) \oplus \lambda^2_{+}(M,TM) \oplus \Lambda^{1,1}(M,TM)$.
\begin{lema} \label{N1-new}
Let $h$ belong to $\Gamma \left (\Sym^{2,-}TM \right )$. Then 
\begin{equation*}
\begin{split}
\la \di_{\nabla^g}\!h(h \cdot ,h \cdot), \di_{\nabla^g}\!h \ra_{L^2}=&\Vert \di\!_{\nabla^g}^{-}h^2 \Vert^2_{L^2}-\la \partial^g h^3,\partial^g h \ra_{L^2}\\
+& \la h^2 \circ \di_{\nabla^g}\!h-3 h \circ \di_{\nabla^g}\!h^2, \bdel h\ra_{L^2}\\
+&\la h \circ \di_{\nabla^g}\!h^2, \di_{\nabla^g}\!h \ra_{L^2}.
\end{split}
\end{equation*}
\end{lema}
\begin{proof}
Start from the type decomposition in \eqref{type-N0}. Because $[h,h]^c$ belongs to 
$\lambda^{2}_{-}(M,TM)$ it follows that $h \circ [h,h]^c$ lives in $\lambda^{2}_{+}(M,TM)$ thus its divergence is in 
$\End^{+}TM$; in particular we must have $\la h \circ [h,h]^c, \di_{\nabla^g}\!h \ra_{L^2}=0$. Applying Proposition \ref{id-main} thus leads to 
\begin{equation} \label{new-p0}
\begin{split}
\la \di_{\nabla^g}\!h(h \cdot ,h \cdot), \di_{\nabla^g}\!h \ra_{L^2}=&\la h \sharp \di_{\nabla^g}h^2-\di_{\nabla^g}h^3, \di_{\nabla^g}h\ra_{L^2}+\la h^2 \circ \bdel h,\di_{\nabla^g}h \ra_{L^2}\\
+&\la h\circ \di_{\nabla^g}^{+}h^2,\di_{\nabla^g}h \ra_{L^2}-\la h \circ \bdel h, h \sharp \di_{\nabla^g}h \ra_{L^2}.
\end{split}
\end{equation}
To obtain the last summand above we have also used symmetry considerations which show that 
$\la h \circ (h \sharp \bdel h), \di_{\nabla^g}h \ra_{L^2}=
\la h \sharp ( h \circ \bdel h), \di_{\nabla^g}h \ra_{L^2}=\la h \circ \bdel h, h \sharp \di_{\nabla^g}h \ra_{L^2}$. 
The terms containing the $\sharp$-action may be further simplified by taking into account that $h \sharp \di_{\nabla^g}h=\di_{\nabla^g}h^2-[h,h]^{\FN}$ and also the algebraic type of the complex bracket. Indeed
\begin{equation*}
\begin{split}
\la h \sharp \di_{\nabla^g}h^2, \di_{\nabla^g}h\ra_{L^2}=&\la \di_{\nabla^g}h^2, h \sharp \di_{\nabla^g}h\ra_{L^2}=
\Vert \di_{\nabla^g} h^2 \Vert^{2}_{L^2}-\la \di_{\nabla^g} h^2, [h,h]^{\FN}\ra_{L^2}.
\end{split}
\end{equation*}
Next we use once again the decomposition of the bracket $[h,h]^{\FN}$ into complex types. Since $\delta^g[h,h]^c$ belongs to 
$\Sym^{2,-}TM$ we obtain that $\la \di_{\nabla^g} h^2, [h,h]^{c}\ra_{L^2}=0$. Thus
\begin{equation*}
\begin{split}
\la \di_{\nabla^g} h^2, [h,h]^{\FN}\ra_{L^2}=&\la \di_{\nabla^g} h^2,h \circ \bdel h+\di_{\nabla^g}^{+}h^2-h \sharp \bdel h\ra_{L^2}\\
=&
\la h \circ \di_{\nabla^g} h^2, \bdel h \ra_{L^2}+\Vert \di_{\nabla^g}^{+}h^2\Vert^2_{L^2}-\la h \sharp \bdel h,\di_{\nabla^g}h^2\ra_{L^2}.
\end{split}
\end{equation*}
However, since $h$ is $J$-anti-invariant we have that $h \sharp \di^{-}_{\nabla^g}h^2$ belongs to the space 
$\Omega^{1,1}(M,TM)$ which is orthogonal to $\lambda^2_{-}(M,TM)$. It follows that 
$$\la h \sharp \bdel h,\di_{\nabla^g}h^2\ra_{L^2}=\la \bdel h,h \sharp \di_{\nabla^g}h^2\ra_{L^2}=\la \bdel h,h \sharp \di^{+}_{\nabla^g}h^2\ra_{L^2}=\la \bdel h^3,\bdel h\ra_{L^2}-\la h^2 \circ \bdel h, \bdel h \ra_{L^2}$$
after taking Corollary \ref{ids-2} with $H=h$ into account. Gathering terms we arrive at 
\begin{equation} \label{new-p1}
\begin{split}
\la h \sharp \di_{\nabla^g}h^2, \di_{\nabla^g}h\ra_{L^2}=&
\Vert \di_{\nabla^g} h^2 \Vert^{2}_{L^2}-\Vert \di^{+}_{\nabla^g} h^2 \Vert^{2}_{L^2}+\la \bdel h^3,\bdel h\ra_{L^2}\\
-&\la h \circ \di_{\nabla^g}\!h^2+h^2 \circ \di_{\nabla^g}\!h, \bdel h\ra_{L^2}.
\end{split}
\end{equation}
Similarly, we have $\la h \circ \bdel h, h \sharp \di_{\nabla^g}h \ra_{L^2}=\la h \circ \bdel h,\di_{\nabla^g}h^2 \ra_{L^2}-
\la h \circ \bdel h, [h,h]^{\FN}\ra_{L^2}$. We record that  $h \circ \bdel h$ belongs to $\lambda^2_{+}(M,TM)$ and use the type decomposition of $[h,h]^{\FN}$; namely $[h,h]^c$ belongs to $\lambda^2_{-}(M,TM)$ whilst $\di_{\nabla^g}^{+}h^2-h \sharp \bdel h$ lies in $\Omega^{1,1}(M,TM)$. Since the latter spaces are both orthogonal to $\lambda^2_{+}(M,TM)$ we end up with 
\begin{equation} \label{new-p2}
\la h \circ \bdel h, h \sharp \di_{\nabla^g}h \ra_{L^2}=\la h \circ \bdel h,\di_{\nabla^g}h^2 \ra_{L^2}-\la h^2 \circ \bdel h, \bdel h\ra_{L^2}.
\end{equation}
The claim follows now by plugging equations \eqref{new-p1} and \eqref{new-p2} into \eqref{new-p0} whilst also observing that type considerations make that 
$\la h\circ \di_{\nabla^g}^{+}h^2,\di_{\nabla^g}h \ra_{L^2}=\la h\circ \di_{\nabla^g}h^2,\di_{\nabla^g}h-\bdel h \ra_{L^2}$ as well as 
$\Vert \di_{\nabla^g}h \Vert^2_{L^2}=\Vert \di_{\nabla^g}^{+}h \Vert^2_{L^2}+\Vert \di_{\nabla^g}^{-}h \Vert^2_{L^2}$.
\end{proof}
Observe that the first two summands in the above formula come essentially from differential operators acting on $h^2$ respectively operators 
involving $\bdel h$. The only quantity in Lemma \ref{N1-new} not amenable to such an expression so far is 
$\la h \circ \di_{\nabla^g}\!h^2, \di_{\nabla^g}\!h \ra_{L^2}$.
Next we provide 
a second, direct way, of computing the quantity 
$\la \di_{\nabla^g}\!h(h \cdot ,h \cdot), \di_{\nabla^g}\!h \ra_{L^2}$ and eventually eliminate $\la h \circ \di_{\nabla^g}\!h^2, \di_{\nabla^g}\!h \ra_{L^2}$. We start with the following preliminary
\begin{lema} \label{AAL}Assume that $h \in \Gamma \left (\Sym^{2,-}TM \right )$. Then
\begin{equation*}
\begin{split}
 \la (\nabla^g_{e_k}h)he_i, h(\nabla^g_{e_i}h)e_k \ra_{L^2}=&\la \di_{\nabla^g}h,h^2 \circ \di_{\nabla^g}h-h \circ \di_{\nabla^g}h^2 \ra_{L^2}\\
+&\la \mathrm{L}(h,h^2),h\ra_{L^2}-\la \mathrm{L}(h,h),h^2\ra_{L^2}.
\end{split}
\end{equation*}
\end{lema}
\begin{proof}
By essentially using Leibniz's rule we compute 
\begin{equation*}
\begin{split}
\la (\nabla^g_{e_k}h)he_i, h(\nabla^g_{e_i}h)e_k \ra_{L^2}=&\la (\nabla^g_{e_k}h^2)e_i, h(\nabla^g_{e_i}h)e_k \ra_{L^2}-
\la (\nabla^g_{e_k}h)e_i, h^2(\nabla^g_{e_i}h)e_k \ra_{L^2}\\
=&\la \di_{\nabla^g}h^2(e_k,e_i), h(\nabla^g_{e_i}h)e_k \ra_{L^2}+\la \nabla^g_{e_i}h^2, h \nabla^g_{e_i}h \ra_{L^2}\\
-&\la \di_{\nabla^g}h(e_k,e_i), h^2(\nabla^g_{e_i}h)e_k \ra_{L^2}-\la \nabla^g_{e_i}h, h^2 \nabla^g_{e_i}h \ra_{L^2}\\
=&-\la \di_{\nabla^g}h^2,h \circ \di_{\nabla^g}h \ra_{L^2}+\la \di_{\nabla^g}h,h^2 \circ \di_{\nabla^g}h \ra_{L^2}\\
+&\la \nabla^g_{e_i}h^2, h \nabla^g_{e_i}h \ra_{L^2}-\la \nabla^g_{e_i}h, h^2 \nabla^g_{e_i}h \ra_{L^2}.
\end{split}
\end{equation*}
Because $h$(and hence $\nabla_{e_i}h$) is symmetric we have $\la \nabla^g_{e_i}h^2, h \nabla^g_{e_i}h \ra_{L^2}=\la \nabla^g_{e_i}h^2 \circ \nabla^g_{e_i}h,h \ra_{L^2}=\la \mathrm{L}(h,h^2),h\ra_{L^2}$; similarly $\la \nabla^g_{e_i}h, h^2 \nabla^g_{e_i}h \ra_{L^2}=\la \nabla^g_{e_i}h \circ \nabla^g_{e_i}h,h^2 \ra_{L^2}=\la \mathrm{L}(h,h),h^2\ra_{L^2}$ and the claim is proved.

\end{proof}
Based on this we obtain a second expression for $\la \di_{\nabla^g}\!h(h \cdot ,h \cdot), \di_{\nabla^g}\!h \ra_{L^2}$ by direct computation, as follows.
\begin{pro} \label{N2-new+}
Let $h$ belong to $\Gamma \left (\Sym^{2,-}TM \right )$. Then 
\begin{equation*}
\begin{split}
\la \di_{\nabla^g}\!h(h \cdot ,h \cdot), \di_{\nabla^g}\!h \ra_{L^2}=&\Vert \di_{\nabla^g}\!h^2 \Vert^2_{L^2}+\la \partial^gh^3,\partial^g h\ra_{L^2}-\la \mathrm{L}(h,h),h^2 \ra_{L^2}\\
+&\la h \circ \di_{\nabla^g}h^2+h^2 \circ \di_{\nabla^g}h, \bdel h\ra_{L^2}\\
-&3\la h \circ \di_{\nabla^g}h^2,\di_{\nabla^g}\!h \ra_{L^2}.
\end{split}
\end{equation*}
\end{pro}
\begin{proof}
We use the shorthand notation $D_1:=\di_{\nabla^g}\!h(h \cdot ,h \cdot)$ and split $D_1=\alpha_1-\alpha_2$; explicitly 
$\alpha_1(X,Y)=(\nabla^g_{hX}h)hY$ and $\alpha_2(X,Y)=(\nabla^g_{hY}h)hX$. Below we compute the divergence of $\alpha_2$; using the Ricci identity shows that  
\begin{equation*}
\begin{split}
(\delta^g\alpha_2)_X=&-((\nabla^g)^2_{e_i,hX}h)he_i-\left (\nabla^g_{(\nabla^g_{e_i}h)X}h \right )he_i-(\nabla^g_{hX}h)(\nabla^g_{e_i}h)e_i\\
=&[R(e_i,hX),h]he_i-((\nabla^g)^2_{hX,e_i}h)he_i-\left (\nabla^g_{(\nabla^g_{e_i}h)X}h \right )he_i+(\nabla^g_{hX}h)\delta^gh.
\end{split}
\end{equation*}
Now $((\nabla^g)^2_{hX,e_i}h)he_i=
\nabla^g_{hX} \left ((\nabla^g_{e_i}h)he_i \right )-(\nabla^g_{e_i}h)(\nabla^g_{hX}h)e_i$ and also the curvature term reads 
$[R(e_i,hX),h]he_i=(\ring{R}h^2-h \circ \ring{R}h)hX$. After re-arranging terms 
we find 
\begin{equation*}
\begin{split}
(\delta^g\alpha_2)_X=&(\ring{R}h^2-h \circ \ring{R}h)hX+(\nabla^g_{e_i}h)(\nabla^g_{hX}h)e_i-\left (\nabla^g_{(\nabla^g_{e_i}h)X}h \right )he_i\\
+&(\nabla^g_{hX}h)\delta^gh-\nabla^g_{hX}\xi
\end{split}
\end{equation*}
where $\xi:=(\nabla^g_{e_i}h)he_i$. Since $\xi=h\delta^gh-\delta^gh^2$ the divergence terms above can be computed from 
\begin{equation*}
\begin{split}
(\nabla^g_{hX}h)\delta^gh-\nabla^g_{hX}\xi=\nabla^g_{hX}(h \delta^gh-\xi)-h\nabla^g_{hX}\delta^gh=\nabla^g_{hX}\delta^gh^2-
h \nabla^g_{hX}\delta^gh.
\end{split}
\end{equation*}
Recording that $(\nabla^g_{e_i}h)(\nabla^g_{hX}h)e_i=\mathbf{i}(h,h)(hX,\cdot)$ we arrive at 
\begin{equation} \label{AA0}
\begin{split}
(\delta^g\alpha_2)_X=&(\ring{R}h^2-h \circ \ring{R}h)hX+\mathbf{i}(h,h)(hX,\cdot)-\left (\nabla^g_{(\nabla^g_{e_i}h)X}h \right )he_i\\
+&\nabla^g_{hX}\delta^gh^2-
h \nabla^g_{hX}\delta^gh.
\end{split}
\end{equation}
We now record that 
\begin{equation} \label{AA1}
\begin{split}
g( \left (\nabla^g_{(\nabla^g_{e_i}h)e_j}h \right )he_i,he_j )= g((\nabla^g_{e_i}h)e_j,e_k)g(\nabla^g_{e_k}h)he_i,he_j)=
g(\nabla^g_{e_k}h)he_i,h(\nabla^g_{e_i}h)e_k).
\end{split}
\end{equation}
This is essentially granted by the symmetry of $h$. Also 
\begin{equation} \label{AA2}
\begin{split}
\la \nabla^g_{he_i}\delta^gh^2-
h \nabla^g_{he_i}\delta^gh, he_i \ra_{L^2}=&\la \nabla^g_{e_i}\delta^gh^2, h^2e_i\ra_{L^2}-\la \nabla^g_{e_i}\delta^gh, h^3e_i\ra_{L^2}\\
=&\Vert \delta^gh^2 \Vert^2_{L^2}-\la \delta^gh, \delta^gh^3 \ra_{L^2}.
\end{split}
\end{equation}
Take the $L^2$-scalar product with $h$ in \eqref{AA0} by taking into account \eqref{AA1}, \eqref{AA2} and also 
Lemma \ref{AAL} in order to express the right hand side in \eqref{AA1}; we find 
\begin{equation*} 
\begin{split}
\la \delta^g\alpha_2,h \ra_{L^2}=&\la \ring{R}h^2-h \circ \ring{R}h,h^2 \ra_{L^2}\\
+& \la \mathbf{i}(h,h),h^2 \ra_{L^2}+\la \di_{\nabla^g}h,-h^2 \circ \di_{\nabla^g}h+h \circ \di_{\nabla^g}h^2 \ra_{L^2}\\
-&\la \mathrm{L}(h,h^2),h\ra_{L^2}+\la \mathrm{L}(h,h),h^2\ra_{L^2}+\Vert \delta^gh^2 \Vert^2_{L^2}-\la \delta^gh, \delta^gh^3 \ra_{L^2}. 
\end{split}
\end{equation*}
In order to bring this to final form we use the following facts. Firstly 
$$\la \mathbf{i}(h,h),h^2 \ra_{L^2}=\la h \circ \di_{\nabla^g}h,\di_{\nabla^g}h^2\ra_{L
^2}+\mathbf{q}(h,h^3)-\la h,\mathrm{L}(h^2,h)
\ra_{L^2}
$$ according to Lemma \ref{i-1} and also $\la h \circ \di_{\nabla^g}h,\di_{\nabla^g}h^2 \ra_{L^2}=\la  \di_{\nabla^g}h,h\circ \di_{\nabla^g}h^2 \ra_{L^2}$ since $h$ is symmetric. Secondly, $\mathbf{q}(h,h^3)=\la (\ring{R}-E)h,h^3\ra_{L^2}
+\la \delta^gh, \delta^gh^3\ra_{L^2}$ by Lemma \ref{q}. Thirdly, Lemma \ref{L-exp} ensures that 
$2\la \mathrm{L}(h^2,h),h \ra_{L^2}=\la (\Delta_E+2\ring{R})h^2,h^2 \ra_{L^2}$. A straightforward algebraic calculation based on these facts leads to 
\begin{equation*} 
\begin{split}
\la \delta^g\alpha_2,h \ra_{L^2}=&3\la h \circ \di_{\nabla^g}h^2,\di_{\nabla^g}h \ra_{L^2}-
\la h \circ \di_{\nabla^g}h^2+h^2 \circ \di_{\nabla^g}h,\di_{\nabla^g}h \ra_{L^2}\\
+&\la \mathrm{L}(h,h),h^2\ra_{L^2}-\la (\Delta_E+\ring{R}+E)h^2,h^2 \ra_{L^2}+\Vert \delta^gh^2 \Vert^2_{L^2}.
\end{split}
\end{equation*}
Since $\alpha_1$ is a section of  $\Lambda^1M \otimes \Lambda^{1,1}M$ it follows that $\delta^g\alpha_1$ is $J$-invariant, in particular $\la \delta^g\alpha_1,h \ra_{L^2}=0$. The claim follows now from Lemma \ref{quad-K}.
\end{proof}
After taking linear combinations in Lemma \ref{N1-new} respectively Proposition \ref{N2-new+} in order to eliminate 
$\la h \circ \di_{\nabla^g}h^2,\di_{\nabla^g}\!h \ra_{L^2}$ we obtain the following explicit 
\begin{coro} \label{N3-new}
Assume that $h$ belongs to $\Gamma \left (\Sym^{2,-}TM \right )$. Then 
\begin{equation*}
\begin{split}
4\la \di_{\nabla^g}\!h(h \cdot ,h \cdot), \di_{\nabla^g}\!h \ra_{L^2}=&\Vert \di_{\nabla^g}h^2\Vert^2_{L^2}+3\Vert \di^{-}_{\nabla^g}h^2\Vert^2_{L^2}-2\la \partial^gh^3,\partial^gh  \ra_{L^2}-\la \mathrm{L}(h,h),h^2\ra_{L^2}\\
+&4\la h^2 \circ \di_{\nabla^g}\!h-2 h \circ \di_{\nabla^g}\!h^2,\bdel h\ra_{L^2}
\end{split}
\end{equation*}
\end{coro}
The terms on the first displayed line above can be further simplified by using the Weitzenb\"ock formulas in section \ref{K-g}. However this will be only done at the end of the section, in order to keep the presentation easy to follow. 
\begin{rema} \label{rest-id}
By combining Lemma \ref{N1-new} respectively Proposition \ref{N2-new+} we can also obtain an explicit expression for 
$\la h \circ \di_{\nabla^g}h^2,\di_{\nabla^g}\!h \ra_{L^2}$ by a straightforward algebraic computation which leads to
\begin{equation*}
4\la h \circ \di_{\nabla^g}h^2,\partial^g h \ra_{L^2}=\Vert \di^{+}_{\nabla^g}h^2 \Vert^2_{L^2}+2\la 
\partial^g h^3,\partial^g h \ra_{L^2}-\la \mathrm{L}(h,h),h^2 \ra_{L^2}
\end{equation*}
for all $h$ in $\Gamma \left (\Sym^{2,-}TM \right )$. 
However this will not be needed in this paper hence we omit additional details.
\end{rema}
The final arguments needed in order to compute the norm of the Kodaira-Spencer bracket consists in comparing $\la \di_{\nabla^g}\!h(h \cdot ,h \cdot), \di_{\nabla^g}\!h \ra_{L^2}$ and $\Vert [h,h]^c \Vert^2$. This is done in two steps in the following 
\begin{pro} \label{N4-new}
Assume that $h$ belongs to $\Gamma \left (\Sym^{2,-}TM \right )$. The following hold
\begin{itemize}
\item[(i)] the norm of the Fr\"olicher-Nijenhuis bracket compares to that of the complex bracket according to
\begin{equation} \label{normb-N}
\begin{split}
\Vert [h,h]^{\FN}\Vert^2_{L^2}=&\Vert [h,h]^c \Vert^2_{L^2}
+\Vert \di_{\nabla^g}^{+}h^2 \Vert^2_{L^2}-2\la \bdel h^3, \bdel h\ra_{L^2}\\
+&\la 3 h^2 \circ \bdel h+h^2 \sharp \bdel h,\bdel h \ra_{L^2}
\end{split}
\end{equation}
\item[(ii)] we have 
\begin{equation*}
\begin{split}
2\la \di_{\nabla^g}\!h(h \cdot ,h \cdot), \di_{\nabla^g}\!h \ra_{L^2}=&\Vert [h,h]^{\FN} \Vert^2_{L^2}-2\la \delta^g[h,h]^{\FN},h^2\ra_{L^2}+\Vert \di_{\nabla^g}\!h^2 \Vert^2_{L^2}\\
-&\la h^2 \sharp \di_{\nabla^g}\!h, \di_{\nabla^g}\!h \ra_{L^2}.
\end{split}
\end{equation*}
\end{itemize}
\end{pro}
\begin{proof}
(i) Taking the norm in \eqref{type-N0} shows that 
\begin{equation*}
\begin{split}
\Vert [h,h]^{\FN}\Vert^2_{L^2}=&\Vert [h,h]^c \Vert^2_{L^2}+\la h^2 \circ \bdel h,\bdel h \ra_{L^2}
+\Vert \di_{\nabla^g}^{+}h^2 \Vert^2_{L^2}-2\la h \sharp \bdel h, \di_{\nabla^g}^{+}h^2\ra_{L^2}
+\Vert h \sharp \bdel h \Vert^2_{L^2}.
\end{split}
\end{equation*}
Using Corollary \ref{ids-2} with $H=h$ yields $\la h \sharp \bdel h, \di_{\nabla^g}^{+}h^2\ra_{L^2}=\la \bdel h^3, \bdel h\ra_{L^2}-\la h^2 \circ \bdel h,\bdel h \ra_{L^2}$. Furthermore 
\begin{equation*}
\Vert h \sharp \bdel h \Vert^2_{L^2}=\la \bdel h, h \sharp (h \sharp \bdel h )\ra_{L^2}=\la \bdel h, h^2 \sharp \bdel h+2\bdel h(h \cdot, h \cdot)\ra_{L^2}=\la \bdel h, h^2 \sharp \bdel h\ra_{L^2}.
\end{equation*}
In the last line we have used that $\bdel h(h \cdot, h \cdot)$ and $\bdel h$ are orthogonal since 
they belong to $\lambda^2_{+}(M,TM)$ respectively $\lambda^2_{-}(M,TM)$. The claim follows by gathering these facts.
 \\
(ii)This identity does not use the fact that the metric $g$ is K\"ahler, nor that 
$h \in \Gamma \left ( \Sym^{2,-}TM \right )$. For notational convenience write again $\mathrm{D}_1:=\di_{\nabla^g}\!h(h \cdot ,h \cdot)$. Start from the algebraic identity 
$h \sharp (h \sharp \di_{\nabla^g}\!h)=h^2 \sharp \di_{\nabla^g}\!h+2\mathrm{D}_1$. Using \eqref{bra-nac} this yields further 
\begin{equation} \label{D12-n}
2D_1=h \sharp \di_{\nabla^g}\!h^2-h^2 \sharp \di_{\nabla^g}\!h-h \sharp [h,h]^{\FN}.
\end{equation}
Now use again that $h \sharp \di_{\nabla^g}\!h=\di_{\nabla^g}\!h^2 -[h,h]^{\FN}$ and that the operator $h \sharp \cdot$ is symmetric. The claim then follows from taking the scalar product with $\di_{\nabla^g}\!h$ in \eqref{D12-n}.
\end{proof}
As a consequence we can relate the quantity $\la \di_{\nabla^g}\!h(h \cdot ,h \cdot), \di_{\nabla^g}\!h \ra_{L^2}$ to the norm of the complex bracket as follows.
\begin{coro} \label{corN5-nw}Assume that $h$ belongs to $\Gamma \left (\Sym^{2,-}TM \right )$.Then 
\begin{equation*}
\begin{split}
2\la \di_{\nabla^g}\!h(h \cdot ,h \cdot), \di_{\nabla^g}\!h \ra_{L^2}=&\Vert [h,h]^c \Vert^2_{L^2}+\Vert \di_{\nabla^g}^{-}h^2 \Vert^2_{L^2}-\la \mathbf{D}h,h^2\ra_{L^2}\\
+&\la  h^2 \circ \bdel h, \bdel h \ra_{L^2}-4\la  h \circ \bdel h, \di_{\nabla^g}h^2\ra_{L^2}-\la h^2 \sharp \bdel h,\bdel h \ra_{L^2}\\
+&2 \la \bdel h, \delta^g h^2 \wedge h \ra_{L^2}.
\end{split}
\end{equation*}
\end{coro}
\begin{proof}
Firstly, using Proposition \ref{hash-o} with $H_{+}=h^2$ and $H_{-}=h$ yields 
\begin{equation*}
\begin{split}
\la h^2 \sharp \di_{\nabla^g}h, \di_{\nabla^g}h \ra_{L^2}=&\la \mathbf{D}h ,h^2\ra_{L^2}
+2\la h^2 \sharp \bdel h, \bdel h \ra_{L^2}+2 \la \delta^g (h \circ \bdel h ), h^2\ra_{L^2}
-2 \la \bdel h, \delta^gh^2 \wedge h\ra_{L^2}.
\end{split}
\end{equation*}
Secondly, using that $h \sharp \di_{\nabla^g}h=\di_{\nabla^g}h^2-[h,h]^{\FN}$ we see that 
equation \eqref{new-p1} reads 
\begin{equation*}
\la \delta^g[h,h]^{\FN},h^2\ra_{L^2}=\Vert \di^{+}_{\nabla^g}h^2 \Vert^2_{L^2}+\la h \circ \bdel h ,\di_{\nabla^g} h^2\ra_{L^2}+
\la h^2 \circ \bdel h ,\bdel h\ra_{L^2}-\la \bdel h^3, \bdel h \ra_{L^2}.
\end{equation*}
Now we plug these facts as well as the expression for 
$\Vert [h,h]^{\FN}\Vert^2_{L^2}$ given by \eqref{normb-N} into part (ii) of Lemma \ref{N4-new}. The claim follows by a purely algebraic computation. 
\end{proof}
Assembling the preparatory material developed so far eventually leads to the proof of the main result in this section, namely the explicit expression for the $L^2$-norm of the complex bracket.
\begin{teo} \label{N5}
Assume that $h \in \Gamma \left (\Sym^{2,-}TM \right )$. Then 
\begin{equation*}
\begin{split}
2 \Vert [h,h]^c\Vert^2_{L^2}=&\la (\Delta_E+4\ring{R})h^2
+\delta^{\star_g}(2\delta^gh^2+\di\!\tr h^2),h^2\ra_{L^2}-
2\la \Delta_Eh,h^3\ra_{L^2}\\
+&2\la h^2 \circ \bdel h, \bdel h\ra_{L^2}+2\la h^2 \sharp \bdel h, \bdel h\ra_{L^2}\\
+&2\la \mathrm{div}_1(h,h),h^2 \ra_{L^2}-4 \la \bdel h, \delta^gh^2 \wedge h \ra_{L^2}.
\end{split}
\end{equation*} 
\end{teo}
\begin{proof}
By comparing the expressions for  $\la \di_{\nabla^g}\!h(h \cdot ,h \cdot), \di_{\nabla^g}\!h \ra_{L^2}$ found in Corollary \ref{N3-new} respectively Corollary \ref{corN5-nw} we find 
\begin{equation*}
2 \Vert [h,h]^c\Vert^2_{L^2}=\mathrm{p}(h)+2\la h^2 \circ \bdel h, \bdel h\ra_{L^2}+2\la h^2 \sharp \bdel h, \bdel h\ra_{L^2}-4 \la \bdel h, \delta^gh^2 \wedge h \ra_{L^2}
\end{equation*}
where $\mathrm{p}(h)=2 \la \mathbf{D}h,h^2\ra_{L^2}+\Vert \di_{\nabla^g}h^2\Vert^2_{L^2}+\Vert \di_{\nabla^g}^{-}h^2\Vert^2_{L^2}-2
\la \partial^gh^3, \partial^g h\ra_{L^2}-\la \mathrm{L}(h,h),h^2\ra_{L^2}$. Thus in order to prove the claim there remains to 
bring the auxiliary expression $\mathrm{p(}h)$ to its final form. In doing so we first collect the main ingredients needed which are 
\begin{itemize}
\item[$\bullet$] $\Vert \di_{\nabla^g}h^2\Vert^2_{L^2}=\la (\Delta_E+E+\ring{R})h^2,h^2\ra_{L^2}-\Vert \delta^gh^2\Vert^2_{L^2}$ by \eqref{wz1}
\item[$\bullet$] $\Vert \di_{\nabla^g}^{-}h^2\Vert^2_{L^2}=\la (\tfrac{1}{2}\Delta_E+E)h^2,h^2\ra_{L^2}-\Vert \delta^g h^2 \Vert^2_{L^2}$ by using part (i) in Corollary \ref{wz-K2}, with $H=h^2$
\item[$\bullet$] $\la \partial^gh^3, \partial^g h\ra_{L^2}=\la (\tfrac{1}{2}\Delta_E+E
+\ring{R})h^3,h\ra_{L^2}$ by using part (ii) in Corollary \ref{wz-K2}
\item[$\bullet$] $\la \mathrm{L}(h,h),h^2\ra_{L^2}=\la (\Delta_E+2\ring{R})h,h^3\ra_{L^2}-\tfrac{1}{2}\la (\Delta_E+2\ring{R})h^2,h^2 \ra_{L^2}$ by Lemma \ref{L-exp}.
\end{itemize}
We also take into account the expression for $\mathbf{D}(h)$ as given in section \ref{v12} namely 
$$ 2\mathbf{D}(h)=2\{h, \ring{R}h\}-(\Delta_E-2\ring{R})h^2+2\delta^{\star_g}(2\delta^g+\tfrac{1}{2}\di\!\tr)h^2+2\mathrm{div}_1(h,h).
$$
After gathering all these facts a straightforward algebraic computation shows that 
$$ \mathrm{p}(h)=\la (\Delta_E+4\ring{R})h^2
+\delta^{\star_g}(2\delta^gh^2+\di\!\tr h^2),h^2\ra_{L^2}-
2\la \Delta_Eh,h^3\ra_{L^2}+2\la \mathrm{div}_1(h,h),h^2 \ra_{L^2}
$$
and the claim is fully proved.
\end{proof}
For instances when $\bdel h=0$ and $\delta^gh=0$ the formula for the norm of the Kodaira-Spencer simplifies considerably and we thus obtain the following 
\begin{pro} \label{N3}
Assume that $h \in \Gamma \left (\Sym^{2,-}TM \right )$ satisfies $\bdel h=0$ and $\delta^gh=0$. Then
\begin{equation*}
2 \Vert [h,h]^c\Vert^2_{L^2}=\la (\Delta_E+4\ring{R})h^2,h^2\ra_{L^2}+\la \delta^{\star_g}(2\delta^gh^2+\di\!\tr h^2),h^2\ra_{L^2}.
\end{equation*} 
\end{pro}
Strickingly, the $L^2$-norm of $[h,h]^c$ is expressed in terms of an operator acting solely on $h^2$. We also note that the formula 
for the norm of the complex bracket is needed in full generality since the final aim is to polarise it as indicated below.
\begin{coro} \label{N7}
Assume that $h \in \Gamma \left (\Sym^{2,-}TM \right )$ satisfies $\bdel h=0$ and $\delta^gh=0$. Then
\begin{equation*}
\begin{split}
4 \la [h,h]^c, [h,H]^c \ra_{L^2}=& \la \{h,\left (\Delta_E+4\ring{R}+\delta^{\star_g,+}(2\delta^g+\tfrac{1}{2}\di\!\tr ) \right )h^2\}-\widetilde{\Delta}_Eh^3,H\ra_{L^2}\\
-&2\la \mathrm{div}_2(h),H\ra_{L^2}
\end{split}
\end{equation*}
for all $H \in \Gamma \left (\Sym^{2,-}TM \right )$.
\end{coro}
\begin{proof}
The idea is to polarise the equation in Theorem \ref{N5} as follows. We use the latter theorem for $h+tH$ and determine the coefficient 
of $t$ for each quantity. Because the Kodaira-Spencer bracket is symmetric we see that $\frac{\di}{\di\!t}_{\vert t=0}\Vert [h+tH,h+tH]^c\Vert^2_{L^2}=4\la [h,h]^c, [h,H]^c \ra_{L^2}.$ We record, for repeated use in the subsequent, that 
$$ \frac{\di}{\di\!t}_{\vert t=0}(h+tH)^2=\{h,H\}.$$
Since the linear operator $\Delta_E+4\ring{R}+2\delta^{\star_g}\delta^g$ is self-adjoint we thus obtain
$$\frac{\di}{\di\!t}_{\vert t=0}\la (\Delta_E+4\ring{R}+2\delta^{\star_g}\delta^g)(h+tH)^2,(h+tH)^2 \ra_{L^2}=2
\la \{h,(\Delta_E+4\ring{R}+2\delta^{\star_g}\delta^g)h^2\},H \ra_{L^2}.
$$
Further on, since $\delta^{\star_g}\di\tr$ is a linear operator, differentiating show that 
\begin{equation*}
\begin{split}
\frac{\di}{\di\!t}_{\vert t=0}\la \delta^{\star_g}\di\tr(h+tH)^2, (h+tH)^2\ra_{L^2}=&\la \delta^{\star_g}\di\tr h^2,\{h,H\} \ra_{L^2}+\la \delta^{\star_g}\di \tr \{h,H\},h^2\ra_{L^2}\\
=&\la \{h,\delta^{\star_g}\di \tr h^2\}+2(\di^{\star}\delta^g h^2)h,H \ra_{L^2}.
\end{split}
\end{equation*}
To deal with the cubic terms we use that $\Delta_Eh=0$ to arrive at 
\begin{equation*}
\begin{split}
\frac{\di}{\di\!t}_{\vert t=0}\la \Delta_E(h+tH),(h+tH)^3 \ra_{L^2}=\la \Delta_EH, h^3 \ra_{L^2}=\la \Delta_Eh^3, H \ra_{L^2}.
\end{split}
\end{equation*}
Taking into account that $\bdel h=0$ we see that the terms 
\begin{equation*}
\begin{split}
&\la (h+tH)^2 \circ \bdel (h+tH), \bdel (h+tH) \ra_{L^2}=t^2\la (h+tH)^2 \circ \bdel H, \bdel H \ra_{L^2}\\
&\la (h+tH)^2 \sharp \bdel (h+tH), \bdel (h+tH) \ra_{L^2}=t^2\la (h+tH)^2 \sharp \bdel H, \bdel H \ra_{L^2}
\end{split}
\end{equation*}
have vanishing coefficient on $t$. Finally we deal with the pure divergence terms as follows; we have 
\begin{equation*}
\begin{split}
\frac{\di}{\di\!t}_{\vert t=0}\la \mathrm{div}_1(h+tH,h+tH),(h+tH)^2 \ra_{L^2}=&\la \mathrm{div}_1(h,h), \{h,H\}\ra_{L^2}+
2\la \mathrm{div}_1(h,H),h^2 \ra_{L^2}\\
=&2\la \mathrm{div}_1(h,H),h^2 \ra_{L^2}
\end{split}
\end{equation*}
after also taking into account that $\delta^gh=0$. Similar arguments yield by expansion and using that 
$\bdel h=0$ the equality 
$$\frac{\di}{\di\!t}_{\vert t=0}\la \bdel(h+tH), \delta^g(h+tH)^2 \wedge (h+tH) \ra_{L^2}=\la \bdel H, \delta^g h^2 \wedge h \ra_{L^2}.
$$
Summarising, we have showed that
\begin{equation*}
\begin{split}
4 \la [h,h]^c, [h,H]^c \ra_{L^2}=& \la \{h,(\Delta_E+4\ring{R})h^2+\delta^{\star_g,+}(2\delta^g+\tfrac{1}{2}\di\!\tr )h^2\}-\Delta_Eh^3
,H\ra_{L^2}\\
+&2\la \mathrm{div}_1(h,H),h^2\ra_{L^2}-2 \la \bdel H, \delta^gh^2 \wedge h \ra_{L^2}+\la (\di^{\star}\delta^gh^2)h
,H\ra_{L^2}.
\end{split}
\end{equation*}
Since $h$ is divergence free we have $\la \mathrm{div}_1(h,H),h^2\ra_{L^2}=\la \delta^{\star_{g}}\delta^gh^3,H\ra_{L^2}-
\la \delta^{\star_g}(h \delta^gH),h^2\ra_{L^2}$ by \eqref{div-1} and the claim follows easily from the definition of the operators 
$\mathrm{div}_2$(see \eqref{div-2}) respectively the perturbed Einstein operator $\widetilde{\Delta}_E$.

\end{proof}


\subsection{Obstructions to third order deformation } \label{obs-3K}
In this section we will use the symmetric part of the Einstein equation to third order obtained in Theorem \ref{E3-1} in order 
to derive in an explicit way the obstruction to deformation. This is based on the elliptic theory of the perturbed Einstein operator 
$\widetilde{\Delta}_E$. 

In order to conclude we need to determine the term containing the tensor $h_2$ in Theorem \ref{thm-54}, equation \eqref{xxx3}. To this extent we use the Einstein deformation theory to second order, see section \ref{rev-O2} for definitions 
and main facts. Throughout this section we assume that the Einstein deformation $g_t \in \mathscr{M}_1$ has Taylor series expansion $g^{-1}g_t=\id+th_1+\tfrac{t^2}{2!}h_2+\tfrac{t^3}{3!}h_3+o(t^4)$ normalised according to $\delta^gh_1=0$ and \eqref{gauge-1}.
The last preparatory fact which is needed is the following 
\begin{lema} \label{v-fin}
The tensors $h_2-h_1^2$ and $h_1$ satisfy the following 
\begin{itemize}
\item[(i)] we have 
\begin{equation*}
\begin{split}
\bfv(h_1,h_2-h_1^2,H)=&2 \left ( \la [h_1,h_2-h_1^2]^c, \bdel H\ra_{L^2}+ \la [h_1,H]^c, \bdel \mathbf{h}_2 \ra_{L^2}\right )\\
-&\tfrac{1}{2} \la \delta^g \{h_1,h_2-h_1^2\}, \delta^gH \ra_{L^2}+\tfrac{1}{8} \la \{h_1, \delta^{\star_g,+}\di\tr h_1^2\},H\ra_{L^2}
\end{split}
\end{equation*}
for all $H \in \Gamma \left (\Sym^{2,-}TM \right )$
\item[(ii)]
in particular, 
$$ \la \bfv(h_1,h_1),h_2-h_1^2\ra_{L^2}=2\la \delta^g[h_1,h_1]^c, \mathbf{h}_2 \ra_{L^2}+\tfrac{1}{4}\la \delta^gh_1^2, \di\!\tr h_1^2\ra_{L^2}.
$$
\end{itemize}
\end{lema}
\begin{proof}
(i) The tensor $h_2-h_1^2$ belongs to 
$\Gamma \left (\Sym^{2,-}TM \right )$ by \eqref{alg-T1} and also satisfies the equation $\bdel (h_2-h_1^2)=\bdel \mathbf{h}_2$(see \eqref{2nd-O}) as well as $\delta^g(h_2-h_1^2)=-\tfrac{1}{4}\di\!\tr h_1^2$ according to \eqref{gauge-1}; since we also know that $\bdel h_1=0$ and $ \delta^gh_1=0$ the first part of the claim follows from Lemma \ref{3-}.\\
(ii) We have $\la [h_1,h_1]^c, \bdel \mathbf{h}_2 \ra_{L^2}=\la [h_1,h_1]^c, \di_{\nabla^g}\! \mathbf{h}_2 \ra_{L^2}=\la \delta^g[h_1,h_1]^c,  \mathbf{h}_2 \ra_{L^2}$ by type considerations and duality. The claim 
follows by taking $H=h_1$ in (i) and using again that $\bdel h_1=0$ and $ \delta^gh_1=0$.
\end{proof}
All arguments are now in place in order to prove the main result of this section. Below we identify the obstruction to third 
order deformation and also make fully explicit the component of the Einstein equation on $\Gamma \left (\Sym^{2,-}TM \right )$.
\begin{teo} \label{main1-bN}
Assume that $(M^{2m},g,J)$ is compact K\"ahler Einstein with $E<0$ and that $g_t \in \mathscr{M}_1$ is a small time Einstein deformation of 
$g$ with Taylor expansion at $t=0$ given by $g^{-1}g_t=\id+th_1+\tfrac{t^2}{2}h_2+\tfrac{t^3}{3!}h_3+o(t^4)$. After normalising $g_t$ under the action of $\mathbf{G}$ according to $\delta^gh_1=0$ and 
\eqref{gauge-1} we must have 
\begin{equation*}
\bdel \mathbf{h}_2+[h_1,h_1]^c=0.
\end{equation*}
In addition, the component on $\Gamma \left (\Sym^{2,-}TM \right )$ of the third order Einstein equation reads 
\begin{equation} \label{3-eqn}
\begin{split}
\Delta_E\mathbf{h}_3 +\delta^{\star_g,-}\delta^g \left (-2\mathbf{h}_3 +\tfrac{2}{3}h_3^{+}-\{h_1,h_2-h_1^2\} \right ) &=2\delta^g[h_1,h_2-h_1^2]^c
\end{split}
\end{equation}
where the tensor $\mathbf{h}_3 \in \Gamma \left (\Sym^{2,-}TM \right )$ is determined from $-\tfrac{1}{3}h_3^{-}+\tfrac{1}{2}h_1^3=\mathbf{h}_3$.
\end{teo}
\begin{proof}
Let $H$ belong to $\Gamma \left ( \Sym^{2,-}TM \right )$. We plug the identity for $\la [h_1,h_1]^c,[h_1,H]^c\ra_{L^2}$ given in Corollary \ref{N7} in the Einstein equation as described in 
Theorem \ref{E3-1}. After an easy algebraic computation we obtain 
\begin{equation*}
\begin{split}
\la \widetilde{\Delta}_E(\tilde{\mathbf{u}}_3-\tfrac{1}{2}h_1^3)+\tfrac{1}{3}\delta^{\star_g}\di\tr \mathbf{u}_3,H\ra_{L^2}=&\bfv(h_1,h_2-h_1^2,H)+2\la [h_1,h_1]^c,[h_1,H]^c\ra_{L^2}\\
-&\tfrac{1}{8}\la \{h_1,\delta^{\star_g,+}\di\tr h_1^2\},H \ra_{L^2}.
\end{split}
\end{equation*} 
Using part (i) in Lemma \ref{v-fin} thus yields further 
\begin{equation*}
\begin{split}
\la \widetilde{\Delta}_E(\tilde{\mathbf{u}}_3-\tfrac{1}{2}h_1^3)+\tfrac{1}{3}\delta^{\star_g}\di\tr \mathbf{u}_3,H\ra_{L^2}
=&2\la [h_1,h_2-h_1^2]^c,\bdel H\ra_{L^2}+2 \la [h_1,H]^c, [h_1,h_1]^c+\bdel \mathbf{h}_2 \ra_{L^2}\\
-&\tfrac{1}{2} \la \delta^g \{h_1,h_2-h_1^2\}, \delta^gH \ra_{L^2}.
\end{split}
\end{equation*} 
Since $\bdel h_1=0$ and $\delta^gh_1=0$, in particular $\widetilde{\Delta}_Eh_1=0$, 
taking $H=h_1$ shows that the $L^2$-scalar product $\la [h_1,h_1]^c, [h_1,h_1]^c+\bdel \mathbf{h}_2 \ra_{L^2}=0$. Because
$\bdel \mathbf{h}_2+[h_1,h_1]^c$ is harmonic it follows that $\Vert \bdel \mathbf{h}_2+[h_1,h_1]^c\Vert^2_{L^2}=0$ hence 
$$\bdel \mathbf{h}_2+[h_1,h_1]^c=0.$$
It follows that 
\begin{equation*}
\begin{split}
&\la \widetilde{\Delta}_E(\tilde{\mathbf{u}}_3-\tfrac{1}{2}h_1^3)+\delta^{\star_g,-}\left ( \tfrac{1}{3}\di\tr \mathbf{u}_3
+\tfrac{1}{2}\delta^g \{h_1,h_2-h_1^2\}\right ),H\ra_{L^2}
=2\la [h_1,h_2-h_1^2]^c,\bdel H\ra_{L^2}.
\end{split}
\end{equation*} 
We now clarify the remaining algebraic details as follows; using that $\widetilde{\Delta}_E=\Delta_E-
\delta^{\star_g}(2\delta^g+\di\!\tr)$  we see that that the operator on the left hand side in the above equation is given by 
\begin{equation*}
\Delta_E(\tilde{\mathbf{u}}_3-\tfrac{1}{2}h_1^3)-\delta^{\star_g,-}\delta^g \left (
2\tilde{\mathbf{u}}_3-h_1^3-\tfrac{1}{2}\{h_1,h_2-h_1^2\} \right )-\delta^{\star_g,-}\di\!\tr \left (\tilde{\mathbf{u}}_3-\tfrac{1}{2}h_1^3-\tfrac{1}{3}\mathbf{u}_3 \right ).
\end{equation*}
Since $\tilde{\mathbf{u}}_3=\tfrac{1}{3}\mathbf{u}_3 +\tfrac{3}{2}h_1^3$ it follows that 
$\tr \left (\tilde{\mathbf{u}}_3-\tfrac{1}{2}h_1^3-\tfrac{1}{3}\mathbf{u}_3 \right )=\tr (h_1^3)=0$ since 
$h_1^3$ is $J$-anti-invariant. By also using the expression for $\mathbf{u}_3$ in \eqref{u3-d+} we obtain 
$$2\tilde{\mathbf{u}}_3-h_1^3-\tfrac{1}{2}\{h_1,h_2-h_1^2\}=-\tfrac{2}{3}h_3+\{h_1,h_2-h_1^2\}+h_1^3=
2\mathbf{h}_3-\tfrac{2}{3}h_3^{+}+\{h_1,h_2-h_1^2\}.$$
Therefore the term containing the operator $\delta^{\star_g,-}\delta^g$ on the left hand side of \eqref{3-rd} reads as stated. Finally, from 
$$ \tilde{\mathbf{u}}_3-\tfrac{1}{2}h_1^3=-\tfrac{1}{3}h_3+\tfrac{3}{4}\{h_1,h_2-h_1^2\}+\tfrac{1}{2}h_1^3
$$
and having $h_2-h_1^2$ in $\Gamma \left ( \Sym^{2,-}TM \right )$ we conclude that 
$ (\tilde{\mathbf{u}}_3-\tfrac{1}{2}h_1^3)^{-}=\mathbf{h}_3$. The proof is finished by recalling that the 
Einstein operator preserves the space $\Gamma \left ( \Sym^{2,-}TM \right )$, hence 
$\la \Delta_E(\tilde{\mathbf{u}}_3-\tfrac{1}{2}h_1^3),H \ra_{L^2}=\la \Delta_E\mathbf{h}_3,H\ra_{L^2}$.
\end{proof}
There remains to demonstrate that there are no further obstructions to solving equation \eqref{3-eqn}. This will 
be done in section \ref{E+3}, Theorem \ref{main-LL}, after finding the correct gauge normalisation to third order for the tensor $h_3$.
\section{The component of the Einstein equation on $\Gamma \left (\Sym^{2,+}TM \right )$} \label{+part}
Firstly we outline a few representation theoretical facts which will be systematically relied on in the subsequent.
In this section we will use the algebraic action defined according to 
$(h,\alpha) \in \Sym^2 TM \oplus \Lambda^{p}M \mapsto [h,\alpha] \in \Lambda^pM$ where 
$$ [h,\alpha](X_1, \ldots, X_p):=\alpha(hX_1, \ldots, X_p)+\dots +\alpha(X_1,\ldots, hX_p).
$$
Note this is induced by the Lie algebra action of $\mathfrak{gl}(TM)$ on $\Lambda^{\star}M$. This action is symmetric on $\Lambda^3M$, that is 
$g([h,T_1],T_2)=g(T_1,[h,T_2])$. In what follows we view $3$-forms $T \in \Omega^3M$ as elements of 
$\Omega^2(M,TM)$ via $(X,Y) \mapsto T(X,Y)$ where $g(T(X,Y),Z)=T(X,Y,Z)$. To eliminate confusing numerical factors, the inner product we work with on $3$-forms is thus $g(T_1,T_2)=\tfrac{1}{2} g(T_1(e_i,e_j),T_2(e_i,e_j))$. 

In addition we recall that the map given by $T \in \Lambda^{(1,2)+(2,1)}M \mapsto T^{-} \in 
\lambda_{+}^2(M,TM)$ is a vector bundle isomorphism, where
$$2T^{-}(X,Y)=T(X,Y)-T(JX,JY).$$
Recall that having $T^{-}$ in $\lambda_{+}^2(M,TM)$ amounts to $g(T^{-}(X,JY),JZ)=g(T^{-}(X,Y),Z)$. This 
follows from having $T$ of real type $(1,2)+(2,1)$, property which may be equivalenty re-phrased as $\mathfrak{J}T=JT$ where the actions 
$\mathfrak{J}T(X,Y,Z)=T(JX,Y,Z)+T(X,JY,Z)+T(X,Y,JZ)$ respectively $JT(X,Y,Z)=T(JX,JY,JZ)$. We also consider the linear embedding given by 
$T \in \Lambda^{(1,2)+(2,1)}M \mapsto T^{+} \in 
\Lambda^{1,1}(M,TM)$ where 
$$2T^{+}(X,Y)=T(X,Y)+T(JX,JY).$$

Moving on to differential geometric facts we record that 
\begin{equation} \label{d+1}
g((\di_{\nabla^g} H)(X,Y),Z)=-(\nabla^g_{JZ}A)(X,Y)+\di\!A(X,Y,JZ).
\end{equation}
whenever $H$ belongs to $
\Gamma \left (\Sym^{2,+}TM \right )$ and where $A:=H \circ J \in \Omega^{1,1}M$ and $T:=\di\!A$. 
We begin by providing more information regarding the functional $\bfv$. Specifically we determine the remaining 
mixed-type components in $\bfv$ as indicated in the following 
\begin{pro} \label{v+}
Consider a pair $(H,h)$ in $\Gamma \left (\Sym^{2,+}TM \right ) \oplus \Gamma \left (\Sym^{2,-}TM \right )$. We have 
\begin{equation*}
\bfv(H,H,h)=-\tfrac{2}{3} g(T,[h,T])-\la [H, \mathrm{\di^{-}} \delta^gH]+\{H, \delta^{\star_g,-}\delta^gH\}-\delta^{\star_g,-}\delta^gH^2,h \ra_{L^2}
\end{equation*}
where the forms $A:=H \circ J$ belongs to $\Omega^{1,1}M$ and $T:=\di\! A \in \Omega^{(1,2)+(2,1)}M$. 
\end{pro}
\begin{proof}
We begin by computing the quantity $g(h \sharp \di_{\nabla^g}H,  \di_{\nabla^g}H)$ using the expression for  $\di_{\nabla^g}H$ in 
\eqref{d+1} which entails 
\begin{equation*}
g\left ((h \sharp \di_{\nabla^g}H)(X,Y),Z \right )=-g(\{\nabla^g_{JZ}A,h\}X,Y)+T(hX,Y,JZ)+T(X,hY,JZ).
\end{equation*}
Observe now that the tensors $\{\nabla^g_{JZ}A,h\}$ respectively $\nabla^g_{JZ}A$ are $g$-orthogonal since the first is $J$-anti-invariant whilst the second is $J$-invariant. After using again \eqref{d+1} to expand $\di_{\nabla^g}H$ this allows computing the scalar product 
\begin{equation*}
\begin{split}
2 g(h \sharp \di_{\nabla^g}H,  \di_{\nabla^g}H)=&g\left ((h \sharp \di_{\nabla^g}H)(e_i,e_j),e_k \right ) g((\di_{\nabla^g}H)(e_i,e_j),e_k)\\
=&-g(\{\nabla^g_{Je_k}A,h\}e_i,e_j)T(e_i,e_j,Je_k)\\
&-g((\nabla^g_{Je_k}A)e_i,e_j) \left (T(he_i,e_j,Je_k)+T(e_i,he_j,Je_k) \right )\\
&+T(e_i,e_j,Je_k) \left (T(he_i,e_j,Je_k)+T(e_i,he_j,Je_k) \right ).
\end{split}
\end{equation*}
Replacing the frame $\{e_k\}$ with $\{Je_k \}$  and taking into account that $h$ is symmetric then yields 
\begin{equation} \label{T0-N}
\begin{split}
2 g(h \sharp \di_{\nabla^g}H,  \di_{\nabla^g}H)=&-2g(\{\nabla^g_{e_k}A,h\}e_i,e_j)T(e_i,e_j,e_k)\\
&+T(e_i,e_j,e_k) \left (T(he_i,e_j,e_k)+T(e_i,he_j,e_k) \right ).
\end{split}
\end{equation}
We now compute each summand separately, starting with the second; since $T$ is a $3$-form we have 
\begin{equation} \label{T1-N}
\begin{split}
&T(e_i,e_j,e_k) \left (T(he_i,e_j,e_k)+T(e_i,he_j,e_k) \right )=2T(e_i,e_j,e_k) T(he_i,e_j,e_k)\\
=&\tfrac{2}{3}T(e_i,e_j,e_k)[h,T](e_i,e_j,e_k)=\tfrac{4}{3}g(T,[h,T]).
\end{split}
\end{equation}
Further on, the product rule grants that 
\begin{equation*}
\begin{split}
-g(\{\nabla^g_{e_k}A,h\}e_i,e_j)T(e_i,e_j,e_k)=&
\{A,\nabla^g_{e_k}h\}(e_i,e_j)T(e_i,e_j,e_k)-\nabla_{e_k}\{A,h\}(e_i,e_j)T(e_i,e_j,e_k).
\end{split}
\end{equation*}
The crucial observation for continuing the proof is that the first sum above can be expressed using the braket $[H,H]^{\FN}$ as follows. We have, since $A$ is skew-symmetric, 
\begin{equation} \label{T2-N}
\begin{split}
\{A,\nabla^g_{e_k}h\}(e_i,e_j)T(e_i,e_j,e_k)=&\left (-(\nabla^g_{e_k}h)(e_i,Ae_j)+(\nabla^g_{e_k}h)(Ae_i,e_j) \right ) T(e_i,e_j,e_k)\\
=&(\nabla^g_{e_k}h)(e_i,e_j) \left (-T(Ae_i,e_j,e_k)+T(e_i,Ae_j,e_k) \right )\\
=&(\nabla^g_{e_k}h)(e_i,e_j) \left ( g(A \circ T(e_j,e_k),e_i)+ g(A \circ T(e_i,e_k),e_j) \right ).
\end{split}
\end{equation}
Now recall that (see proof of Lemma 5.4 in \cite{NS-E})
\begin{equation*}
g([H,H]^{\FN}(X,Y),Z)=-\alpha(JX,JY,Z)+T(X,Y,AZ)=-\alpha(JX,JY,Z)-g(A \circ T(X,Y),Z)
\end{equation*}
where $\alpha$ is a $3$-form on $M$. Using this in \eqref{T2-N} leads to 
\begin{equation*} \label{T3-N}
\begin{split}
&\{A,\nabla^g_{e_k}h\}(e_i,e_j)T(e_i,e_j,e_k)\\
=&-(\nabla^g_{e_k}h)(e_i,e_j) 
\left (\alpha(Je_j,Je_k,e_i)+g([H,H]^{\FN}(e_j,e_k),e_i) \right )\\
&-(\nabla^g_{e_k}h)(e_i,e_j) \left (\alpha(Je_i,Je_k,e_j)+g([H,H]^{\FN}(e_i,e_k),e_j) \right )\\
=& -(\nabla^g_{e_k}hJ)(e_i,e_j)J\alpha(e_j,e_k,e_i)-(\nabla^g_{e_k}hJ)(e_i,e_j)J\alpha(e_i,e_k,e_j)\\
+&(\nabla^g_{e_k}h)(e_i,e_j) \left (g([H,H]^{\FN}(e_k,e_j),e_i)+g([H,H]^{\FN}(e_k,e_i),e_j) \right )
\end{split}
\end{equation*}
after using that $h$ and $hJ$ are symmetric. Now the second displayed line above vanishes since $\alpha$ is $3$-form and $hJ$ is symmetric; after integration 
by parts we thus arrive at 
\begin{equation} \label{T4-N}
\begin{split}
&\{A,\nabla^g_{e_k}h\}(e_i,e_j)T(e_i,e_j,e_k)=2\la h, \delta^g[H,H]^{\FN}\ra_{L^2}.
\end{split}
\end{equation}
Plugging \eqref{T1-N} and \eqref{T4-N} into \eqref{T0-N} shows that 
\begin{equation*}
\begin{split}
\la h \sharp \di_{\nabla^g}H,  \di_{\nabla^g}H \ra_{L^2}=&2\la h, \delta^g[H,H]^{\FN}\ra_{L^2}+\tfrac{2}{3}\la T,[h,T] \ra_{L^2}
-\nabla^g_{e_k}\{A,h\}(e_i,e_j)T(e_i,e_j,e_k).
\end{split}
\end{equation*}
Integration by parts shows that 
\begin{equation*}
\begin{split}
-\nabla^g_{e_k}\{A,h\}(e_i,e_j)T(e_i,e_j,e_k)=-\{A,h\}(e_i,e_j)\di^{\star}T(e_i,e_j)=\la h, \{A,\di^{\star}T\}\ra_{L^2}.
\end{split}
\end{equation*}
Since the Hodge Laplacian preserves $\Omega^{1,1}M$ and $T=\di\!A$ type considerations show that $\la h, \{A,\di^{\star}T\}\ra_{L^2}=
\la h, \{A, \Delta^gA-\di\di^{\star}A\} \ra_{L^2}=-\la h, \{A, \di^{-} \di^{\star}A\} \ra_{L^2}$. Since $A=H \circ J$ and 
$\di^{\star}A =(J \delta^gH)^{\flat}$ we have 
$$-\{A, \di^{-} \di^{\star}A\}=[H,\di^{-}(\di^{\star}A )\circ J]=[H,\di^{-} \delta^g H ].
$$
Here we have used the general formula $\di^{-}(JX)^{\flat} \circ J=\di^{-} X^{\flat}$ whenever $X$ is a vector field on $M$.
Summarising we get 
\begin{equation} \label{T5-N}
\begin{split}
\la h \sharp \di_{\nabla^g}H,  \di_{\nabla^g}H \ra_{L^2}=&2\la h, \delta^g[H,H]^{\FN}\ra_{L^2}+\tfrac{2}{3}\la T,[h,T] \ra_{L^2}
+\la h, [H,\di^{-} \delta^g H ]\ra_{L^2}.
\end{split}
\end{equation}
Using now \eqref{bfv-1} shows that 
\begin{equation} 
\begin{split}
\bfv(H,H,h)=-\tfrac{2}{3} \la T,[h,T] \ra_{L^2}+\la h, -[H, \di^{-} \delta^gH]+\widetilde{\mathrm{L}}(H,H)\ra_{L^2}.
\end{split}
\end{equation}
Now $\widetilde{\mathrm{L}}(H,H)=\{H,\delta^g\di_{\nabla^g}H\}-\delta^g\di_{\nabla^g}H^2$ according to \eqref{L-t}; using the Weitzenb\"ock formula \eqref{wz1} we obtain $\{H,\delta^g\di_{\nabla^g}H\}=\{H, -\delta^{\star_g}\delta^gH+(\Delta_E+E+\ring{R})H\}$.
Since $\Delta_E+E+\ring{R}$ preserves $\Gamma \left (\Sym^{2,+}TM \right )$ it follows that $\{H,\delta^g\di_{\nabla^g}H\}_{-}=
-\{H, \delta^{\star_g,-}\delta^gH\}$. A similar argument based again on \eqref{wz1} and the $J$-invariance of $H^2$ shows that the component of $-\delta^g\di_{\nabla^g}H^2$ on $\Gamma \left (\Sym^{2,-}TM \right )$ equals $\delta^{\star_g,-}\delta^gH^2$. Thus
$$ \la \widetilde{\mathrm{L}}(H,H),h\ra_{L^2}=-\la \{H, \delta^{\star_g,-}\delta^gH\},h \ra_{L^2}+\la \delta^{\star_g,-}\delta^gH^2,h \ra_{L^2}
$$
and the claim follows. 
\end{proof}
The main results in this section now reads 
\begin{pro} \label{v+2}
Consider a pair $(H,h)$ in $\Gamma \left (\Sym^{2,+}TM \right )\oplus \Gamma \left (\Sym^{2,-}TM \right )$ which additionally satisfies $\bdel h=0$ and $\delta^gh=0$. We have 
\begin{equation*}
\bfv(h,h^2,H)=-\tfrac{2}{3}\la T, [h, \di(h^2\circ J)]\ra_{L^2} -\la \tfrac{1}{2}\{h,\delta^{\star_g,-}\delta^gh^2\}+\delta^{\star_g,+}\delta^g h^3-\tfrac{1}{2}[h,\bdel \delta^gh^2],H\ra_{L^2}
\end{equation*}
where $A:=H \circ J \in \Omega^{1,1}M$ and $T=\di\!A \in \Omega^{(1,2)+(2,1)}M$.
\end{pro}
\begin{proof}
Follows by polarisation from Proposition \ref{v+}; explicitly we take $H=h^2+tH$ in that proposition and identify the coeficient of $t$ of the resulting identity. Since the operator $[h, \cdot ] :\Omega^3M \to \Omega^3M$ is symmetric we essentially only need to deal with the last two summands in Proposition \ref{v+}. Indeed, the coefficient of $t$ in $\la \{h^2+tH,\delta^{\star_g,-}\delta^g(h^2+tH)\},h\ra_{L^2}$ equals 
\begin{equation*}
\la \{h^2,\delta^{\star_g,-}\delta^gH\}+\{H,\delta^{\star_g,-}\delta^gh^2\},h\ra_{L^2}=\la 2\delta^{\star_g,+}\delta^gh^3+\{h,\delta^{\star_g,-}\delta^gh^2\},H \ra_{L^2}.
\end{equation*}
Similarly, the coefficient of $t$ in $\la [h^2+tH, \di^{-} \delta^g(h^2+tH)],h \ra_{L^2}$ equals 
$$ \la [h^2,\di^{-} \delta^gH ]+[H,\di^{-} \delta^gh^2 ],h\ra_{L^2}=-\la [h,\di^{-} \delta^gh^2],H\ra_{L^2}
$$
since the operators $h$ and $h^2$ mutually commute. To conclude it is enough to to gather these facts and use the symmetry of 
$\bfv$ which entails that the coefficient of $t$ in $\bfv(h^2+tH,h^2+tH,h)$ equals $2\bfv(h,h^2,H)$.
\end{proof}
\subsection{The component of $\mathbf{w}$ on $\Gamma \left (\Sym^{2,+}TM \right ) $} \label{w+N}
The initial approach is similar to the one used for computing the component of $\mathbf{w}$ on 
$\Gamma 
\left (\Sym^{2,-}TM \right )$. However there are two major differences which are of purely algebraic nature and are contained in the two lemmas below. The first of those complements Lemma 
\ref{quad-K} as follows.
 \begin{lema} \label{quad-K+}
Let the pair $(h,H)$ belong to $\Gamma \left (\Sym^{2,-}TM \right )\oplus \Gamma \left (\Sym^{2,+}TM \right )$ and assume 
that $\overline{\partial} h=0$. Then 
$$ g(h^2\circ \di_{\nabla^g}h+h \circ \di_{\nabla^g}h^2,\di_{\nabla^g} H)=g(h \circ \di_{\nabla^g}h^2 ,J\di\!A)+g(\di_{\nabla^g} h^3,\di_{\nabla^g} H)$$
where $A:=H \circ J \in \Omega^{1,1}M$.
\end{lema}
\begin{proof}
As in the proof of Lemma \ref{quad-K} we have 
$$ (h^2\di_{\nabla^g}h+h \di_{\nabla^g}h^2)(X,Y)=\gamma_XY-\gamma_YX+\di_{\nabla^g}\!h^3(X,Y)
$$
where the tensor $\gamma:TM \to \lambda^2_J$ is given by $\gamma_X=[h^2,\nabla^g_Xh]$. Now split $\gamma=\gamma^{-}+\gamma^{+}$ where 
$\gamma^{-}_{JX}JY=-\gamma^{-}_XY$ and $\gamma^{+}_{JX}JY=\gamma^{+}_XY$; thus
$g(\gamma^{+}_X \cdot, \cdot)=T(X, \cdot ,\cdot)-T(X, J\cdot ,J\cdot)$ for some $T \in \Lambda^{(2,1)+(1,2)}M$. Since both $h$ and 
$h^2$ are symmetric direct computation shows that $\mathrm{a}(\gamma)=\mathrm{a}(h^2 \circ \di_{\nabla^g}h)$. Since 
$\di_{\nabla^g}h$ belongs to $\Omega^{1,1}(M,TM)$ it follows that 
$$ \mathrm{a}(\gamma) \in \Omega^{(1,2)+(2,1)}M.
$$
A the same time, since $h$ and $h^2$ commute we have $\gamma_X=[h,\nabla^g_Xh^2]$ hence 
$$ \mathrm{a}(\gamma)=\mathrm{a}(h \circ \di_{\nabla^g}h^2).
$$
This allows determining the $3$-form $T$ as follows. Firstly we have $\mathrm{a}(\gamma^{+})=3T+\mathfrak{J}(JT)=2T$; secondly, 
$\mathrm{a}\gamma=\mathrm{a}(\gamma^{-})+\mathrm{a}(\gamma^{+})$. Because $\mathrm{a}(\gamma^{-})$ belongs to $\lambda^3M$ and 
$T$ lies in $\Omega^{(1,2)+(2,1)}M$ we infer that $\mathrm{a}(\gamma^{-})=0$. Hence 
$$ 2T=\mathrm{a}(h \circ \di_{\nabla^g}h^2).
$$
Next the tensor $(X,Y) \mapsto \gamma^{-}_XY-\gamma^{-}_YX$ belongs to $\lambda^2_{-}(M,TM)$, in particular its divergence is $J$-anti-invariant. Also record that 
$$ g(\gamma^{+}_XY-\gamma^{+}_YX,Z)=2g(T^{+}(X,Y),Z)
$$
by essentially taking into account that $T \in \Omega^{(1,2)+(2,1)}M$. Therefore 
$$g(h^2\circ \di_{\nabla^g}h+h \circ \di_{\nabla^g}h^2,\di_{\nabla^g} H)=2 g(T^{+},\di_{\nabla^g} H)+g(\di_{\nabla^g} h^3,\di_{\nabla^g} H).
$$
Finally, we indicate with $\Psi=\mathrm{a}(\di^{+}_{\nabla^g} H)$ and observe that 
$$2 g(T^{+},\di_{\nabla^g} H)=2 g(T,\di_{\nabla^g}^{+} H)=\tfrac{2}{3}g(T,\Psi)=
g(h \circ \di_{\nabla^g}h^2, \Psi)$$ after also taking into account that $2T=\mathrm{a}(h \circ \di_{\nabla^g}h^2)$.
Also recall that the scalar product $g(h \circ \di_{\nabla^g}h^2, \Psi)=\tfrac{1}{2}g((h \circ \di_{\nabla^g}h^2)(e_i,e_j),\Psi(e_i,e_j))$ since we consider the $3$-form $\Psi$ as an element of 
$\Omega^2(M,TM)$. Since $A$ is $J$-invariant equation \eqref{d+1} entails 
\begin{equation*} 
g((\di^{+}_{\nabla^g} H)(X,Y),Z)=-(\nabla^g_{JZ}A)(X,Y)+\tfrac{1}{2} \left ( \di\!A(X,Y,JZ)+\di\!A(JX,JY,JZ) \right ).
\end{equation*}
It follows that 
$$\mathrm{a}(\di^{+}_{\nabla^g} H)=\tfrac{1}{2}(\mathfrak{J}\di\!A+3J\di\!A)-J\di\!A=J\di\!A$$ 
since $\di\!A$ belongs to 
$\Omega^{(1,2)+(2,1)}M$ and the action $\mathfrak{J}=J$ on the latter space. The proof of the claim is now complete.
\end{proof}
The next algebraic observation which is needed is the following 
 \begin{lema} \label{quad-K+N}
Consider a pair $(h,H)$ in $\Gamma \left (\Sym^{2,-}TM \right ) \oplus \Gamma \left (\Sym^{2,+}TM \right )$. Then 
\begin{equation*}
g(h \sharp \di^{-}_{\nabla^g}H, \di_{\nabla^g}h^2)=g(J\di\!A, h \circ \di_{\nabla^g}h^2 )+\tfrac{1}{3}g([h,T],\di (h^2J))
\end{equation*}
where $A=H \circ J$ in $\Omega^{1,1}M$ and $T=\di\!A$ in $\Omega^{(2,1)+(2,1)}M$.
\end{lema}
\begin{proof}
We use \eqref{d+1} to see that 
\begin{equation*} 
2g((\di^{-}_{\nabla^g} H)(X,Y),Z)=T(X,Y,JZ)-T(JX,JY,JZ).
\end{equation*}
Since $h$ is $J$-anti-invariant, it follows that 
\begin{equation*} 
2g((h \sharp \di^{-}_{\nabla^g} H)(X,Y),Z)=2g([h,T]^{+}(X,Y),JZ)-T(X,Y,hJZ)+JT(X,Y,hZ).
\end{equation*}
In addition 
\begin{equation*}
\begin{split}
-T(X,Y,hJZ)=&T(X,Y,JhZ)=(\mathfrak{J}T)(X,Y,hZ)-g(T(JX,Y)+T(X,JY),hZ)\\
=&(\mathfrak{J}T)(X,Y,hZ)-2g(T^{-}(X,JY),hZ)\\
=&(JT)(X,Y,hZ)+2g(Jh \circ T^{-}(X,Y),Z).
\end{split}
\end{equation*}
Here we have used that $g(T^{-}(X,JY),Z)=-g(T^{-}(X,Y),JZ)$, since $T^{-}$ belongs to 
$\lambda^2_{+}(M,TM)$, and also that $\mathfrak{J}T=JT$. Summarising these considerations yields
$$
h \sharp \di^{-}_{\nabla^g} H=-J \circ [h,T]^{+}+h \circ JT+Jh \circ T^{-}.
$$
The claim follows by using type considerations; indeed 
$$ g(Jh \circ T^{-},\di_{\nabla^g}h^2)=g(Jh \circ T^{-},\di^{-}_{\nabla^g}h^2)=0
$$
since $Jh \circ T^{-}$ belongs to $\lambda^2_{-}(M,TM)$, whilst $\di^{-}_{\nabla^g}h^2$ lies in $\lambda^2_{+}(M,TM)$. Thus 
\begin{equation*}
g(h \sharp \di^{-}_{\nabla^g}H, \di_{\nabla^g}h^2)=g(h \circ JT,\di_{\nabla^g}h^2)-g(J \circ [h,T]^{+}, \di_{\nabla^g}h^2).
\end{equation*}
To harness the last summand we observe that 
$$ -g(J \circ [h,T]^{+}, \di_{\nabla^g}h^2)=g([h,T], J \circ \di^{+}_{\nabla^g}h^2).
$$
Taking into account that $[h,T]$ is a $3$-form we see that 
\begin{equation*}
\begin{split}
g([h,T], J \circ \di^{+}_{\nabla^g}h^2)=&\tfrac{1}{2}g([h,T](e_i,e_j),e_k)g(J \circ \di^{+}_{\nabla^g}h^2(e_i,e_j),e_k)\\
=&
\tfrac{1}{6}g([h,T](e_i,e_j),e_k)g(\mathrm{a}(J \circ \di^{+}_{\nabla^g}h^2)(e_i,e_j),e_k)\\
=&\tfrac{1}{3}g([h,T], \di(h^2J))
\end{split}
\end{equation*}
after also taking into that $\mathrm{a}(J \circ \di^{+}_{\nabla^g}h^2)=\di(h^2J)$; the latter fact follows from \eqref{d+1}, see also end of the proof of Lemma \ref{quad-K+} for a similar argument. The claim follows now by gathering terms.
\end{proof}
Moving forward to determining the operator $\mathbf{w}$ requires one last additional step; namely we need to render 
explicit the component of $\mathbf{p}(h^2,h)$ respectively $\delta^g[h^2,h]^{\FN}$ on the space $\Gamma \left (\Sym^{2,+}TM \right )$. These considerations 
are summarised in the lemma below for the proof of which only straightforward type considerations are needed. 
\begin{lema} \label{sum-N}
Assume that $h \in \Gamma \left (\Sym^{2,-}TM \right )$ satisfies $\bdel h=0$ and $\delta^g h=0$. Whenever $H$ belongs to $\Gamma \left (\Sym^{2,+}TM \right )$ the following hold
\begin{itemize}
\item[(i)] we have $\tfrac{1}{2}\la \mathbf{p}(h^2,h),H \ra_{L^2}=\la \delta^g(h \circ \di_{\nabla^g}h^2),H\ra_{L^2}+
\la \delta^gh^2, \delta^g(h \circ H)\ra_{L^2}$
\item[(ii)] the bracket $[h^2,h]^{\FN}$ satisfies 
\begin{equation*}
\begin{split}
2\la \delta^g[h^2,h]^{\FN},H \ra_{L^2}=&\bfv(h,h^2,H)+\la H \sharp \di_{\nabla^g}h,\di_{\nabla^g}h^2\ra_{L^2}\\
-&
\la \delta^{\star_g,+}\delta^g h^3-\tfrac{1}{2}\{h,\delta^{\star_g,-}\delta^gh^2\}, H\ra_{L^2}
\end{split}
\end{equation*}
\item[(iii)] we have $-\la H \sharp \di_{\nabla^g}h, \di_{\nabla^g}h^2 \ra_{L^2}=\la \di_{\nabla^g}H, h \circ \di_{\nabla^g}h^2+
h^2 \circ \di_{\nabla^g}h \ra_{L^2}+\la \delta^g(h \circ H), \delta^gh^2 \ra_{L^2}$
\item[(iv)] the operator $\mathbf{p}$ satisfies $\la h \circ \mathbf{p}(h,h),H\ra_{L^2}=\tfrac{1}{2}\la [\di^{-} \delta^gh^2,h],H \ra_{L^2}$.
\end{itemize}
\end{lema}
\begin{proof}
(i) As observed at several places in the paper $\tfrac{1}{2}\mathbf{p}(h^2,h)=\mathbf{i}(h^2,h)$. By Lemma \ref{i-1} we have 
$\la \mathbf{i}(h^2,h),H \ra_{L^2}=\la \delta^g(h \circ \di_{\nabla^g}h^2),H\ra_{L^2}+\mathbf{q}(h^2,h \circ H)-
\la h, \mathrm{L}(H,h^2) \ra_{L^2}.$ Since $H$ and $h^2$ are $J$-invariant the definition of $\mathrm{L}$ ensures that 
$\mathrm{L}(H,h^2)$ belongs to $\Gamma \left (\Sym^{2,+}TM \right )$ thus $\la h, \mathrm{L}(H,h^2) \ra_{L^2}=0$. Furthermore the operator $\ring{R}-E$ preserves complex tensor type; since $h^2$ is $J$-invariant and $h \circ H$ is $J$-anti-invariant Lemma \ref{q} guarantees the equality
$\mathbf{q}(h^2,h \circ H)=\la \delta^gh^2, \delta^g(h \circ H)\ra_{L^2}$. The claim follows by collecting these facts.\\
(ii) After taking into account the comparison formula \eqref{bfv-1} it turns out that proving the claim
amounts to the computation of 
$\la \widetilde{\mathrm{L}}(h,h^2),H \ra_{L^2}$. According to \eqref{L-t} we have 
$$2\widetilde{\mathrm{L}}(h,h^2)=\{h,\delta^g \di_{\nabla^g}h^2\}+\{h^2,\delta^g \di_{\nabla^g}h\}-2\delta^g \di_{\nabla^g}h^3.$$
Use the Weitzenb\"ock formula \eqref{wz1}, that $\delta^gh=0$ and also that the operator $\Delta_E+\ring{R}+E$ preserves the spaces $\Gamma \left (\Sym^{2,\pm}TM \right )$; it follows that 
$$\la \widetilde{\mathrm{L}}(h,h^2),H \ra_{L^2}=\la \delta^{\star_g,+}\delta^g h^3-\tfrac{1}{2}\{h,\delta^{\star_g,-}\delta^gh^2\}, H\ra_{L^2}$$ and the claim is proved.\\
(iii) We use Proposition \ref{o3-p} with $h_1=H,h_2=h$ respectively $h_3=h^2$. Since the operator $\mathrm{L}$ satisfies 
$\mathrm{L}(\Gamma \left (\Sym^{2,+}TM \right ), \Gamma \left (\Sym^{2,\pm}TM \right )) \subseteq 
\Gamma \left (\Sym^{2,\pm}TM \right )$ and $\nabla^g h^2$ respectively 
$\nabla^g h$ live in orthogonal spaces it follows that 
$$\la H \sharp \di_{\nabla^g}h, \di_{\nabla^g}h^2 \ra_{L^2}=-\la \di_{\nabla^g}H, h \circ \di_{\nabla^g}h^2+
h^2 \circ \di_{\nabla^g}h \ra_{L^2}-\mathbf{q}(h \circ H,h^2)-\mathbf{q}(h^2 \circ H,h).
$$
In addition, type considerations based on Lemma \ref{q}, similar to those in the proof of (i), show 
that $\mathbf{q}(h \circ H,h^2)=\la \delta^g(h \circ H), \delta^gh^2 \ra_{L^2}$ and $\mathbf{q}(h^2 \circ H,h)=0$ and the claim follows.\\
(iv) The approach is similar to the one in Lemma \ref{w-step}; however the computation is much simpler due to the systematic use of type considerations. Because $h$ is symmetric we have that 
$g(h \circ \mathbf{p}(h,h),H)=g(\mathbf{p}(h,h),H_{-})+g(\mathbf{p}(h,h),A_{-})$ 
where the tensor $H_{-}:=\tfrac{1}{2}\{h,H\}$ belongs to $\Gamma \left (\Sym^{2,-}TM \right )$ and $A_{-}:=\tfrac{1}{2}[h,H] \in \Gamma \left (\lambda^2M \right )$. Because $\di_{\nabla^g}h$ belongs to 
$\Omega^{1,1}(M,TM)$ and $\mathrm{L}(h,h) \in \Gamma \left (\Sym^{2,+}TM \right )$, as entailed by the definition of $\mathrm{L}$, using Lemma \ref{quad-1L} shows that $g(\mathbf{p}(h,h),H_{-})=0$. By Proposition \ref{p-skew} we have $\la \mathbf{p}(h,h),A_{-} \ra_{L^2}=2\la 
\delta^g[h,h]^{\FN},A_{-}\ra_{L^2}$. Furthermore we expand the divergence 
$\delta^g[h,h]^{\FN}=\delta^g[h,h]^c+\delta^g\di^{+}_{\nabla^g}h^2$ and recall that $\delta^g[h,h]^{c}$ is a $\TT$-tensor (see Proposition \ref{del-TT}) whilst 
the component of $\delta^g\di^{+}_{\nabla^g}h^2$ on $\Gamma \left (\lambda^2M \right )$ is given by $-\tfrac{1}{2}\di^{-} \delta^gh^2$ by Proposition \ref{wz_K},(i). It follows that 
$$\la h \circ \mathbf{p}(h,h),H\ra_{L^2}=-\la \di^{-} \delta^gh^2 ,A_{-}\ra_{L^2}=-\tfrac{1}{2}\la \di^{-} \delta^gh^2,[h,H]\ra_{L^2}=
\tfrac{1}{2}\la [\di^{-} \delta^gh^2,h],H\ra_{L^2}$$ and the claim is fully proved. 
\end{proof}
We may now derive the explicit expression for the component of $\mathbf{w}$ on $\Gamma \left ( \Sym^{2,+}TM \right )$ as follows.
\begin{pro} \label{w+}
Assume that $h \in \Gamma \left (\Sym^{2,-}TM \right )$ satisfies $\bdel h=0$ and $\delta^g h=0$. We have 
\begin{equation*}
\begin{split}
\la \mathbf{w}(h), H\ra_{L^2}=&-\bfv(h^2,h,H)+\la \delta^{\star_g,+}\delta^gh^3-\tfrac{1}{2}\{h,\delta^{\star_g,-}\delta^gh^2\}, H\ra_{L^2}
\end{split}
\end{equation*}
whenever $H$ belongs to $\Gamma \left ( \Sym^{2,+}TM \right )$.
\end{pro}
\begin{proof}
We start from Proposition \ref{id-main}; since $\di_{\nabla^g}h(h \cdot, h \cdot)$ belongs to $\Omega^{1,1}(M,TM)$ and the bracket 
$[h,h]^{\FN}=[h,h]^c+\di^{+}_{\nabla^g}h^2$ the component of the latter on $\Omega^{1,1}(M,TM)$ reads 
\begin{equation*}
\di_{\nabla^g}h(h \cdot, h \cdot)=h \sharp \di^{-}_{\nabla^g}h^2-\partial^g h^3+h \circ \di^{+}_{\nabla^g}h^2.
\end{equation*}
Taking the scalar product with $\di_{\nabla^g}H=\di_{\nabla^g}^{+}H+\di_{\nabla^g}^{-}H$ shows that 
\begin{equation*}
\begin{split}
\la \di_{\nabla^g}h(h \cdot, h \cdot), \di_{\nabla^g}H \ra_{L^2}=&\la \di_{\nabla^g}h(h \cdot, h \cdot), \di^{+}_{\nabla^g}H \ra_{L^2}\\
=&\la h \sharp  \di_{\nabla^g}h^2, \di_{\nabla^g}H-\di_{\nabla^g}^{-}H \ra_{L^2}-
\la \partial^g h^3,\di_{\nabla^g}^{+}H \ra_{L^2} \\
+&\la h \circ \di_{\nabla^g}h^2,\di_{\nabla^g}H-\di_{\nabla^g}^{-}H \ra_{L^2}.
\end{split}
\end{equation*}
Now $\la h \circ \di_{\nabla^g}h^2,\di_{\nabla^g}^{-}H \ra_{L^2}=\la h \circ \di_{\nabla^g}^{-}h^2,\di_{\nabla^g}^{-}H \ra_{L^2}=0$
since $h \circ \di_{\nabla^g}^{-}h^2$ belongs to $\lambda^2_{-}(M,TM)$ whilst $\di_{\nabla^g}^{-}H$ lies in $\lambda^2_{+}(M,TM)$. Furthermore we take this fact into account after also recalling that $Q(h)=\di_{\nabla^g}h(h \cdot, h \cdot)+h^2 \sharp \di_{\nabla^g}h-\di_{\nabla^g}h^3$ according to the definition of $Q$ in Lemma \ref{der-Df}; since the Fr\"olicher-Nijenhuis bracket may be expressed (see also \eqref{bra-nac2}) according to $2[h^2,h]^{\FN}=-(h \sharp \di_{\nabla^g}h^2+h^2 \sharp \di_{\nabla^g}h)+2\di_{\nabla^g}h^3$ we end up with 
\begin{equation*}
\begin{split}
\la \delta^gQ(h),H \ra_{L^2}=&-2\la [h,h^2]^{\FN},\di_{\nabla^g}H \ra_{L^2}+\la \delta^g(h \circ \di_{\nabla^g}h^2),H \ra_{L^2}
-\la \di_{\nabla^g}h^2, h \sharp \di_{\nabla^g}^{-}H  \ra_{L^2}\\
+&\la \di_{\nabla^g}h^3, \di_{\nabla^g}H \ra_{L^2}-\la \partial^g h^3, \di_{\nabla^g}^{+}H \ra_{L^2}.
\end{split}
\end{equation*}
Furthermore $\la \di_{\nabla^g}h^3, \di_{\nabla^g}H \ra_{L^2}-\la \partial^g h^3, \di_{\nabla^g}^{+}H \ra_{L^2}=
\la \bdel h^3, \di_{\nabla^g}^{-}H \ra_{L^2}=0$ since $\bdel h^3 \in \lambda_{-}^2(M,TM)$ and $\di_{\nabla^g}^{-}H \in 
\lambda_{+}^2(M,TM)$. Using successively parts (i) and (ii) in Lemma \ref{sum-N} thus yields 
\begin{equation*}
\begin{split}
\la \delta^g Q(h)-\tfrac{1}{2}\mathbf{p}(h^2,h), H\ra_{L^2}=&-\bfv(h^2,h,H)\\
-& \la \di_{\nabla^g}h^2, h \sharp \di_{\nabla^g}^{-}H \ra_{L^2}-\la H \sharp \di_{\nabla^g}h, \di_{\nabla^g}h^2 \ra_{L^2}-
\la\delta^gh^2, \delta^g(h \circ H) \ra_{L^2}\\
+&\la \delta^{\star_g,+}\delta^g h^3-\tfrac{1}{2}\{h,\delta^{\star_g,-}\delta^gh^2\}, H\ra_{L^2}.
\end{split}
\end{equation*}
Taking into account the expression for $-\la H \sharp \di_{\nabla^g}h, \di_{\nabla^g}h^2 \ra_{L^2}$ obtained in part (iii) in Lemma \ref{sum-N} hence leads to 
\begin{equation} \label{w+1N}
\begin{split}
\la \delta^g Q(h)-\tfrac{1}{2}\mathbf{p}(h^2,h),H\ra_{L^2}=&-\bfv(h^2,h,H)\\
+& \la h \circ \di_{\nabla^g}h^2+h^2 \circ \di_{\nabla^g}h, \di_{\nabla^g}H\ra_{L^2}-\la \di_{\nabla^g}h^2, h \sharp \di_{\nabla^g}^{-}H \ra_{L^2}\\
+&\la \delta^{\star_g,+}\delta^g h^3-\tfrac{1}{2}\{h,\delta^{\star_g,-}\delta^gh^2\}, H\ra_{L^2}.
\end{split}
\end{equation}
Now we recall that by definition $\mathbf{w}(h)=2\delta^gQ(h)-\mathbf{p}(h^2,h)-h \circ \mathbf{p}(h,h)$ together with 
the identity $\la h \circ \mathbf{p}(h,h),H\ra_{L^2}=\tfrac{1}{2}\la [\di^{-} \delta^gh^2,h],H \ra_{L^2}$ granted by Lemma \ref{sum-N}, (iv). Furthermore by combining Lemmas \ref{quad-K+} and \ref{quad-K+N} we get 
\begin{equation*}
\begin{split}
&\la h \circ \di_{\nabla^g}h^2+h^2 \circ \di_{\nabla^g}h, \di_{\nabla^g}H\ra_{L^2}-\la \di_{\nabla^g}h^2, h \sharp \di_{\nabla^g}^{-}H \ra_{L^2}\\
=&\la\di_{\nabla^g}h^3, \di_{\nabla^g}H\ra_{L^2}-\tfrac{1}{3} \la [h,T], \di(h^2 \circ J) \ra_{L^2}
\end{split}
\end{equation*}
where $T:=\di\!A$ and $A:=h \circ J$ belongs to $\Omega^{1,1}M$. In addition, type considerations based on the Weitzenb\"ock formula \eqref{wz1} show that
$\la\di_{\nabla^g}h^3, \di_{\nabla^g}H\ra_{L^2}+\la \delta^{\star_g,+}\delta^g h^3,H \ra_{L^2}=0$. Plugging all these facts into \eqref{w+1N} yields 
\begin{equation*}
\begin{split}
\la \mathbf{w}(h), H\ra_{L^2}=&-2\bfv(h^2,h,H)
-\tfrac{2}{3} \la [h,T], \di(h^2 \circ J) \ra_{L^2}-\la \{h,\delta^{\star_g,-}\delta^gh^2\}, H\ra_{L^2}\\
&-\tfrac{1}{2}\la [\di^{-} \delta^gh^2,h],H \ra_{L^2}.
\end{split}
\end{equation*}
The claim follows now by expressing $-\tfrac{2}{3} \la [h,T], \di(h^2 \circ J) \ra_{L^2}$ with the aid of Proposition \ref{v+2}.
\end{proof}
\subsection{The Einstein equation on $\Gamma \left (\Sym^{2,+}TM \right )$} \label{E+3}
This is again based on Theorem \ref{thm-54}. Before proceeding with details we establish a few preparatory results. Let the vector field $\mathbf{f}$ on $M$ be determined from 
$$g(\mathbf{f},X)=g(h_1\delta^g(h_2-h_1^2),X)+g(\bdel (h_2-h_1^2), X \wedge h_1).$$
This vector field can be computed explicitly; based on the obstruction to integrability developed before we can show that 
$\mathbf{f}$ occurs as the divergence of an explicit polynomial tensor in $h_1$. This plays a crucial role in identifying the correct 
gauge normalisation, to third order, of the Einstein equation.
\begin{lema} \label{div-old}
Assume that $h \in \Gamma\left ( \Sym^{2,-}TM \right )$ satisfies $\bdel h=0$ and $\delta^gh=0$. Then 
$$\la \nabla^g_{JX}h,hJ \ra_{L^2}=-\la X, 2\delta^gh^2+\tfrac{1}{2}\di\!\tr h^2 \ra_{L^2}
$$
whenever $X \in \Gamma \left ( TM \right )$.
\end{lema}
\begin{proof}
This is \cite[Lemma 4.13]{N1} applied to the special case when $\bdel h=0$ and $\delta^gh=0$.
\end{proof}
\begin{pro} \label{f-vf}
We have 
$\mathbf{f}=\delta^g(h_1^3-\tfrac{1}{4}\tr(h_1^2)h_1).$
\end{pro}
\begin{proof}
Since $\bdel(h_2-h_1^2)+[h_1,h_1]^c=0$ by Theorem \ref{main1-bN} and also $\delta^g(h_2-h_1^2)=-\tfrac{1}{4}\di\!\tr(h_1^2)$ by the gauge normalisation in \eqref{gauge-1} we obtain 
$$ -\la \mathbf{f},X\ra_{L^2}=\tfrac{1}{4}\la h_1\di\!\tr(h_1^2),X \ra_{L^2}+\la [h_1,h_1]^c,X \wedge h_1 
\ra_{L^2}
$$
whenever $X \in \Gamma \left (TM \right )$.
Thus we only need to determine the quantity $\la [h_1,h_1]^c,X \wedge h_1 
\ra_{L^2}$. In doing so we proceed directly from the explicit expression for the bracket $[h_1,h_1]^c$ as indicated below. We have 
\begin{equation*} 
\begin{split}
2[h_1,h_1]^c(X,Y)=&[h_1,h_1]^{\FN}(X,Y)-[h_1J,h_1J]^{\FN}(X,Y)\\
=&-(\nabla^g_{h_1X}h_1)Y+(\nabla^g_{h_1Y}h_1)X+(\nabla^g_{h_1JX}h_1)JY-(\nabla^g_{h_1JY}h_1)JX.
\end{split}
\end{equation*}
At the same time, expanding the exterior product $X \wedge h_1$ yields
\begin{equation*} 
\begin{split}
\la [h_1,h_1]^c,X \wedge h_1 
\ra_{L^2}=&\tfrac{1}{2} \la [h_1,h_1]^c(e_i,e_j), g(X,e_i)he_j- g(X,e_j)he_i \ra_{L^2}\\
=&\la [h_1,h_1]^c(X,e_j),he_j \ra_{L^2}.
\end{split}
\end{equation*}
It follows that 
\begin{equation*} 
\begin{split}
2\la [h_1,h_1]^c,X \wedge h_1 
\ra_{L^2}=&-\la \nabla^g_{h_1X}h_1,h_1 \ra_{L^2}+\la (\nabla^g_{h_1e_j}h_1)X,h_1e_j\ra_{L^2}+\la \nabla^g_{h_1JX}h_1J,h_1 \ra_{L^2}\\
&-\la (\nabla^g_{h_1Je_j}h_1)JX,h_1e_j\ra_{L^2}.
\end{split}
\end{equation*}
However $-\la (\nabla^g_{h_1Je_j}h_1)JX,h_1e_j\ra_{L^2}=\la (\nabla^g_{h_1e_j}h_1)X,h_1e_j\ra_{L^2}=\la X, (\nabla^g_{e_j}h_1)h_1^2e_j\ra_{L^2}$ 
after the basis change $e_j \mapsto Je_j$ and using that $h_1$ is symmetric. But $(\nabla^g_{e_j}h_1)h_1^2e_j=-\delta^gh_1^3+h_1\delta^gh_1^2$ after using Leibniz's rule thus 
\begin{equation*} 
\begin{split}
2\la [h_1,h_1]^c,X \wedge h_1 
\ra_{L^2}=&-\la \nabla^g_{h_1X}h_1,h_1 \ra_{L^2}+2\la -\delta^gh_1^3+h_1\delta^gh_1^2,X \ra_{L^2}+\la \nabla^g_{h_1JX}h_1J,h_1 \ra_{L^2}.
\end{split}
\end{equation*}
Furthermore the terms involving the covariant derivative $\nabla^g$ can de determined as follows. We have $\la \nabla^g_{h_1JX}h_1J,h_1 \ra_{L^2}=\la \nabla^g_{Jh_1X}h_1,h_1J \ra_{L^2}$ since $h_1$ anti-commutes with $J$; in addition  
$\la \nabla^g_{Jh_1X}h_1,h_1J \ra_{L^2}=-\la h_1X, 2\delta^gh_1^2+\tfrac{1}{2}\di\!\tr h_1^2 \ra_{L^2}$ by Lemma 
\ref{div-old}. Because 
the scalar product $\la \nabla^g_{h_1X}h_1,h_1 \ra_{L^2}=\tfrac{1}{2}\la h_1X, \di\!\tr(h_1^2) \ra_{L^2}$ gathering terms leads to 
$$\la [h_1,h_1]^c,X \wedge h_1 
\ra_{L^2}=-\la \tfrac{1}{2}h_1 \di\!\tr(h_1^2)+\delta^gh_1^3,X \ra_{L^2}.
$$
Thus $\mathbf{f}=\delta^gh_1^3+\tfrac{1}{4}h_1\di\!\tr(h_1^2)$; since $h_1$ is divergence free $\delta^g(\tr(h_1^2)h_1)=-h_1\di\!\tr(h_1^2)$ and the claim is proved.
\end{proof} 
 
Before proceeding to the proof of the main result we need to establish the following preliminary 
\begin{lema} \label{simp+}
Let $g_t \in \mathscr{M}_1$ be an Einstein deformation with Taylor series expansion given by $g^{-1}g_t=\id+t h_1+\tfrac{t^2}{2}h_2+o(t^3)$, normalised according to \eqref{gauge-1}. The tensors $h_1$ and $h_2$ satisfy 
\begin{itemize}
\item[(i)]
$
\tfrac{1}{4}\{h_1,(\widetilde{\Delta}_E+2E+\delta^{\star}\di\tr )h_2\}_{+}=\tfrac{1}{4}\{h_1,\widetilde{\Delta}^{-}_E(h_2-h_1^2)\}+\tfrac{E}{2}\{h_1,h_2-h_1^2\}-\tfrac{1}{2}\{h_1,\delta^{\star_g,-}\delta^gh_1^2\}
$
\item[(ii)] we have 
\begin{equation*}
\begin{split}
&\la \mathbf{v}(h_1,h_2-h_1^2)-
\tfrac{1}{4}\{h_1,(\widetilde{\Delta}_E+2E+\delta^{\star}\di\tr )h_2\},H \ra_{L^2}\\
=&\la \tfrac{1}{4}\widetilde{\Delta}_E\{h_1,h_2-h_1^2\}+\tfrac{1}{2}\{h_1,\delta^{\star_g,-}\delta^gh_1^2\}+\delta^{\star_{g},+}\mathbf{f},H \ra_{L^2}
\end{split}
\end{equation*}
for all $H$ be in $\Gamma(\Sym^{2,+}TM)$. 
\end{itemize}
\end{lema}
\begin{proof}
(i) Using that $h_2-h_1^2$ is trace free, that is $\tr(h_2)=\tr(h_1^2)$ we see that the components of $(\widetilde{\Delta}_E+\delta^{\star}\di\tr )h_2$ on $\Gamma \left (\Sym^{2,\pm}TM \right )$ may be 
computed from 
\begin{equation*} 
\begin{split}
(\widetilde{\Delta}_E+\delta^{\star}\di\tr )h_2=&\widetilde{\Delta}_E(h_2-h_1^2)+\widetilde{\Delta}_Eh_1^2+\delta^{\star_g}\di\!\tr(h_1^2)\\
=&\widetilde{\Delta}^{-}_E(h_2-h_1^2)-2\delta^{\star_g,+}\delta^g(h_2-h_1^2)\\
+&\Delta_Eh_1^2-2\delta^{\star,+}\delta^gh_1^2
-2\delta^{\star,-}\delta^gh_1^2.
\end{split}
\end{equation*}
Since $h_2-h_1^2$ is $J$-anti-invariant, $h_1^2$ is $J$-invariant and the operators $\widetilde{\Delta}^{-}_E$ and $ \Delta_E$ preserve 
tensor type the claim follows by straightforward type considerations .\\
(ii) We first compute $\mathbf{v}(h_1,h_2-h_1^2,H)$ with the aid of Proposition \ref{vm-7bis} which we apply with $h=h_1, H_{-}=h_2-h_1^2$ 
and $H_{+}=H$. Since $\delta^g(h_2-h_1^2)$ is exact by \eqref{gauge-1} we have $\di^{-}\! \delta^g(h_2-h_1^2)=0$ thus 
\begin{equation*} 
\begin{split}
\bfv(h_1,h_2-h_1^2,H)=&\tfrac{1}{4} \la \{h_1, \widetilde{\Delta}_E^{-}(h_2-h_1^2)\}+(\widetilde{\Delta}_E+2E)\{h_1,h_2-h_1^2\},H \ra_{L^2}\\
+&\la \delta^{\star_g,+}(h_1\delta^g(h_2-h_1^2),H \ra_{L^2}+
\la  \bdel (h_2-h_1^2), \delta^gH \wedge h_1 \ra_{L^2}.
\end{split}
\end{equation*}
The claim follows now from (i) and the definition of the vector field $\mathbf{f}$. 
\end{proof}
The main result of this section is the complete solution of the 3-rd order Einstein equation. The latter can be fully solved by suitable normalisation under the gauge group $\mathbf{G}$ and elementary Hodge theory. We record the algebraic expression 
\begin{equation} \label{tr-3}
\tfrac{1}{3}\mathbf{u}_3-\tfrac{1}{4}\{h_1,h_2-h_1^2\}=\tfrac{1}{3} \left (-h_3+\tfrac{3}{2}\{h_1,h_2-h_1^2\}-\tfrac{3}{2}h_1^3 \right )
\end{equation}
which essentially follows from \eqref{u3-d+}. Since $h_1^3$ is $J$-anti-invariant we have $\tr(h_1^3)=0$; thus the tensor above is trace-free by Lemma \ref{tr-u3}. This fact will play a crucial role in the normalisation procedure below.
\begin{teo} \label{vmmm2}
Let $g_t \in \mathscr{M}_1$ be an Einstein deformation with Taylor series expansion given by $g^{-1}g_t=\id+t h_1+\tfrac{t^2}{2}h_2+\tfrac{t^3}{6}h_3
+o(t^4)$, normalised according to \eqref{gauge-1}, that is such that 
$\delta^gh_1=0$ and $\delta^g(h_2-h_1^2)+\tfrac{1}{4}\di\!\tr(h_1^2)=0$. The following hold 
\begin{itemize}
\item[(i)] the component of the $3$-rd order Einstein equation on $\Gamma \left (\Sym^{2,+}TM \right )$ reads
\begin{equation*}
\Delta_E\left (h_3^{+}-\tfrac{3}{2}\{h_1,h_2-h_1^2\} \right )=2\delta^{\star_g,+}\mathbf{X}_3(g_t)
\end{equation*}
where the vector field 
$$\mathbf{X}_3(g_t)=\delta^g \left (h_3-\tfrac{3}{2}\{h_1,h_2-h_1^2\}-\tfrac{3}{2}h_1^3+\tfrac{3}{8}\tr(h_1^2)h_1 \right )+\tfrac{3}{8}\di\!\tr \{h_1,h_2-h_1^2\}$$
\item[(ii)] after normalising $g_t$ by a gauge transformation in $\mathbf{G}$ such that 
in addition to \eqref{gauge-1} we also have $\di\!\mathbf{X}_3(g_t)=0$, we must have 
$\mathbf{X}_3(g_t)=0$ that is 
\begin{equation} \label{norm-3}
\delta^g \left (h_3-\tfrac{3}{2}\{h_1,h_2-h_1^2\}-\tfrac{3}{2}h_1^3+\tfrac{3}{8}\tr(h_1^2)h_1 \right )=-\tfrac{3}{8}\di\!\tr \{h_1,h_2-h_1^2\}
\end{equation} 
and also 
\begin{equation*}
h_3^{+}=\tfrac{3}{2}\{h_1,h_2-h_1^2\}.
\end{equation*}
\end{itemize}
\end{teo}
\begin{proof}
(i) From equation \eqref{xxx3} in Theorem \ref{thm-54} we get 
\begin{equation*}
\begin{split}
\la (\widetilde{\Delta}_E+\delta^{\star_g}\di\!\tr)(\tfrac{1}{3}\mathbf{u}_3), H \ra_{L^2}=&\la \mathbf{v}(h_1,h_2-h_1^2)-
\tfrac{1}{4}\{h_1,(\widetilde{\Delta}_E+2E+\delta^{\star}\di\tr )h_2\},H \ra_{L^2}\\
+&\la \bfv(h_1,h_1^2)+\mathbf{w}(h_1),H \ra_{L^2}
\end{split}
\end{equation*}
for all $H$ in $\Gamma \left (\Sym^{2,+}TM \right )$. Using Lemma \ref{simp+} to express the first line above and also Proposition \ref{w+} for the second line yields 
\begin{equation*}
\begin{split}
\la \widetilde{\Delta}_E \left ( \tfrac{1}{3}\mathbf{u}_3-\tfrac{1}{4}\{h_1,h_2-h_1^2\} \right ),H \ra_{L^2}=&\la 
\delta^{\star_g,+}\left ( \mathbf{f}+\delta^gh_1^3-\tfrac{1}{3}\di \tr \mathbf{u}_3 \right ), H \ra_{L^2}.
\end{split}
\end{equation*}
Now we expand $\widetilde{\Delta}_E=\Delta_E-\delta^{\star_g}(2\delta^g+\di\!\tr)$ and use \eqref{tr-3} which also grants 
that the component $\left ( \tfrac{1}{3}\mathbf{u}_3-\tfrac{1}{4}\{h_1,h_2-h_1^2\} \right )^{+}=
\tfrac{1}{3} \left (-h_3^{+}+\tfrac{3}{2}\{h_1,h_2-h_1^2\} \right )$. After taking into account 
that the Einstein operator $\Delta_E$ preserves tensor type the claim follows now after some algebraic computation from Proposition \ref{f-vf}. Note that for obtaining the terms involving $\di\!\tr$ in $\mathbf{X}_3(g)$ we use that 
$\tr(\mathbf{u}_3)=\tfrac{3}{4}\tr\{h_1,h_2-h_1^2\}$, as granted by \eqref{tr-3} and Lemma \ref{tr-u3}. \\
(ii) After operating a time dependent gauge transformation as provided by Proposition \ref{norm-2nd} we may assume that the family $g_t$ satisfies \eqref{gauge-1} as well as $\di\!\mathbf{X}_3(g_t)=0$. Since 
$h_3^{+}-\tfrac{3}{2}\{h_1,h_2-h_1^2\} $ is trace free by \eqref{tr-3}, applying the trace in (i) shows that $\tr \delta^{\star_g,+}\mathbf{X}_3(g_t)=0$. Since elements of $\Sym^{2,-}TM$ are trace free we have $\tr \delta^{\star_g,-}\mathbf{X}_3(g_t)=0$ hence we end up with 
$\tr \delta^{\star_g}\mathbf{X}_3(g_t)=0$. Because we have normalised such that $\mathbf{X}_3(g_t)=\di\!\varphi$ for some function 
$\varphi$ it follows that 
$\Delta^g \varphi=0$. Having $M$ compact ensures that $\varphi$ is constant and thus $\mathbf{X}_3(g_t)=0$. Finally, since $E<0$ the Einstein 
operator acting on $\Gamma \left (\Sym^{2,+}TM \right )$ has vanishing kernel, fact which forces $h_3^{+}-\tfrac{3}{2}\{h_1,h_2-h_1^2\} =0$.
\end{proof}
To finish this section we now show that the normalisation in \eqref{norm-3} is consistent with the component of the Einstein equation on $\Gamma \left (\Sym^{2,-}TM \right )$ described in 
\eqref{3-eqn} and allows bringing the latter to final form;  we explicitly relate equation \eqref{3-eqn} to the third order integrability of the almost complex structure $J_t$ defined by the Cayley transform. 
To show that the Einstein equation to third order is consistent with the gauge normalisation we will use the perturbed 
Kodaira-Spencer bracket (see \eqref{TT-defn} for the definition) as main tool.
\begin{teo}  \label{main-LL}
Assume that the Taylor coefficients $h_1,h_2, h_3$ of the Einstein deformation $g_t \in \mathscr{M}_1$ with 
$g^{-1}g_t=\id+th_1+\tfrac{t^2}{2!}h_2+\tfrac{t^3}{3!}h_3+o(t^4)$are normalised under the gauge group $\mathbf{G}$ according to $\delta^gh_1=0$ as well as \eqref{gauge-1} and \eqref{norm-3}. Then
\begin{itemize}
\item[(i)]
the component of the third order Einstein equation on $\Gamma \left (\Sym^{2,-}TM \right )$ reads 
\begin{equation} \label{thm11-i}
\widetilde{\Delta}_E^{-} \left (h_3^{-}-\tfrac{3}{2}h_1^3+\tfrac{3}{8}\tr(h_1^2)h_1 \right )=-6 \delta^g[h_1,h_2-h_1^2]^c_{\TT}
\end{equation}
\item[(ii)] the Taylor coefficient to third order in the Maurer-Cartan operator $\bdel \mathrm{C}_t-
[\mathrm{C}_t, \mathrm{C}_t]^c$ is harmonic, in other words the tensor 
$$ \bdel \mathrm{C}_3-6[\mathrm{C}_1,\mathrm{C}_2]^c=-\tfrac{1}{2} \left (\bdel (h_3^{-}-\tfrac{3}{2}h_1^3)+3[h_1,h_2-h_1^2]^c \right ) \in \ker \delta^g \cap \ker \bdel
$$
\item[(iii)] equation \eqref{thm11-i} is unobstructed as an equation in $h_3^{-}$; explicitly 
we must have 
\begin{equation*}
h_3^{-}-\tfrac{3}{2}h_1^3+\tfrac{3}{8}\tr(h_1^2)h_1=\mathbf{q}_3+\delta^{\star_g,-}\grad \mathbf{f}_3=\mathbf{q}_3-\tfrac{1}{2}\L_{J\grad \mathbf{f}_3}
\end{equation*}
where the pair $(\mathbf{q}_3,\mathbf{f}_3)$ in $\TT^{-}(M,g) \oplus C^{\infty}(M,g)$ satisfies 
$$ \Delta_E\mathbf{q}_3=-6\delta^g[h_1,h_2-h_1^2]^c_{\TT} \ \mathrm{and} \ (\Delta-2E)\mathbf{f}_3=-\tfrac{3}{4}\tr \{h_1,h_2-h_1^2\}.
$$
\end{itemize}
\end{teo}
\begin{proof}
(i) Since $\tfrac{2}{3}h_3^{+}-\{h_1,h_2-h_1^2\}=0$  the component on $\Gamma \left (\Sym^{2,-}TM \right )$ of the third order Einstein equation reads 
\begin{equation} \label{3-eqnN}
\begin{split}
\la \widetilde{\Delta}_E^{-}\mathbf{h}_3,H \ra_{L^2}&=2\la \delta^g[h_1,h_2-h_1^2]^c,H \ra_{L^2}
\end{split}
\end{equation}
for all $H \in \Gamma \left (\Sym^{2,-}TM \right )$. 

To bring this to final form we taken into the definition of the perturbed Kodaira-Spencer bracket 
from \eqref{TT-defn} and expand 
$$[h_1,h_2-h_1^2]^c_{\TT}=[h_1,h_2-h_1^2]^c+\tfrac{1}{4} \left (\delta^g(h_2-h_1^2) \wedge h_1-\delta^gJ(h_2-h_1^2) \wedge Jh_1 \right )$$
by also using that $h_1$ and $Jh_1$ are divergence free. Because $\delta^g(h_2-h_1^2)=-\tfrac{1}{4}\di\!\tr(h_1^2)$ by \eqref{gauge-1} and also 
$\bdel h_1=0$ this leads to the comparison formula 
\begin{equation} \label{comp-TT1}
[h_1,h_2-h_1^2]^c_{\TT}=[h_1,h_2-h_1^2]^c-\tfrac{1}{8}\bdel (\tr(h_1^2)h_1).
\end{equation}
Thus the symmetric component in $\delta^g[h_1,h_2-h_1^2]^c$ is given by 
$$\la \delta^g[h_1,h_2-h_1^2]^c_{\TT}+\tfrac{1}{8}\delta^g \bdel (\tr(h_1^2)h_1),H 
\ra_{L^2}=\la \delta^g[h_1,h_2-h_1^2]^c_{\TT}+\tfrac{1}{16}
\widetilde{\Delta}_E^{-} (\tr(h_1^2)h_1), H \ra_{L^2}
$$ 
by using part (ii) in Proposition \ref{wz_K}. Since $\mathbf{h}_3=-\tfrac{1}{3}h_3^{-}+\tfrac{1}{2}h_1^3$ an algebraic manipulation shows that 
\begin{equation*}
\la \widetilde{\Delta}_E^{-} \left (h_3^{-}-\tfrac{3}{2}h_1^3+\tfrac{3}{8}\tr(h_1^2)h_1 \right ),H \ra_{L^2}=-6\la \delta^g[h_1,h_2-h_1^2]^c_{\TT}, H \ra_{L^2}.
\end{equation*}
Since $\delta^g[h_1,h_2-h_1^2]^c_{\TT}$ is a $\TT$-tensor by (i) in Proposition \ref{del-TT}, in particular it is symmetric, the claim is fully proved.\\
(ii) Recall that $\bdel h_1=0$ as well as $\bdel (h_2-h_1^2)=-[h_1,h_1]^c$. From the general properties of the Kodaira-Spencer bracket (see e.g. \cite{Huy}) we have $\bdel [h_1,h_2-h_1^2]^c=-[h_1,[h_1,h_1]^c]^c=0$. 
Thus $\bdel [h_1,h_2-h_1^2]^c_{\TT}=0$ as well, as entailed by \eqref{comp-TT1}. Next we recall that 
$$ \delta^g\left (h_3^{-}-\tfrac{3}{2}h_1^3+\tfrac{3}{8}\tr(h_1^2)h_1 \right )=-\tfrac{3}{8}\di\!\tr \{h_1,h_2-h_1^2\}
$$
as entailed by \eqref{norm-3}; in particular $ \di^{-} \delta^g\left (h_3^{-}-\tfrac{3}{2}h_1^3+\tfrac{3}{8}\tr(h_1^2)h_1 \right )=0.$ Thus by using successively \eqref{comp-TT1}, the Weitzenb\"ock formula in Proposition 
\ref{wz_K}, (i) and also part (i) we find that 
\begin{equation*}
\begin{split}
&\delta^g\left (\bdel (h_3^{-}-\tfrac{3}{2}h_1^3)+3[h_1,h_2-h_1^2]^c \right ) \\
=& \delta^g \bdel \left (h_3^{-}-\tfrac{3}{2}h_1^3+\tfrac{3}{8}\tr(h_1^2)h_1 \right )+3[h_1,h_2-h_1^2]^c_{\TT}=0.
\end{split}
\end{equation*}
The claim on the Maurer-Cartan operator follows from the expression for the Taylor coefficients of the Cayley transform in \eqref{cay-co}.\\
(iii) Recall the general formula $\Delta_E \delta^{\star_g,-}X=\delta^{\star_g,-}(\Delta-2E)X$ for all $X \in \Gamma (TM)$. Thus letting $\mathbf{f}_3$ solve the equation $(\Delta-2E)\mathbf{f}_3=-\tfrac{3}{4}\tr \{h_1,h_2-h_1^2\}$(which is always possible since $E<0$) we use \eqref{norm-3} to see that equation \eqref{thm11-i} entails 
$ \Delta_E \mathbf{q}_3=-6\delta^g[h_1,h_2-h_1^2]^c_{\TT}
$ where 
$\mathbf{q}_3:=h_3^{-}-\tfrac{3}{2}h_1^3+\tfrac{3}{8}\tr(h_1^2)h_1-\delta^{\star_g,-}\grad \mathbf{f}_3 $. The claim follows from the fact that $\delta^g[h_1,h_2-h_1^2]^c_{\TT}$ belongs to $\TT^{-}(M,g)$ combined with 
the fact that the Einstein operator preserves the space $\TT^{-}(M,g)$.
\end{proof}


\begin{thebibliography}{99}
\bibitem{Al-Se}
B.Alexandrov, Uwe Semmelmann, 
\textit{Deformations of nearly parallel $\mathrm{G}_2$-structures}, 
Asian J. Math. {\bf{16}} (2012), 713--744. 

\bibitem{Ba}
Wafaâ Batat, Stuart James Hall, Thomas Murphy, James Waldron, \textit{Rigidity of $SU_n$-type symmetric spaces}, 
Int. Math. Res. Not. IMRN 2024, no. 3, 2066--2098.

\bibitem{Besse}
A.Besse, \textit{Einstein manifolds}, Classics in Mathematics, Springer-Verlag, Berlin, 2008.

\bibitem{bourg}
Jean-Pierre  Bourguignon, 
 \textit{Les varietes de dimension 4 a signature non nulle dont la courbure est harmonique sont d'Einstein},
Invent. Math. {\bf 63} (1981), no. 2, 263--286. 

\bibitem{Dai}
Xianzhe Dai, Xiaodong Wang, Guofang Wei, \textit{
On the variational stability of K\"ahler-Einstein metrics}, 
Comm. Anal. Geom. {\bf{15}} (2007), no. 4, 669--693.

\bibitem{deBMa}
P.de Bartolomeis, V.Matveev, \textit{Some remarks on Nijenhuis bracket, formality, and K\"ahler
manifolds}, Adv. Geom. {\bf{13}} (2013), no. 4, 571--581.

\bibitem{HaSe}
Stuart James Hall, Paul Schwahn, Uwe Semmelmann, \textit{On the rigidity of the complex Grassmannians}, 
Trans. Amer. Math. Soc. {\bf{378}} (2025), no. 6, 4335--4367.

\bibitem{AU}
Konstantin Heil, Andrei Moroianu, Uwe Semmelmann,
\textit{Killing and conformal Killing tensors},
J. Geom. Phys. {\bf 106} (2016), 383--400.

\bibitem{Hori}
E. Horikawa, \textit{Algebraic surfaces of general type with small $c_1^2$,I}, Ann. Math. {\bf{104}} (1976) (2), 357--387.

\bibitem{Huy}
Daniel Huybrechts, \textit{Complex Geometry. An introduction}, Universitext, Springer-Verlag, Berlin, 2005.

\bibitem{Koiso}
Norihito Koiso, 
\textit{Rigidity and deformability of Einstein metrics},
Osaka Math. J. {\bf{19}} (1982), 643--668.

\bibitem{Koiso-I}
Norihito Koiso, \textit{Einstein metrics and complex structures}, 
Invent. Math. {\bf{73}} (1983), no. 1, 71--106.

\bibitem{Kollar}
Ivan Kol\'ar, Peter Michor, Jan Slov\'ak, \textit{Natural operations in differential geometry}, Springer-Verlag, Berlin, 1993.

\bibitem{LeB-1}
C. LeBrun, \textit{Einstein metrics and Mostow rigidity}, Math. Res. Lett. {\bf{2}} (1995), 1--8. 

\bibitem{LeB-2}
C. LeBrun, \textit{Einstein metrics and the Yamabe problem}, in: Trends in Mathematical Physics,
Studies in Adv. Math. 13, AMS/IP (1999).

\bibitem{MNS}
Andrei Moroianu, Paul-Andi Nagy, Uwe Semmelmann, \textit{Deformations of nearly-K\"ahler structures}, Pacific Journal of Mathematics {\bf{235}} (2008), no.1, 57--72.

\bibitem{NSE_LMS}
Paul-Andi Nagy, Uwe Semmelmann, \textit{Deformations of nearly $\mathrm{G}_2$-structures}, J.London Math.Soc. (2) {\bf{104}} (2021), 1795--1811.

\bibitem{NS-G2}
Paul-Andi Nagy, Uwe Semmelmann, \textit{The $\mathrm{G}_2$ geometry of 3-Sasaki structures},
J. Geom. Anal. {\bf{34}} (2024), no. 2, Paper No. 61, 53 pp.

\bibitem{NS-E}
Paul-Andi Nagy, Uwe Semmelmann, \textit{Second order Einstein deformations}, 
J. Math. Soc. Japan {\bf{77}} (2025), no. 2, 345--389.

\bibitem{N1}
Paul-Andi Nagy, \textit{Einstein deformations of K\"ahler-Einstein metrics}, \url{https://arxiv.org/abs/2603.09028}.

\bibitem{T1}
Gang Tian, \textit{
Smoothness of the universal deformation space of compact Calabi-Yau manifolds and its Petersson-Weil metric},  Mathematical aspects of string theory (San Diego, Calif., 1986), 629--646,
Adv. Ser. Math. Phys., 1, World Sci. Publishing, Singapore, 1987.

\bibitem{T2}
Andrey N. Todorov, \textit{
The Weil-Petersson geometry of the moduli space of $SU(n \geq 3)$ (Calabi-Yau) manifolds. I.}, 
Comm. Math. Phys. 126 (1989), no. 2, 325--346.

\bibitem{S-Sw}
Paul Schwahn, Uwe Semmelmann, \textit{Einstein metrics, their moduli spaces and stability}, \url{https://arxiv.org/pdf/2507.18463}.
\end{thebibliography}
\end{document}